\def\todaysdate{14\textsuperscript{th} March 2025}
\documentclass[reqno,a4paper]{article}
\usepackage{etex}
\usepackage[T1]{fontenc}
\usepackage[utf8]{inputenc}
\usepackage{lmodern}

\usepackage[style=alphabetic,giveninits=true,backref,backend=biber,isbn=false,url=false,maxbibnames=99,maxalphanames=4]{biblatex}

\renewbibmacro{in:}{}
\newbibmacro{string+doi}[1]{\iffieldundef{doi}{\iffieldundef{url}{#1}{\href{\thefield{url}}{#1}}}{\href{http://dx.doi.org/\thefield{doi}}{#1}}}
\DeclareFieldFormat*{title}{\usebibmacro{string+doi}{\emph{#1}}}
\DefineBibliographyStrings{english}{%
  backrefpage = {$\uparrow$},
  backrefpages = {$\uparrow$},
}

\usepackage{amssymb,amsmath,amsthm}
\usepackage{thmtools,xcolor}
\usepackage{units}
\usepackage{graphicx}
\usepackage[all]{xy}
\usepackage{tikz}
\usetikzlibrary{arrows,calc,positioning,decorations.pathreplacing}
\usepackage{tikz-cd}
\usepackage{relsize}
\usepackage{mhsetup}
\usepackage{mathtools}
\usepackage{stmaryrd}
\usepackage{tocloft}
\usepackage{paralist} 
\usepackage[left=3cm,right=3cm,top=3cm,bottom=3cm]{geometry} 
\usepackage[bottom]{footmisc}
\usepackage{multirow}
\usepackage[normalem]{ulem}
\usepackage[colorlinks,citecolor=blue,linkcolor=blue,urlcolor=blue,filecolor=blue,breaklinks]{hyperref}
\usepackage{footnotebackref}
\usepackage[format=plain,indention=1em,labelsep=endash,labelfont={sc},textfont={small},margin=2em,singlelinecheck=false]{caption}
\usepackage{subcaption}
\usepackage[mathscr]{euscript}
\usepackage{accents}
\usepackage{xparse}
\usepackage[section,below]{placeins}

\definecolor{lightblue}{rgb}{0.8,0.8,1}

\declaretheoremstyle[
  spaceabove=\topsep,
  spacebelow=\topsep,
  headpunct=,
  numbered=no,
  postheadspace=1ex,
  headfont=\normalfont\bfseries,
  bodyfont=\normalfont\itshape,
]{italic}
\declaretheoremstyle[
  spaceabove=\topsep,
  spacebelow=\topsep,
  headpunct=,
  numbered=no,
  postheadspace=1ex,
  headfont=\normalfont\bfseries,
  bodyfont=\normalfont\upshape,
]{upright}
\declaretheorem[style=italic,name=Theorem,numbered=yes,numberwithin=section]{thm}
\declaretheorem[style=italic,name=Lemma,numbered=yes,numberlike=thm]{lem}
\declaretheorem[style=italic,name=Proposition,numbered=yes,numberlike=thm]{prop}
\declaretheorem[style=italic,name=Corollary,numbered=yes,numberlike=thm]{coro}
\declaretheorem[style=italic,name=Fact,numbered=yes,numberlike=thm]{fact}
\declaretheorem[style=italic,name=Theorem,numbered=yes,numberwithin=section]{athm}
\declaretheorem[style=italic,name=Proposition,numbered=yes,numberlike=athm]{aprop}

\declaretheorem[style=upright,name=Definition,numbered=yes,numberlike=thm]{defn}
\declaretheorem[style=upright,name=Remark,numbered=yes,numberlike=thm]{rmk}
\declaretheorem[style=upright,name=Outlook,numbered=yes,numberlike=thm]{outlook}
\declaretheorem[style=upright,name=Warning,numbered=yes,numberlike=thm]{wng}
\declaretheorem[style=upright,name=Example,numbered=yes,numberlike=thm]{eg}
\declaretheorem[style=upright,name=Notation,numbered=yes,numberlike=thm]{notation}
\declaretheorem[style=upright,name=Convention,numbered=yes,numberlike=thm]{convention}

\makeatletter
\renewcommand*{\@seccntformat}[1]{\upshape\csname the#1\endcsname.\hspace{1ex}}
\renewcommand*{\section}{\@startsection{section}{1}{\z@}%
	{2.5ex \@plus 1ex \@minus 0.2ex}%
	{1.5ex \@plus 0.2ex}%
	{\normalfont\Large\bfseries}}
\renewcommand*{\subsection}{\@startsection{subsection}{2}{\z@}%
	{2.5ex \@plus 1ex \@minus 0.2ex}%
	{1.5ex \@plus 0.2ex}%
	{\normalfont\large\bfseries}}
\renewcommand*{\subsubsection}{\@startsection{subsubsection}{3}{\z@}%
	{2.5ex \@plus 1ex \@minus 0.2ex}%
	{1.5ex \@plus 0.2ex}%
	{\normalfont\normalsize\bfseries}}
\newcommand*{\subsubsubsection}{\@startsection{paragraph}{4}{\z@}%
	{2.5ex \@plus 1ex \@minus 0.2ex}%
	{1.5ex \@plus 0.2ex}%
	{\normalfont\normalsize\bfseries}}
\newcommand*{\subsubsubsubsection}{\@startsection{paragraph}{5}{\z@}%
	{2.5ex \@plus 1ex \@minus 0.2ex}%
	{1.5ex \@plus 0.2ex}%
	{\normalfont\normalsize\bfseries}}
\makeatother

\newcommand{\Mod}{\ensuremath{\mathrm{Mod}}}
\newcommand{\LB}{\mathbf{LB}}
\newcommand{\LCS}{\varGamma}
\newcommand{\GL}{\mathrm{GL}}
\newcommand{\BM}{\mathrm{BM}}
\newcommand{\B}{\mathbf{B}}
\newcommand{\wB}{\mathbf{wB}}
\newcommand{\C}{\mathbf{C}}
\newcommand{\exwB}{\widetilde{\mathbf{w}}\mathbf{B}}
\newcommand{\RB}{\mathbf{RB}}
\newcommand{\R}{\mathbf{R}}
\newcommand{\F}{\mathbf{F}}
\newcommand{\bk}{\mathbf{k}}

\newcommand{\id}{\mathrm{id}}
\newcommand{\compact}{\mathrm{cpt}}
\newcommand{\ab}{\mathrm{ab}}
\newcommand{\longtwoheadrightarrow}{%
  \longrightarrow\mathrel{\mkern-14mu}\rightarrow
}

\newcommand{\incl}[3][right]%
{%
\draw[<-,>=#1 hook] #2 to ($ #2!0.5!#3 $);
\draw[->] ($ #2!0.5!#3 $) to #3;%
}
\newcommand{\inclusion}[5][right]%
{%
\draw[<-,>=#1 hook] #4 to ($ #4!0.5!#5 $) node[#2,font=\small]{#3};
\draw[->] ($ #4!0.5!#5 $) to #5;%
}

\newenvironment{itemizeb}%
{\begin{compactitem}

}%
{\end{compactitem}}

\newcommand{\cL}{\mathcal{L}}

\newcommand{\bC}{\mathbb{C}}
\newcommand{\bD}{\mathbb{D}}

\newcommand{\bN}{\mathbb{N}}

\newcommand{\bR}{\mathbb{R}}

\newcommand{\bZ}{\mathbb{Z}}

\renewcommand{\geq}{\geqslant}
\renewcommand{\leq}{\leqslant}
\renewcommand{\footnoterule}{%
  \kern -3pt
  \hrule width \textwidth height 0.4pt
  \kern 2.6pt
}

\definecolor{dred}{rgb}{0.7,0,0}

\definecolor{dblue}{rgb}{0,0,0.8}

\definecolor{dgreen}{rgb}{0,0.5,0}

\definecolor{Arthur}{rgb}{0.3,0,1}

\definecolor{Martin}{rgb}{0,0.6,0.1}

\begin{document}
\title{\Large\bfseries The pro-nilpotent Lawrence-Krammer-Bigelow representation}
\author{\normalsize Martin Palmer\footnote{\textit{Institutul de Matematică Simion Stoilow al Academiei Române, 21 Calea Griviței, 010702 București, Romania}, email: \texttt{mpanghel@imar.ro}; \textit{School of Mathematics, University of Leeds, Leeds, LS2 9JT, UK}, email: \texttt{m.d.palmer-anghel@leeds.ac.uk}}{~} and Arthur Souli{\'e}\footnote{\textit{Normandie Univ., UNICAEN, CNRS, LMNO, 14000 Caen, France}, email: \texttt{artsou@hotmail.fr}, \texttt{arthur.soulie@unicaen.fr}}}
\date{\normalsize\todaysdate}
\maketitle

\begin{abstract}
We construct a $3$-variable enrichment of the Lawrence-Krammer-Bigelow (LKB) representation of the braid groups, which is the limit of a pro-nilpotent tower of representations having the original LKB representation as its bottom layer. We also construct analogous pro-nilpotent towers of representations of surface braid groups and loop braid groups.

\


\end{abstract}

\tableofcontents

\section*{Introduction}
\addcontentsline{toc}{section}{Introduction}

The Lawrence-Krammer-Bigelow representation was first introduced by Lawrence~\cite{Lawrence1} as part of a more general family of representations of the braid group $\B_{n}$. It was shown, independently by Bigelow~\cite{bigelow2001braid} and by Krammer~\cite{KrammerLK}, to be faithful, thus proving that $\B_{n}$ is linear; in other words, the braid group $\B_{n}$ acts faithfully on a finite-dimensional vector space.

This representation is constructed via the action of $\B_{n}$, which is also the mapping class group of the $n$-times punctured $2$-disc $\bD_{n}$, on the configuration space $C_{2}(\bD_{n})$ of two unordered, distinct points in $\bD_{n}$. There is a certain local system $\cL_{2}$ on $C_{2}(\bD_{n})$, defined over the ground ring $\bZ[q^{\pm 1},t^{\pm 1}] = \bZ[\bZ^2]$, which is preseved by this action of $\B_{n}$. The induced action on the second Borel-Moore homology group
\begin{equation}
\label{eq:LKB}
V(2) \coloneqq H_{2}^{\BM}(C_{2}(\bD_{n});\cL_{2})
\end{equation}
is the \emph{Lawrence-Krammer-Bigelow representation}. More generally, considering configurations of $k\geq 2$ points in $\bD_{n}$, one obtains the family of \emph{Lawrence representations}
\begin{equation}
\label{eq:Lawrence}
V(k) \coloneqq H_{k}^{\BM}(C_{k}(\bD_{n});\cL_{k})
\end{equation}
of the braid group $\B_{n}$. The $\bZ[\bZ^2]$-module $V(k)$ is free and has rank $\binom{n+k-2}{k}$ \cite[Lem.~3.1]{BigelowHomrep}.

One main significance of the Lawrence-Krammer-Bigelow (LKB) representation is its use in the proof of the linearity of the braid groups. Another significance of the whole family of Lawrence representations is their deep connection to the Jones polynomial \cite{Lawrence1993,BigelowJones}, the $\mathfrak{sl}_N$ polynomials \cite{Lawrence1996,BigelowHOMFLY} and the coloured Jones polynomials \cite{AnghelTopologicalModel}.

\paragraph*{The pro-nilpotent LKB representation.}
Our first main construction upgrades the representation \eqref{eq:LKB} of $\B_{n}$ to a \emph{pro-nilpotent representation}: a compatible family of representations over the group rings $\bZ[Q_{r}]$ for each $Q_{r}$ in a tower of groups $Q_{\bullet}$, where the nilpotency class of $Q_{r}$ is $r-1$. See \S\ref{s:pro-nilpotent} for the precise definitions.

\begin{athm}\label{athm:pro-nilpotent}
There is a nilpotent tower of groups $Q_{\bullet}$ with $Q_{2} = \bZ^2$ and a pro-nilpotent representation of $\B_{n}$ over $Q_{\bullet}$ whose bottom layer is equal to \eqref{eq:LKB}.
\end{athm}

\paragraph*{Ribbon Lawrence representations.}
In fact, the pro-nilpotent representation of Theorem~\ref{athm:pro-nilpotent} is induced, for $n\geq 3$, by a representation of $\B_{n}$ over $\bZ[Q_{\infty}]$ for a certain subgroup $Q_{\infty}$ of the pro-nilpotent group $\lim(Q_{\bullet})$. This subgroup $Q_{\infty}$ is isomorphic to the semi-direct product $\bZ^2 \rtimes_\varphi \bZ$ where $\varphi(1)$ is the automorphism of $\bZ^2$ that swaps the two summands, which in turn is isomorphic to the ribbon braid group on two strands $\RB_{2}$. This is a special case of a more general construction involving the level $k$ Lawrence representation \eqref{eq:Lawrence}. In the following theorem we denote the ribbon braid group on $k$ strands by $\RB_{k} \cong \bZ^k \rtimes \B_{k}$.

\begin{athm}\label{athm:ribbon-Lawrence}
There is a well-defined representation
\begin{equation}
\label{eq:ribbon-Lawrence}
\mathscr{RL}_{k} \colon \B_{n} \longrightarrow \mathrm{Aut}_{\bZ[\RB_{k}]}(V_{\R}(k))
\end{equation}
that recovers \eqref{eq:Lawrence} after reducing along the abelianisation $\RB_{k} \twoheadrightarrow (\RB_{k})^{\ab} \cong \bZ^2$. For $k=2$ and $n\geq 3$, this representation induces the pro-nilpotent tower of representations of Theorem~\ref{athm:pro-nilpotent}.
\end{athm}

In the above statement, the \emph{reduction} of $V_{\R}(k)$ along the abelianisation $\RB_{k} \twoheadrightarrow \bZ^2$ means the tensor product $V_{\R}(k) \otimes_{\bZ[\RB_{k}]} \bZ[\bZ^2]$, viewing $\bZ[\bZ^2]$ as a $\bZ[\RB_{k}]$-module via the map $\RB_{k} \twoheadrightarrow \bZ^2$.

In particular, for $k=2$, Theorem~\ref{athm:ribbon-Lawrence} upgrades the LKB representation, defined over $\bZ[q^{\pm 1},t^{\pm 1}]$, to a representation defined over the non-commutative three-variable Laurent polynomial ring
\begin{equation}
\label{eq:Lambda}
\Theta = \bZ[\bZ^2 \rtimes \bZ] = \bZ\langle q_{1}^{\pm 1},q_{2}^{\pm 1},t^{\pm 1} \rangle / (q_{1} q_{2} = q_{2} q_{1} , q_{1} t = t q_{2} , q_{2} t = t q_{1})
\end{equation}
that recovers the original LKB representation when setting $q_{1} = q_{2}$. We explicitly compute the matrices for this representation in a natural basis:
\begin{athm}\label{athm:3-variable-LKB}
The representation $\mathscr{RL}_{2}$ of $\B_{n}$ over the ring $\Theta$ is given concretely by assigning to the generators $\sigma_{i}$ of $\B_{n}$ the matrices depicted in Table~\ref{tab:LKB3} on page \pageref{tab:LKB3}.
\end{athm}

The representation $\mathscr{RL}_{2}$ and each layer of the pro-nilpotent LKB representation of Theorem~\ref{athm:pro-nilpotent} are faithful, as a consequence of the faithfulness of the LKB representation (see Proposition~\ref{prop:faithfulness}).

We also note that the representation $\mathscr{RL}_{2}$ of $\B_{n}$ over $\Theta$ may be embedded into a complex representation (of twice the dimension) using an embedding of $\Theta$ into the matrix ring $\mathrm{Mat}_{2}(\bC)$; this makes certain properties of the representation, such as questions of irreducibility, more tractable. In more detail, $V_\R(2)$ is free as a module over $\Theta$ and an embedding of rings $\Theta \hookrightarrow \mathrm{Mat}_{2}(\bC)$ induces an embedding of groups $\GL_m(\Theta) \hookrightarrow \GL_{2m}(\bC)$ for any $m$. Such an embedding is given for example as follows.

\begin{aprop}\label{aprop:embedding-into-matrix-ring}
In the notation of \eqref{eq:Lambda}, an embedding $\Theta \hookrightarrow \mathrm{Mat}_{2}(\bC)$ is given by the assignments
\begin{equation}
\label{eq:embedding-into-matrix-ring}
q_{1} \longmapsto \left(\begin{matrix} w & 0 \\ 0 & x \end{matrix}\right)
\qquad
q_{2} \longmapsto \left(\begin{matrix} x & 0 \\ 0 & w \end{matrix}\right)
\qquad
t \longmapsto \left(\begin{matrix} 0 & y \\ z & 0 \end{matrix}\right)
\end{equation}
for any choice of $w,x,y,z \in \bC$ such that $\{ w , x , yz \} \subset \bC$ is an algebraically independent subset.
\end{aprop}

Explicit formulas for this embedding of the representation $\mathscr{RL}_{2}$ into a representation over $\bC$ are obtained by applying the assignments \eqref{eq:embedding-into-matrix-ring} to each entry of the matrices in Table~\ref{tab:LKB3} (page \pageref{tab:LKB3}).

\paragraph*{More pro-nilpotent representations.}
The construction of the pro-nilpotent representation of Theorem~\ref{athm:pro-nilpotent} follows a general recipe that we describe in \S\ref{s:general-recipe}. The key ingredient in this recipe, if one wishes to construct a pro-nilpotent representation of a group $\Gamma$, is a surjection of groups
\begin{equation}
\label{eq:surjection-of-groups}
G \relbar\joinrel\twoheadrightarrow \Gamma .
\end{equation}
Whether or not the recipe produces a pro-nilpotent representation on this input depends on the properties of the lower central series of $G$, and in fact in a more subtle way on the interaction between the lower central series of $G$ and the lower central series of $\mathrm{ker}(G \twoheadrightarrow \Gamma)$. For the construction of Theorem~\ref{athm:pro-nilpotent}, the surjection \eqref{eq:surjection-of-groups} is $\B_{2,n} \twoheadrightarrow \B_{n}$, where $\B_{2,n}$ is the partitioned braid group with $n+2$ strands partitioned into two blocks of sizes $n$ and $2$ respectively, and the surjection forgets the two strands of the second block.

Our construction applies to many other settings of this form, providing a wide variety of pro-nilpotent representations of surface braid groups $\B_{n}(S)$ and of (extended) welded braid groups $\wB_{n}$
and $\exwB_{n}$ (see \S\ref{s:LBn} for background on these).

\begin{athm}\label{athm:other-pro-nilpotent-reps}
There are pro-nilpotent representations of the groups $\Gamma$, induced by the surjections \eqref{eq:surjection-of-groups}, for each of the pairs listed in Table~\ref{table:pro-nilpotent-representations} on page \pageref{table:pro-nilpotent-representations}.
\end{athm}

\begin{table}[p]
    \centering
    \begin{tikzpicture}
    [x=1.8mm,y=0.6mm]
    \begin{scope}
        \draw[black!80] (0,0) rectangle (60,-60);
        \draw[blue!80] (11,-11) rectangle (39,-39);
        \draw[red!80] (31,-31) rectangle (59,-59);
        \node at (5,5) {$101$};
        \node at (15,5) {$200$};
        \node at (25,5) {$110$};
        \node at (35,5) {$020$};
        \node at (45,5) {$011$};
        \node at (55,5) {$002$};
        \node[anchor=east] at (-1,-5) {$101$};
        \node[anchor=east] at (-1,-15) {$200$};
        \node[anchor=east] at (-1,-25) {$110$};
        \node[anchor=east] at (-1,-35) {$020$};
        \node[anchor=east] at (-1,-45) {$011$};
        \node[anchor=east] at (-1,-55) {$002$};
        \node at (5,-5) {$1$};
        \node at (5,-15) {$0$};
        \node at (5,-25) {$q_{2}$};
        \node at (5,-35) {$(1-t)q_{2}$};
        \node at (5,-45) {$1$};
        \node at (5,-55) {$0$};
        \node at (15,-5) {$0$};
        \node at (15,-15) {$1$};
        \node at (15,-25) {$1$};
        \node at (15,-35) {$1$};
        \node at (15,-45) {$0$};
        \node at (15,-55) {$0$};
        \node at (25,-5) {$0$};
        \node at (25,-15) {$0$};
        \node at (25,-25) {$-q_{2}$};
        \node at (25,-35) {$(t-1)q_{2}$};
        \node at (25,-45) {$0$};
        \node at (25,-55) {$0$};
        \node at (35,-5) {$0$};
        \node at (35,-15) {$0$};
        \node at (35,-25) {$0$};
        \node at (35,-35) {$-tq_{1}q_{2}$};
        \node at (35,-45) {$0$};
        \node at (35,-55) {$0$};
        \node at (45,-5) {$0$};
        \node at (45,-15) {$0$};
        \node at (45,-25) {$0$};
        \node at (45,-35) {$(t-1)q_{1}q_{2}$};
        \node at (45,-45) {$-q_{1}$};
        \node at (45,-55) {$0$};
        \node at (55,-5) {$0$};
        \node at (55,-15) {$0$};
        \node at (55,-25) {$0$};
        \node at (55,-35) {$q_{1}q_{2}$};
        \node at (55,-45) {$q_{1}$};
        \node at (55,-55) {$1$};
        \node[anchor=north west,align=justify,text width=104mm] at (0,-65) {\small The $6 \times 6$ submatrix for the basis elements of the form $\cdots xyz\cdots $, with the $y$ in the $i$-th position. Here, and below, $\cdots$ indicates a string of $0$s. If $i=1$, the $x$ is omitted and we consider the red (bottom right) submatrix. If $i=n-1$, the $z$ is omitted and we consider the blue (top left) submatrix.};
    \end{scope}
    \begin{scope}[yshift=-65mm,xshift=-8mm]
        \draw[black!80] (0,0) rectangle (30,-30);
        \draw[blue!80] (1,-1) rectangle (19,-19);
        \node at (5,5) {$100$};
        \node at (15,5) {$010$};
        \node at (25,5) {$001$};
        \node[anchor=east] at (-1,-5) {$100$};
        \node[anchor=east] at (-1,-15) {$010$};
        \node[anchor=east] at (-1,-25) {$001$};
        \node at (5,-5) {$1$};
        \node at (5,-15) {$1$};
        \node at (5,-25) {$0$};
        \node at (15,-5) {$0$};
        \node at (15,-15) {$-q_{1}$};
        \node at (15,-25) {$0$};
        \node at (25,-5) {$0$};
        \node at (25,-15) {$q_{1}$};
        \node at (25,-25) {$1$};
        \node[anchor=north west,align=justify,text width=55mm] at (-1,-35) {\small The $3 \times 3$ submatrices for the basis elements of the form $\cdots 1\cdots xyz\cdots$, with the $y$ in the $i$-th position. If $i=n-1$ then the $z$ is omitted and we consider the blue (top left) submatrix. \\ There are $i-2$ copies of this submatrix (except when $i=1$, when there are none), corresponding to the possible choices in $\{1,\ldots,i-2\}$ for the position of the additional `$1$'.};
    \end{scope}
    \begin{scope}[yshift=-65mm,xshift=62mm]
        \draw[black!80] (0,0) rectangle (30,-30);
        \draw[red!80] (11,-11) rectangle (29,-29);
        \node at (5,5) {$100$};
        \node at (15,5) {$010$};
        \node at (25,5) {$001$};
        \node[anchor=east] at (-1,-5) {$100$};
        \node[anchor=east] at (-1,-15) {$010$};
        \node[anchor=east] at (-1,-25) {$001$};
        \node at (5,-5) {$1$};
        \node at (5,-15) {$1$};
        \node at (5,-25) {$0$};
        \node at (15,-5) {$0$};
        \node at (15,-15) {$-q_{2}$};
        \node at (15,-25) {$0$};
        \node at (25,-5) {$0$};
        \node at (25,-15) {$q_{2}$};
        \node at (25,-25) {$1$};
        \node[anchor=north west,align=justify,text width=55mm] at (-1,-35) {\small The $3 \times 3$ submatrices for the basis elements of the form $\cdots xyz\cdots 1\cdots$, with the $y$ in the $i$-th position. If $i=1$ then the $x$ is omitted and we consider the red (bottom right) submatrix. \\ There are $n-i-2$ copies of this submatrix (except when $i=n-1$, when there are none), corresponding to the possible choices in $\{i+2,\ldots,n-1\}$ for the position of the additional `$1$'.};
        \node at (0,-110) {};
    \end{scope}
    \end{tikzpicture}
    \caption{The matrices of the three-variable LKB representation over the ring \[ \Theta = \bZ[\bZ^2 \rtimes \bZ] = \bZ\langle q_{1}^{\pm 1},q_{2}^{\pm 1},t^{\pm 1} \rangle / (q_{1} q_{2} = q_{2} q_{1} , q_{1} t = t q_{2} , q_{2} t = t q_{1}), \] indicating the action of the generator $\sigma_{i}$ of $\B_{n}$ for $1\leq i\leq n-1$. Recall that, as a $\Theta$-module, the representation is free of rank $\binom{n}{2}$, with basis given by ordered $(n-1)$-tuples of non-negative integers that sum to $2$. \\ \\ When $2\leq i\leq n-2$, the matrix consists of the $6 \times 6$ block described on the top row, together with $(i-2)+(n-i-2)=n-4$ copies of the $3 \times 3$ blocks described on the bottom row, together with the $\binom{n-3}{2} \times \binom{n-3}{2}$ identity matrix. In total, this is a square matrix of size $\binom{n}{2} = 6 + 3(n-4) + \binom{n-3}{2}$. \\ \\ When $i \in \{1,n-1\}$, the matrix consists of one of the $3 \times 3$ blocks described on the top row, together with $n-3$ copies of one of the $2 \times 2$ blocks described on the bottom row, together with the $\binom{n-2}{2} \times \binom{n-2}{2}$ identity matrix. In total, this is a square matrix of size $\binom{n}{2} = 3 + 2(n-3) + \binom{n-2}{2}$.}
    \label{tab:LKB3}
    \addcontentsline{toc}{subsection}{Table 1: Matrices for the three-variable LKB representation}
\end{table}

\bgroup
\def\arraystretch{1.5}
\begin{table}[p]
    \centering
    \begin{tabular}{|c|ll|c|}
        \cline{1-4}
        \multicolumn{4}{|c|}{\large The pro-nilpotent representations constructed in the present paper} \\
        \cline{1-4}
        $\Gamma$ & \multicolumn{2}{c|}{$G$} & $A$ \\
        \cline{1-4}
        \multirow{8}{*}{$\B_{n}$} & $\B_{2,n}$ && $\bZ^2$ \\
        & $\B_{2,\bk,n}$ & (each $k_{i}\geq 3$) & $\bZ^{\binom{l+2}{2}+l'+1}$ \\
        \cline{2-4}
        & $\B_{1,1,1,n}$ && $\bZ^6$ \\
        & $\B_{1,1,1,\bk,n}$ && $\bZ^{\binom{l+4}{2}+l'}$ \\
        \cline{2-4}
        & $\B_{2,2,n}$ && $\bZ^5$ \\
        & $\B_{2,2,\bk,n}$ && $\bZ^{\binom{l+3}{2}+l'+2}$ \\
        \cline{2-4}
        & $\B_{1,2,n}$ && $\bZ^4$ \\
        & $\B_{1,2,\bk,n}$ && $\bZ^{\binom{l+3}{2}+l'+1}$ \\
        \cline{1-4}
        \multirow{4}{*}{$\B_{n}(S)$} & $\B_{2,\bk,n}(S)$ & ($S \neq \bD^2$) & \multirow{4}{*}{Prop.~\ref{prop:abelian-quotients-2}} \\
        \cline{2-3}
        & $\B_{1,\bk,n}(S)$ & ($S \notin \{\bD^2 , \text{Ann} , \text{M{\"o}b}\}$) & \\
        \cline{2-3}
        & $\B_{1,\bk,n}(\text{M{\"o}b})$ & ($\bk \neq \varnothing$) & \\
        \cline{2-3}
        & $\B_{1,1,\bk,n}(\text{Ann})$ && \\
        \cline{1-4}
        \multirow{5}{*}{$\wB_{n}$\textsuperscript{$\dagger$}} & \multirow{5}{*}{$\wB(\lambda_P,\lambda_{S_+},\lambda_S)$} & \multicolumn{1}{|l|}{$(\lambda_P,\lambda_{S_+},\lambda_S) \succcurlyeq (\varnothing,\{n,b\},\varnothing)$} & \multirow{5}{*}{Prop.~\ref{prop:abelian-quotients-3}} \\
        \cline{3-3}
        && \multicolumn{1}{|l|}{$(\lambda_P,\lambda_{S_+},\lambda_S) \succcurlyeq (\varnothing,n,b)$} & \\
        \cline{3-3}
        && \multicolumn{1}{|l|}{$(\lambda_P,\lambda_{S_+},\lambda_S) \succcurlyeq (\varnothing,\{n,1,1\},\varnothing)$} & \\
        \cline{3-3}
        && \multicolumn{1}{|l|}{$(\lambda_P,\lambda_{S_+},\lambda_S) \succcurlyeq (2,n,\varnothing)$} & \\
        \cline{3-3}
        && \multicolumn{1}{|l|}{$(\lambda_P,\lambda_{S_+},\lambda_S) \succcurlyeq (\varnothing,n,1)$} & \\
        \cline{1-4}
        \multirow{5}{*}{$\exwB_{n}$\textsuperscript{$\dagger$}} & \multirow{5}{*}{$\wB(\lambda_P,\lambda_{S_+},\lambda_S)$} & \multicolumn{1}{|l|}{$(\lambda_P,\lambda_{S_+},\lambda_S) \succcurlyeq (\varnothing,b,n)$} & \multirow{5}{*}{Prop.~\ref{prop:abelian-quotients-3}} \\
        \cline{3-3}
        && \multicolumn{1}{|l|}{$(\lambda_P,\lambda_{S_+},\lambda_S) \succcurlyeq (\varnothing,\varnothing,\{n,b\})$} & \\
        \cline{3-3}
        && \multicolumn{1}{|l|}{$(\lambda_P,\lambda_{S_+},\lambda_S) \succcurlyeq (\varnothing,\{1,1\},n)$} & \\
        \cline{3-3}
        && \multicolumn{1}{|l|}{$(\lambda_P,\lambda_{S_+},\lambda_S) \succcurlyeq (2,i,n)$, $i\geq 1$} & \\
        \cline{3-3}
        && \multicolumn{1}{|l|}{$(\lambda_P,\lambda_{S_+},\lambda_S) \succcurlyeq (\varnothing,\varnothing,\{n,1\})$} & \\
        \cline{1-4}
    \end{tabular}
    \caption{The pro-nilpotent representations constructed in the present paper. \\ \\  \textit{Columns}: The first column lists the group $\Gamma$ being represented; the second column lists the auxiliary group $G$ whose (evident) surjection onto $\Gamma$ induces the pro-nilpotent representation. The third column lists the abelian group $A$ whose group ring $\bZ[A]$ is the ground ring for the bottom ($r=2$) layer of the pro-nilpotent representation. For the first two rows, we also have an explicit description of the ground ring of the limit of the tower of representations, as described in Proposition~\ref{prop:residually-nilpotent-quotient}. \\ \\ \textit{Further notation}: In several rows we have fixed an $l$-tuple $\bk = (k_{1},\ldots,k_{l})$ of positive integers and $l' \leq l$ denotes the number of $1\leq i\leq l$ such that $k_{i} \geq 2$. The surface $S$ is assumed to be non-closed, but it may have infinite type. The letter $b$ denotes either one of $\{2,3\}$. The notation $\wB(\lambda_P,\lambda_{S_+},\lambda_S)$ is explained in \S\ref{s:LBn} and the relation $(x_{1},y_{1},z_{1}) \succcurlyeq (x_{2},y_{2},z_{2})$ between triples of partitions means that $x_{2}$ is a sub-partition of $x_{1}$, $y_{2}$ is a sub-partition of $y_{1}$ and $z_{2}$ is a sub-partition of $z_{1}$. \\ \\ \textsuperscript{$\dagger$} For $\wB_{n}$ and $\exwB_{n}$ we construct \emph{weakly} pro-nilpotent representations; see Outlook~\ref{outlook-for-examples} and \S\ref{s:LBn} for why and see \S\ref{s:pro-nilpotent} for the definitions.}
    \label{table:pro-nilpotent-representations}
    \addcontentsline{toc}{subsection}{Table 2: The pro-nilpotent representations constructed in the present paper}
\end{table}
\egroup

\paragraph*{Outline.} The paper is organised as follows. Precise definitions of pro-nilpotent representations are given first in \S\ref{s:pro-nilpotent}. The key technical input for the existence of our pro-nilpotent representations is a refinement of the notion of the (non-)stopping of the lower central series of a group, which has been studied comprehensively for (partitioned) braid groups and their relatives in \cite{DPS}; this refinement is introduced in \S\ref{s:NCP}. In \S\ref{s:general-recipe} we give the general recipe for constructing pro-nilpotent representations of groups, which is a refinement of a recipe introduced by the authors in \cite{PSI}.
In \S\ref{ss:Bn-pro-nilpotent} we apply this in a special case to construct the pro-nilpotent LKB representation (Theorem~\ref{athm:pro-nilpotent}). The ribbon Lawrence representations (Theorem~\ref{athm:ribbon-Lawrence}) are then constructed in \S\ref{ss:Bn-ribbon}; in \S\ref{ss:Bn-matrices} we compute matrices (Theorem~\ref{athm:3-variable-LKB}) for the second ribbon Lawrence representation (which is also the limit of the pro-nilpotent LKB representation) and in \S\ref{ss:embedding-into-matrix-ring} we prove Proposition~\ref{aprop:embedding-into-matrix-ring}, which implies that the second ribbon Lawrence representation may be embedded into a representation over $\bC$. Finally, the construction of our other pro-nilpotent representations (Theorem~\ref{athm:other-pro-nilpotent-reps}) for surface braid groups and for loop braid groups is carried out in \S\ref{s:SBn} and \S\ref{s:LBn} respectively.

\paragraph*{Acknowledgements.}
We would like to thank Paolo Bellingeri and Emmanuel Wagner for interesting discussions related to the topic of this article.
We would like to thank an anonymous referee for the suggestion to look for embeddings of the ring $\Theta$ into a matrix ring over $\bC$, which impelled us to prove Proposition~\ref{aprop:embedding-into-matrix-ring}.
We would also like to thank a second anonymous referee for various helpful suggestions and corrections.
The first author was partially supported by a grant of the Romanian Ministry of Education and Research, CNCS - UEFISCDI, project number PN-III-P4-ID-PCE-2020-2798, within PNCDI III.
The second author was partially supported by the Institute for Basic Science IBS-R003-D1, by a Rankin-Sneddon Research Fellowship of the University of Glasgow and by the ANR Project AlMaRe ANR-19-CE40-0001-01.
In particular, the authors were able to make significant progress on the present article thanks to research visits to Glasgow and Bucharest, funded respectively by the School of Mathematics and Statistics of the University of Glasgow and the above-mentioned grant PN-III-P4-ID-PCE-2020-2798.
Finally, the two authors were also supported by a grant of the Ministry of Research, Innovation and Digitization, CNCS - UEFISCDI, project number PN-IV-P1-PCE-2023-2001, within PNCDI IV.

\section{Pro-nilpotent groups and representations}
\label{s:pro-nilpotent}

In general, a pro-object is a cofiltered diagram of objects in the appropriate category. In the case of nilpotent groups, we will consider just those of a certain form: inverse systems indexed by the natural numbers where each morphism is surjective and the groups have increasing nilpotency class:

\begin{defn}
\label{def:nilpotent-tower-of-groups}
A \emph{nilpotent tower of groups} $Q_{\bullet}$ is a sequence of surjective group homomorphisms
\[
\cdots \longtwoheadrightarrow Q_{r} \longtwoheadrightarrow Q_{r-1} \longtwoheadrightarrow \cdots \longtwoheadrightarrow Q_{2}
\]
where the nilpotency class of $Q_{r}$ is exactly $r-1$.
\end{defn}

\begin{eg}
\label{eg:inverse-system}
For example, if $G$ is a group that is not nilpotent and $\LCS_{r}(G)$ denotes the $r$-th term in its lower central series, in other words $\LCS_{1}(G) = G$ and $\LCS_{r}(G) = [\LCS_{r-1}(G),G]$ for $r\geq 2$, then the inverse system
\begin{equation}
\label{eq:inverse-system-LCQ}
\cdots \longtwoheadrightarrow G/\LCS_{r}(G) \longtwoheadrightarrow G/\LCS_{r-1}(G) \longtwoheadrightarrow \cdots \longtwoheadrightarrow G/\LCS_{2}(G) = G^{\ab}
\end{equation}
is a nilpotent tower of groups.
\end{eg}

Related to Example~\ref{eg:inverse-system}, we recall the following definition.

\begin{defn}
\label{d:pro-nilpotent-completion}
The \emph{pro-nilpotent completion} $\widehat{G}_{\mathrm{nil}}$ of a group $G$ is the inverse limit of \eqref{eq:inverse-system-LCQ}.
\end{defn}

Clearly, if $G$ is nilpotent, then $\widehat{G}_{\mathrm{nil}} \cong G$.

\begin{rmk}
\label{rmk:morphism-from-pronilpotent-completion}
The surjections $G \twoheadrightarrow G/\LCS_{r}(G)$ induce a morphism $G \to \widehat{G}_{\mathrm{nil}}$, which factors as
\begin{equation}
\label{eq:morphism-from-pronilpotent-completion}
G \relbar\joinrel\twoheadrightarrow G/\LCS_{\infty}(G) \lhook\joinrel\longrightarrow \widehat{G}_{\mathrm{nil}},
\end{equation}
where $\LCS_{\infty}(G)$ is the \emph{residue} of the group: $\LCS_{\infty}(G) = \bigcap_{i\geq 1} \LCS_{i}(G)$.
\end{rmk}

Any nilpotent tower of groups $Q_{\bullet}$ induces a sequence of functors
\begin{equation}
\label{eq:tower-of-module-categories}
\cdots \longrightarrow \Mod_{\bZ[Q_{r}]} \longrightarrow \Mod_{\bZ[Q_{r-1}]} \longrightarrow \cdots
\end{equation}
sending a $\bZ[Q_{r}]$-module $V$ to $V \otimes_{\bZ[Q_{r}]} \bZ[Q_{r-1}]$, where we view $\bZ[Q_{r-1}]$ as a $\bZ[Q_{r}]$-module via the ring homomorphism $\bZ[Q_{r}] \to \bZ[Q_{r-1}]$ induced by the given quotient $Q_{r} \twoheadrightarrow Q_{r-1}$.

\begin{defn}
\label{defn:pro-nilpotent-representation}
A \emph{pro-nilpotent representation} of a group $\Gamma$ is a choice of pro-nilpotent tower of groups $Q_{\bullet}$ together with a sequence of functors
\begin{equation}
\label{eq:sequence-of-functors}
\Gamma \longrightarrow \Mod_{\bZ[Q_{r}]}
\end{equation}
for $r\geq 2$ that commute up to natural isomorphism with \eqref{eq:tower-of-module-categories}. Concretely, this means that we have a $\Gamma$-representation $V_{r}$ over $\bZ[Q_{r}]$ for each $r\geq 2$ and isomorphisms $V_{r+1} \otimes_{\bZ[Q_{r+1}]} \bZ[Q_{r}] \cong V_{r}$ of $\Gamma$-representations over $\bZ[Q_{r}]$ for each $r\geq 2$.
\end{defn}

\begin{rmk}
The condition in Definition~\ref{def:nilpotent-tower-of-groups} that $Q_{r}$ must have nilpotency class equal to $r$ implies that the tower of representations involved in a pro-nilpotent representation must increase in complexity as $r\to\infty$.
\end{rmk}

There is also a weaker version of this notion, where we have only natural transformations rather than natural isomorphisms.

\begin{defn}
A \emph{weakly pro-nilpotent representation} of a group $\Gamma$ is a choice of pro-nilpotent tower of groups $Q_{\bullet}$ together with a $\Gamma$-representation $V_{r}$ over $\bZ[Q_{r}]$ for each $r\geq 2$ and homomorphisms
\[
V_{r+1} \otimes_{\bZ[Q_{r+1}]} \bZ[Q_{r}] \longrightarrow V_{r}
\]
of $\Gamma$-representations over $\bZ[Q_{r}]$ for each $r\geq 2$.
\end{defn}

\subsection{Pro-nilpotent representations vs.\ representations over inverse limits}

\begin{notation}
In general, we denote by $\lim_{r}(-)$ the inverse limit of an inverse system of objects indexed by $r$. For a nilpotent tower of groups $Q_{\bullet}$, we denote by $\widehat{Q}_{\bullet} := \lim_{r}(Q_r)$ its inverse limit. Thus in Example~\ref{eg:inverse-system} we have $\widehat{Q}_{\bullet} = \widehat{G}_{\mathrm{nil}}$.
\end{notation}

Any representation of $\Gamma$ over the group-ring $\bZ[\widehat{Q}_{\bullet}]$ induces a pro-nilpotent representation of $\Gamma$ over the pro-nilpotent tower of groups $Q_{\bullet}$. However, the converse does not hold. For simplicity, let us fix $k\geq 1$ and consider only representations that are free modules of rank $k$ over the ground ring. We are thus comparing homomorphisms
\[
\Gamma \longrightarrow \GL_{k}(\bZ[\widehat{Q}_{\bullet}]) = \GL_{k}(\bZ[\lim_{r}(Q_{r})])
\]
with compatible systems of homomorphisms $\Gamma \to \GL_{k}(\bZ[Q_{r}])$, in other words homomorphisms
\[
\Gamma \longrightarrow \lim_{r} (\GL_{k}(\bZ[Q_{r}])).
\]
The question is thus whether the canonical homomorphism
\begin{equation}
\label{eq:canonical-map}
\GL_{k}(\bZ[\lim_{r}(Q_{r})]) \longrightarrow \lim_{r} (\GL_{k}(\bZ[Q_{r}]))
\end{equation}
is an isomorphism of groups, i.e.~whether the endofunctor $\GL_{k}(\bZ[-])$ preserves limits (of this form). Whenever $k\geq 2$ it fails to be an isomorphism:

\begin{lem}
\label{lem:pronilpotent-counterexample}
Let $k\geq 2$ and $Q_{\bullet} = G/\LCS_{\bullet}(G)$ for $G$ equal to the Klein bottle group, i.e.\ the semi-direct product $\bZ \rtimes \bZ$ for the action $\bZ \curvearrowright \bZ$ where $1 \in \bZ$ acts by inversion.
Then \eqref{eq:canonical-map} is not surjective.
\end{lem}

Thus pro-nilpotent representations over a pro-nilpotent tower of groups $Q_{\bullet}$ do not all come from representations over the group-ring $\bZ[\widehat{Q}_{\bullet}]$ of the inverse limit of $Q_{\bullet}$.

\begin{proof}[Proof of Lemma~\ref{lem:pronilpotent-counterexample}]
Let us take $k=2$; the general case $k\geq 2$ will follow by an obvious immediate generalisation of the following argument. First note that $\LCS_{r}(G)$ is the subgroup $2^{r-1}\bZ$ of $\bZ \subseteq \bZ \rtimes \bZ$ for $r\geq 2$ (this may be computed directly, or one may apply \cite[Prop.~A.4]{DPS}), so that
\[
Q_{r} = G/\LCS_{r}(G) = \bZ/2^{r-1}\bZ \rtimes \bZ,
\]
where the generator $1 \in \bZ$ acts by inversion. The inverse limit $\widehat{Q}_{\bullet} = \widehat{G}_{\mathrm{nil}}$ is thus $\widehat{\bZ}_{2} \rtimes \bZ$, where $\widehat{\bZ}_{2}$ is the $2$-adic completion of $\bZ$. Recall that elements of $\widehat{\bZ}_{2}$ may be written as left-infinite strings $\cdots a_{3} a_{2} a_{1}$ where $a_{i} \in \{0,1\}$ and $n \in \bZ \subseteq \widehat{\bZ}_{2}$ corresponds to its expression in binary, continued to the left by an infinite string of $0$s. The quotient
\begin{equation}
\label{eq:truncation}
\widehat{\bZ}_{2} \relbar\joinrel\twoheadrightarrow \bZ/2^{r-1}\bZ
\end{equation}
truncates an infinite string to its rightmost $r-1$ digits. Continuing to write in binary, each quotient
\begin{equation}
\label{eq:remove-digit}
\bZ/2^{r-1}\bZ \relbar\joinrel\twoheadrightarrow \bZ/2^{r-2}\bZ
\end{equation}
removes the left-most digit from a string.

An element of the right-hand side of \eqref{eq:canonical-map} is a compatible sequence $M_{r}$ of invertible $2 \times 2$ matrices over the ring of polynomials in one variable $t$, whose exponents are elements of $\bZ/2^{r-1}\bZ \rtimes \bZ$, where \emph{compatible} means that $q_{r}(M_{r}) = M_{r-1}$, where $q_{r}$ applies $\eqref{eq:remove-digit} \rtimes \id_\bZ$ to each exponent of $t$. An element of the left-hand side of \eqref{eq:canonical-map} is an invertible $2 \times 2$ matrix $M$ over the ring of polynomials in one variable $t$, whose exponents are elements of $\widehat{\bZ}_{2} \rtimes \bZ$. The map \eqref{eq:canonical-map} is given by applying $\eqref{eq:truncation} \rtimes \id_\bZ$ to each exponent of $t$ in $M$.

Let $H$ be the subgroup of the right-hand side of \eqref{eq:canonical-map} consisting of sequences $M_{r}$, where each $M_{r}$ is of the form
\[
\left(\begin{matrix}
    1 & f_{r} \\
    0 & 1
\end{matrix}\right)
\]
for $f_{r} \in \bZ[\bZ/2^{r-1}\bZ \rtimes \bZ]$. Thus $H$ is naturally isomorphic to the underlying additive group of the ring $\lim_{r} \bZ[\bZ/2^{r-1}\bZ \rtimes \bZ]$. The pre-image $\eqref{eq:canonical-map}^{-1}(H)$ is the subgroup of matrices $M$ of the form
\[
\left(\begin{matrix}
    1 & f \\
    0 & 1
\end{matrix}\right)
\]
for $f \in \bZ[\widehat{\bZ}_{2} \rtimes \bZ]$; thus it is naturally isomorphic to the underlying additive group of this ring. We therefore just have to exhibit an element that is not in the image of the canonical homomorphism
\[
\bZ[\widehat{\bZ}_{2} \rtimes \bZ] \longrightarrow \lim_{r} \bZ[\bZ/2^{r-1}\bZ \rtimes \bZ],
\]
in other words a compatible sequence $f_{r}$ of polynomials with exponents in $\bZ/2^{r-1}\bZ \rtimes \bZ$ that cannot be obtained by truncation from a polynomial with exponents in $\widehat{\bZ}_{2} \rtimes \bZ$. For example, we may take the sequence
\[
f_{r} = \sum_{i=0}^{r-2} (t^{10\cdots 0 \rtimes 0} - 1),
\]
where $0\cdots 0$ indicates a string of zeros of length $i$. This sequence cannot arise as truncations of a single polynomial with exponents in $\widehat{\bZ}_{2} \rtimes \bZ$ since there is no upper bound on the length (in binary) of the exponents in the polynomials $f_{r}$.
\end{proof}

For completeness, we briefly consider also the case $k=1$ of the homomorphism \eqref{eq:canonical-map}. In this case it fails to be an isomorphism even in the simpler setting when $Q_{\bullet}$ is a tower of \emph{abelian} groups:

\begin{lem}
Let $k=1$ and let $Q_{\bullet}$ be the tower of abelian groups where $Q_r = \bZ/p^r$ and the maps $\bZ/p^{r+1} \to \bZ/p^r$ are reduction modulo $p^r$ for a prime $p\geq 5$. Then the map \eqref{eq:canonical-map} is not surjective.
\end{lem}
\begin{proof}
The inclusion of the trivial units $\{\pm 1\} \times G$ into the group of units $\bZ[G]^{\times}$ of the integral group ring $\bZ[G]$ induces a commutative square
\begin{equation}
\label{eq:inclusion-of-trivial-units}
\begin{tikzcd}
\bZ[\lim_{r}(Q_{r})]^{\times} \ar[rr,"(*)"] && \lim_{r}(\bZ[Q_{r}]^{\times}) \\
\{\pm 1\} \times \lim_{r}(Q_{r}) \ar[rr,"\cong"] \ar[u] && \lim_{r}(\{\pm 1\} \times Q_{r}) \ar[u,"(**)",swap]
\end{tikzcd}
\end{equation}
whose top horizontal arrow $(*)$ is the map \eqref{eq:canonical-map} for $k=1$ that we aim to prove is not surjective. The bottom horizontal arrow is an isomorphism because the endofunctor $\{\pm 1\} \times -$ preserves limits (as limits commute). The limit group $\lim_{r}(Q_r)$ is the additive group of $p$-adic integers $\widehat{\bZ}_p$, which is torsion-free abelian, hence orderable (see \cite{Levi1913}), and hence a unique product group (see \cite[\S 13, Lem.~1.7]{Passman77} for instance). As a consequence of \cite[Th.~1]{Strojnowski1980}, the group ring of a unique product group has only trivial units when the base ring is an integral domain (see \cite[\S 13.1]{Passman77} for further details), so the left-hand vertical arrow in \eqref{eq:inclusion-of-trivial-units} is an isomorphism. The maps $(*)$ and $(**)$ are therefore identified; we will show that $(**)$ is not surjective by constructing a compatible family of non-trivial units in the integral group rings $\bZ[Q_{r}] = \bZ[\bZ/p^r]$. Let us write $\bZ/p^r = \langle a_r \mid (a_r)^{p^r} = 1 \rangle$ and set $u_r := 1 - a_r + (a_r)^2 \in \bZ[\bZ/p^r]$. Then $u_r$ is a non-trivial unit for $p\geq 5$ (it is known as the \emph{alternating unit based on $a_r$ with parameter $3$}, and is a unit by \cite[Lem.~10.6]{Sehgal1993}). The maps of group rings $\bZ[\bZ/p^{r+1}] \to \bZ[\bZ/p^r]$ induced by reduction modulo $p^r$ send $a_{r+1}$ to $a_r$ and hence $u_{r+1}$ to $u_r$, so this is a compatible family of non-trivial units, in other words an element of $\lim_{r}(\bZ[Q_{r}]^{\times})$ that is not in the image of the right-hand vertical map $(**)$ of \eqref{eq:inclusion-of-trivial-units}. It follows that $(*)$ is not surjective.
\end{proof}

\section{NCP and eNCP homomorphisms}
\label{s:NCP}

As recalled in Example~\ref{eg:inverse-system}, the \emph{lower central series} of a group $G$ is the descending filtration of $G$ defined by $\LCS_{1}(G) = G$ and $\LCS_{r+1}(G) = [G,\LCS_{r}(G)]$. If $\LCS_{r}(G) = \LCS_{r+1}(G)$ for some $r$ (hence $\LCS_{r}(G) = \LCS_{r+i}(G)$ for all $i\geq 1$) we say that the lower central series of $G$ \emph{stops} (at $\LCS_{r}$). The stopping or non-stopping of the lower central series of partitioned braid groups and their relatives has been comprehensively studied in \cite{DPS}. In this section, we introduce an analogue of \emph{non-stopping lower central series} (a property of groups) for group homomorphisms (NCP) and for equivariant group homomorphisms (eNCP).

\begin{defn}
\label{def:ncp}
A group homomorphism $\varphi \colon K \to G$ is called \emph{nilpotency class preserving} (NCP) if, for each $r\geq 1$, we have $\varphi(\LCS_{r}(K)) \not\subseteq \LCS_{r+1}(G)$.
\end{defn}

\begin{rmk}
\label{rmk:NCP-and-non-stopping-LCS}
If $\varphi \colon K \to G$ is NCP, then both $K$ and $G$ have non-stopping lower central series, i.e., we have $\LCS_{r+1}(K) \neq \LCS_{r}(K)$ and $\LCS_{r+1}(G) \neq \LCS_{r}(G)$ for all $r\geq 1$. We also note that the lower central series of a group $G$ is non-stopping if and only if the identity $\id \colon G \to G$ is NCP. The notion of NCP is thus the natural analogue, for group homomorphisms, of the notion of non-stopping lower central series for groups.
\end{rmk}

\begin{defn}
\label{defn:Qr}
For any homomorphism $\varphi \colon K \to G$, define $Q_{r}(\varphi)$ to be the image of the induced homomorphism
\[
K/\LCS_{r} \longrightarrow G/\LCS_{r}.
\]
This is a quotient of $K/\LCS_{r}$ and a subgroup of $G/\LCS_{r}$, so it has nilpotency class at most $r-1$.
\end{defn}

When $K$ is the kernel of a split surjection of $G$, the subgroup $Q_{r}(\varphi)$ of $G/\LCS_{r}$ may alternatively be characterised as the kernel of the induced split surjection after applying $-/\LCS_{r}$:

\begin{lem}
\label{lem:Qr-alternative}
Let $G \twoheadrightarrow \Gamma$ be a split surjection with kernel $K$ and denote the inclusion by $\varphi \colon K \to G$. Then the subgroup $Q_{r}(\varphi)$ of $G/\LCS_{r}$ is equal to the kernel of the split surjection $G/\LCS_{r} \twoheadrightarrow \Gamma/\LCS_{r}$.
\end{lem}
\begin{proof}
This follows from the sequence of equalities
\begin{align*}
Q_{r}(\varphi) &= \mathrm{im}( K/\LCS_{r} \to G/\LCS_{r} ) &&\text{by definition;} \\
&= \mathrm{im} ( K \twoheadrightarrow K/\LCS_{r} \to G/\LCS_{r} ) &&\text{since $K \twoheadrightarrow K/\LCS_{r}$ is surjective;} \\
&= \mathrm{im} ( K \to G \twoheadrightarrow G/\LCS_{r} ) &&\text{as $(-) \to (-)/\LCS_{r}$ is natural, cf.\ \cite[Ex.~2.23]{PSI};} \\
&= \mathrm{ker} ( G/\LCS_{r} \twoheadrightarrow \Gamma/\LCS_{r} ) &&\text{by \cite[Lem.~2.24]{PSI} with $Q(-) = (-)/\LCS_{r}$.} \qedhere
\end{align*}
\end{proof}

\begin{lem}
\label{lem:nilpotency-class}
If $\varphi \colon K \to G$ is NCP, then $Q_{r}(\varphi)$ has nilpotency class exactly $r-1$.
\end{lem}
\begin{proof}
We will prove the contrapositive. Suppose that $Q_{r}(\varphi)$ has nilpotency class at most $r-2$. Since $Q_{r}(\varphi)$ is the image of the composition $K \to K/\LCS_{r} \to G/\LCS_{r}$, this assumption implies that the kernel of the composed map $K \to G/\LCS_{r}$ contains $\LCS_{r-1}(K)$. But the map $K \to G/\LCS_{r}$ also factors as $K \to G \to G/\LCS_{r}$, so its kernel is $\varphi^{-1}(\LCS_{r}(G))$. Thus we have $\varphi(\LCS_{r-1}(K)) \subseteq \LCS_{r}(G)$ and so $\varphi$ is not NCP.
\end{proof}

Suppose that we have a commutative square of groups
\begin{equation}
\label{eq:square}
\begin{tikzcd}
K \ar[d] \ar[rr,"\varphi"] && G \ar[d] \\
K' \ar[rr,"{\varphi'}"] && G'
\end{tikzcd}
\end{equation}
where the vertical homomorphisms are surjective.

\begin{lem}
\label{lem:lifting-NCP}
If $\varphi'$ is NCP then so is $\varphi$.
\end{lem}
\begin{proof}
We will prove the contrapositive. Suppose that $\varphi$ is not NCP, so there exists $r\geq 1$ so that $\varphi(\LCS_{r}(K)) \subseteq \LCS_{r+1}(G)$. Denote each of the surjections $K \twoheadrightarrow K'$ and $G \twoheadrightarrow G'$ by $\pi$. Then
\[
\varphi'(\LCS_{r}(K')) = \varphi'(\pi(\LCS_{r}(K))) = \pi(\varphi(\LCS_{r}(K))) \subseteq \pi(\LCS_{r+1}(G)) = \LCS_{r+1}(G')
\]
so $\varphi'$ is not NCP.
\end{proof}

The following special case gives a useful criterion for a homomorphism to be NCP:

\begin{coro}
\label{coro:lifting-NCP}
Let $\varphi \colon K \to G$ be a homomorphism and $\pi \colon G \to G'$ a surjective homomorphism onto a group $G'$ whose lower central series does not stop, such that $\pi \circ \varphi$ is also surjective. Then $\varphi$ is NCP.
\end{coro}
\begin{proof}
We may set $K' = G'$ and $\varphi' = \id$ in \eqref{eq:square}. Remark~\ref{rmk:NCP-and-non-stopping-LCS} implies that $\id \colon G' \to G'$ is NCP.
\end{proof}

We will also need an equivariant version of the NCP property. Fix a group $\Gamma$ and consider the category of groups equipped with left $\Gamma$-actions and $\Gamma$-equivariant group homomorphisms.

First, we define the \emph{span} of a normal subgroup in this category. Let $G$ be a group equipped with a left $\Gamma$-action and let $N$ be a normal subgroup of $G$ that is $\Gamma$-invariant (for example this holds if $\Gamma$ acts on $G$ by inner automorphisms or if $N$ is a characteristic subgroup). There is therefore a well-defined induced left $\Gamma$-action on the quotient group $G/N$.

\begin{defn}
\label{defn:Gamma-span}
The \emph{$\Gamma$-span} $\langle N \rangle_{\Gamma} \subseteq G$ is defined to be the kernel of the quotient of $G$ onto the coinvariants $(G/N)_{\Gamma}$ of the induced $\Gamma$-action on $G/N$, that is, the quotient $(G/N)/\langle\bar{g}(\gamma\cdot\bar{g})^{-1}\mid\bar{g}\in G/N,\gamma\in\Gamma\rangle$ where $\langle\bar{g}(\gamma\cdot\bar{g})^{-1}\mid\bar{g}\in G/N,\gamma\in\Gamma\rangle$ denotes the smallest normal subgroup of $G/N$ containing all elements of the form $\bar{g}(\gamma\cdot\bar{g})^{-1}$ with $\bar{g}\in G/N$ and $\gamma \in \Gamma$.
\end{defn}

\begin{wng}
This is not the same as the normal subgroup of $G$ generated by the set of elements $\{\gamma \cdot n \mid \gamma \in \Gamma, n\in N \}$. For example, if $\Gamma = G$ acting on itself by conjugation and $N$ is the trivial subgroup, then the normal subgroup generated by $\{\gamma \cdot n \mid \gamma \in \Gamma, n\in N \}$ is trivial whereas $\langle N \rangle_{\Gamma}$ is the commutator subgroup of $G$.
\end{wng}

\begin{defn}
\label{def:encp}
A $\Gamma$-equivariant group homomorphism $\varphi \colon K \to G$ is called \emph{equivariantly nilpotency class preserving} (eNCP) if, for each $r\geq 1$, we have $\varphi(\LCS_{r}(K)) \not\subseteq \langle \LCS_{r+1}(G) \rangle_{\Gamma}$.
\end{defn}

\begin{rmk}
\label{rmk:trivial-Gamma-action}
This recovers Definition~\ref{def:ncp} when $\Gamma$ is the trivial group, or more generally when $\Gamma$ is any group acting trivially on $G$, since in this case $\langle N \rangle_{\Gamma} = N$ for all characteristic subgroups $N \subseteq G$.
\end{rmk}

The notion of eNCP has a lifting property analogous to Lemma~\ref{lem:lifting-NCP}. Suppose that we have a commutative square in the category of groups equipped with left $\Gamma$-actions:
\begin{equation}
\label{eq:esquare}
\begin{tikzcd}
K \ar[d] \ar[rr,"\varphi"] && G \ar[d] \\
K' \ar[rr,"{\varphi'}"] && G'
\end{tikzcd}
\end{equation}
where the vertical homomorphisms are surjective. We first need a technical lemma:

\begin{lem}
\label{lem:Gamma-span-quotients}
Let $\pi \colon G \to G'$ be a $\Gamma$-equivariant homomorphism and let $N \subseteq G$ be a characteristic subgroup such that $\pi(N) \subseteq G'$ is also characteristic. Then $\pi(\langle N \rangle_{\Gamma}) \subseteq \langle \pi(N) \rangle_{\Gamma}$.
\end{lem}
\begin{proof}
This follows from an elementary diagram chase in the commutative square
\begin{equation*}
\begin{tikzcd}
G \ar[r,two heads] \ar[d] & G/N \ar[d] \ar[r,two heads] & (G/N)_{\Gamma} \ar[d] \\
G' \ar[r,two heads] & G'/\pi(N) \ar[r,two heads] & (G'/\pi(N))_{\Gamma},
\end{tikzcd}
\end{equation*}
in which $\langle N \rangle_{\Gamma}$ is the kernel of the composition across the top and $\langle \pi(N) \rangle_{\Gamma}$ is the kernel of the composition across the bottom.
\end{proof}

\begin{lem}
\label{lem:lifting-eNCP}
In diagram \eqref{eq:esquare}, if $\varphi'$ is eNCP then so is $\varphi$.
\end{lem}
\begin{proof}
As usual, we prove the contrapositive. Suppose that $\varphi$ is not eNCP, so there is some $r\geq 1$ so that $\varphi(\LCS_{r}(K)) \subseteq \langle \LCS_{r+1}(G) \rangle_{\Gamma}$. We thus have (denoting both of the the vertical homomorphisms of \eqref{eq:esquare} by $\pi$):
\begin{align*}
\varphi'(\LCS_{r}(K')) = \varphi'(\pi(\LCS_{r}(K))) &= \pi(\varphi(\LCS_{r}(K))) \\
&\subseteq \pi(\langle \LCS_{r+1}(G) \rangle_{\Gamma}) \\
&\subseteq \langle \pi(\LCS_{r+1}(G)) \rangle_{\Gamma} = \langle \LCS_{r+1}(G') \rangle_{\Gamma}
\end{align*}
where the inclusion on the bottom row follows from Lemma~\ref{lem:Gamma-span-quotients}. Thus $\varphi'$ is also not eNCP.
\end{proof}

\begin{coro}
\label{coro:lifting-eNCP}
Let $\varphi \colon K \to G$ be a $\Gamma$-equivariant homomorphism and $\pi \colon G \to G'$ a surjective, $\Gamma$-invariant homomorphism onto a group $G'$ whose lower central series does not stop, such that $\pi \circ \varphi$ is also surjective. Then $\varphi$ is eNCP.
\end{coro}
\begin{proof}
Consider the commutative square \eqref{eq:esquare} with $K' = G'$ equipped with the trivial $\Gamma$-action (we may do this since $\pi \colon G \to G'$ is assumed to be $\Gamma$-\emph{invariant}) and $\varphi' = \id$. By Lemma~\ref{lem:lifting-eNCP} it suffices to prove that $\id \colon G' \to G'$ is eNCP. By Remark~\ref{rmk:trivial-Gamma-action} this is the same as proving that $\id \colon G' \to G'$ is NCP, ignoring the (trivial) $\Gamma$-action. By Remark~\ref{rmk:NCP-and-non-stopping-LCS} this is the same as proving that the lower central series of $G'$ does not stop, which is part of our hypotheses.
\end{proof}

\begin{rmk}
Notice that the only additional hypothesis in Corollary~\ref{coro:lifting-eNCP} compared with Corollary~\ref{coro:lifting-NCP} is that the quotient $\pi \colon G \to G'$ is $\Gamma$-invariant.
\end{rmk}

We will be particularly interested in the setting of a split short exact sequence
\begin{equation}
\label{eq:ses}
\begin{tikzcd}
1 \ar[r] & K \ar[r] & G \ar[r] & \Gamma \ar[l,bend right=30,dashed] \ar[r] & 1.
\end{tikzcd}
\end{equation}
The quotient group $\Gamma$ acts by left-conjugation on $G$ and on $K$ via the given splitting (i.e.~its left action is $\gamma \cdot g = s(\gamma) g s(\gamma)^{-1}$ if $s$ denotes the splitting), so the inclusion $\varphi \colon K \to G$ becomes a $\Gamma$-equivariant group homomorphism.

\begin{rmk}
\label{rmk:lifting-eNCP}
In this setting, Corollary~\ref{coro:lifting-eNCP} implies that a sufficient criterior for $\varphi$ to be eNCP is the existence of a surjection $G \to G'$ onto a group $G'$ whose lower central series does not stop, such that the restriction $K \hookrightarrow G \to G'$ is also surjective and the composition $\Gamma \dashrightarrow G \to G'$ has image contained in the centre of $G'$. In our examples, we will typically have the stronger property that the composition $\Gamma \dashrightarrow G \to G'$ is the trivial map, so we will be in the following situation:
\begin{equation}
\label{eq:ses-with-quotient}
\begin{tikzcd}
1 \ar[r] & K \ar[r] \ar[dr,two heads] & G \ar[r] \ar[d,two heads] & \Gamma \ar[l,bend right=30,dashed] \ar[r] \ar[dl,dashed,"0"] & 1 \\
&& G'. &&
\end{tikzcd}
\end{equation}
\end{rmk}

Let $\varphi \colon K \to G$ be a $\Gamma$-equivariant homomorphism. Since the terms of the lower central series are characteristic, there is a well-defined induced left $\Gamma$-action on $K/\LCS_{r}$ and on $G/\LCS_{r}$, and the induced homomorphism $K/\LCS_{r} \to G/\LCS_{r}$ is $\Gamma$-equivariant. The image of this homomorphism, which by Definition~\ref{defn:Qr} is $Q_{r}(\varphi)$, thus also inherits a well-defined induced left $\Gamma$-action.

\begin{defn}
\label{defn:Qru}
In the above setting, we define:
\[
Q_{r}^{\mathrm{u}}(\varphi) = Q_{r}(\varphi)_{\Gamma} ,
\]
in other words $Q_{r}^{\mathrm{u}}(\varphi)$ is the coinvariants of the induced left $\Gamma$-action on $Q_{r}(\varphi)$.
\end{defn}

The superscript ${}^{\mathrm{u}}$ stands for ``untwisted'', since the purpose of taking this further quotient of $K$ is to ensure that the pro-nilpotent representations that we construct in \S\ref{s:general-recipe} commute with the module structure over the group ring of this quotient. This is in contrast to commuting only up to a ``twist'', which would give a weaker notion of representation than the usual one. See also Remark~\ref{rmk:twisted-version-of-the-construction}.

As a first observation, note that $Q_{r}^{\mathrm{u}}(\varphi)$ has nilpotency class at most $r-1$, since it is a quotient of $Q_{r}(\varphi)$, which has nilpotency class at most $r-1$. The following lemma, which is the equivariant analogue of Lemma~\ref{lem:nilpotency-class}, is the key technical result of this section.

\begin{lem}
\label{lem:e-nilpotency-class}
In the setting of a split short exact sequence \eqref{eq:ses}, if $\varphi \colon K \to G$ is eNCP, then the group $Q_{r}^{\mathrm{u}}(\varphi)$ has nilpotency class exactly $r-1$.
\end{lem}
\begin{proof}
We prove the contrapositive, so we assume that the nilpotency class of $Q_{r}^{\mathrm{u}}(\varphi)$ is at most $r-2$ and will show that $\varphi$ is not eNCP. Consider the following commutative diagram of groups with left $\Gamma$-actions:
\begin{equation}
\label{eq:diagram-of-quotients}
\begin{tikzcd}
K \ar[d,two heads] \ar[rr,"\varphi"] && G \ar[d,two heads] \\
K/\LCS_{r} \ar[r,two heads] & Q_{r}(\varphi) \ar[r,hook] \ar[d,two heads] & G/\LCS_{r} \ar[d,two heads] \\
& Q_{r}^{\mathrm{u}}(\varphi) \ar[r] & (G/\LCS_{r})_{\Gamma}
\end{tikzcd}
\end{equation}
Since $Q_{r}^{\mathrm{u}}(\varphi)$ is a quotient of $K$ and has nilpotency class at most $r-2$, the kernel of the quotient must contain $\LCS_{r-1}(K)$:
\[
\LCS_{r-1}(K) \subseteq \mathrm{ker}(K \twoheadrightarrow Q_{r}^{\mathrm{u}}(\varphi)).
\]
By commutativity of \eqref{eq:diagram-of-quotients}, this implies that
\[
\varphi(\LCS_{r-1}(K)) \subseteq \mathrm{ker}(G \twoheadrightarrow (G/\LCS_{r})_{\Gamma}).
\]
But the right-hand side is the definition of $\langle \LCS_{r}(G) \rangle_{\Gamma}$, so we have shown that $\varphi$ is not eNCP.
\end{proof}

\paragraph*{Outline of the remaining sections.}
In \S\ref{s:general-recipe} we consider the quotients $Q_{r}^{\mathrm{u}}(\varphi)$ of the group $K$ in a split short exact sequence \eqref{eq:ses} and apply Lemma~\ref{lem:e-nilpotency-class} to see that they have nilpotency class exactly $r$, assuming the eNCP property. This is the key technical ingredient in our general construction of pro-nilpotent representations. In \S\ref{s:Bn}--\S\ref{s:LBn}, we then consider specific examples of such split short exact sequences and prove the eNCP property in each case using Corollary~\ref{coro:lifting-eNCP} in the setting of Remark~\ref{rmk:lifting-eNCP}, that is, we produce a quotient group $G'$, fitting into diagram \eqref{eq:ses-with-quotient}, whose lower central series does not stop.

\section{The general recipe}\label{s:general-recipe}

Let $\Gamma$ be a group. We first describe a general recipe for constructing homological representations of $\Gamma$, before explaining how to augment this to produce pro-nilpotent representations. The idea of the general recipe is the same as that of \cite[\S 2]{PSI}, although the details of \cite[\S 2]{PSI} are significantly more involved, as the goal there is to construct representations of a \emph{family} of groups (encoded in a category), rather than a single group.

\paragraph*{Homological representations.}
We suppose that we are given three inputs:
\begin{compactenum}[(1)]
\item\label{input:ses} A split short exact sequence:
\begin{equation}
\label{eq:input-ses}
\begin{tikzcd}
1 \ar[r] & K \ar[r,"\varphi"] & G \ar[r] & \Gamma \ar[l,bend right=30,dashed] \ar[r] & 1.
\end{tikzcd}
\end{equation}
\item\label{input:fibration} A diagram of based, path-connected spaces:
\begin{equation}
\label{eq:input-fibration-sequence}
\begin{tikzcd}
X \ar[r,"i"] & Y \ar[r,"f"] & Z, \ar[l,bend right=40,dashed]
\end{tikzcd}
\end{equation}
where $f$ is a Serre fibration and $i$ is the inclusion of the fibre, that induces \eqref{eq:input-ses} on $\pi_{1}$.
\item\label{input:quotient} A quotient $K \twoheadrightarrow Q$ that is invariant under the left $\Gamma$-action on $K$ induced by \eqref{eq:input-ses}.
\end{compactenum}

\begin{rmk}
There is a canonical choice of the input \eqref{input:fibration}, given the input \eqref{input:ses}. Taking classifying spaces, the quotient $G \twoheadrightarrow \Gamma$ induces a map $BG \to B\Gamma$, which we may assume, up to based homotopy equivalence, is a Serre fibration. Its fibre is then a model for the classifying space $BK$. Taking classifying spaces is functorial, so the section of \eqref{eq:input-ses} induces a section for \eqref{eq:input-fibration-sequence}.

In many of our examples, the input \eqref{input:fibration} will indeed be of this canonical form, i.e.~the spaces $X$, $Y$ and $Z$ that we consider will be aspherical (have vanishing higher homotopy groups). However, not in all cases: in \S\ref{s:LBn} we apply our general construction to loop braid groups using certain configuration spaces of points and loops in the $3$-disc for the spaces $X$, $Y$ and $Z$, which are not aspherical.
\end{rmk}

\begin{lem}
\label{lem:general-recipe}
The three inputs above induce a well-defined representation of $\Gamma$ over the ring $\bZ[Q]$ for each homological degree $i\geq 0$.
\end{lem}
\begin{proof}
Composing the quotient $\pi_{1}(X) = K \twoheadrightarrow Q$ with the left regular representation of $Q$ on its group ring $\bZ[Q]$ (i.e.~the representation $Q\to \GL_{1}(\bZ[Q])$ defined by $q\mapsto (q \cdot -)$ where $(q \cdot -)$ is the left multiplication by $q$), we obtain a rank-$1$ local system on $X$ defined over $\bZ[Q]$. Here, a \emph{rank-$1$ local system on $X$ defined over a ring $R$} means a left representation of $\pi_{1}(X)$ on a right $R$-module that is free of rank $1$; equivalently, a $(\bZ[\pi_{1}(X)],R)$-bimodule that is free of rank $1$ as a right $R$-module. Let us denote this local system by $\cL_Q$. If $X$ is a sufficiently nice space so that it admits a universal cover (as it always will be in our examples), this local system may equivalently be viewed as a bundle of right $\bZ[Q]$-modules over $X$, constructed as follows: the quotient $\pi_{1}(X) = K \twoheadrightarrow Q$ corresponds to a regular covering $X^Q \to X$ with deck transformation group $Q$ (acting on the right); taking free abelian groups fibrewise turns this into a bundle of right $\bZ[Q]$-modules, which is the local system $\cL_Q$.

The fundamental group of the base of any Serre fibration acts (on the left) by homotopy automorphisms on its fibre (see for example \cite[\S 2]{PalmerTillmann} for the construction, although this is a classical fact). Thus we have a left $\Gamma$-action
\begin{equation}
\label{eq:action-by-homotopy-automorphisms}
\Gamma = \pi_{1}(Z) \longrightarrow \pi_{0}(\mathrm{hAut}(X)),
\end{equation}
where $\mathrm{hAut}(X)$ is the grouplike topological monoid of homotopy automorphisms of $X$ (i.e.~it is a monoid under composition, which is continuous with respect to the compact-open topology, and the discrete monoid $\pi_{0}(\mathrm{hAut}(X))$ is a group).
The assumption that $K \twoheadrightarrow Q$ is $\Gamma$-invariant means that the local system $\cL_Q$ on $X$ is invariant under the action \eqref{eq:action-by-homotopy-automorphisms}. There is therefore a well-defined induced (left) action on the twisted homology of $X$ with local system $\cL_Q$ in any degree $i\geq 0$:
\begin{equation}
\label{eq:hom-rep-ordinary-homology}
\Gamma \longrightarrow \mathrm{Aut}_{\bZ[Q]}(H_{i}(X;\cL_Q)).
\end{equation}
This is a representation over $\bZ[Q]$ since the local system $\cL_Q$ is defined over $\bZ[Q]$. More precisely, it is a left representation of $\Gamma$ in the category of right $\bZ[Q]$-modules.
\end{proof}

If the space $X$ is locally compact, we may alternatively take twisted \emph{Borel-Moore} homology in the last step of the above proof, to obtain another representation
\begin{equation}
\label{eq:hom-rep-BM-homology}
\Gamma \longrightarrow \mathrm{Aut}_{\bZ[Q]}(H_{i}^{\BM}(X;\cL_Q))
\end{equation}
of $\Gamma$ over $\bZ[Q]$ for each degree $i\geq 0$. We recall that the Borel-Moore homology group $H_{i}^{\BM}(X;\cL_Q)$ may be defined as the inverse limit $\lim_{A\in\compact(X)}(H_{i}(X,X\smallsetminus  A;\cL_Q))$, where $\compact(X)$ denotes the set of all compact subsets of $X$ partially ordered by inclusion. We also refer the reader to \cite[Chap.~V]{bredonsheaf} for a detailed introduction to Borel-Moore homology. The principal reason why we work with Borel-Moore homology instead of ordinary homology is the structural result of Theorem~\ref{thm:BM-homology-lemma} that we shall recall below.

\paragraph*{Weakly pro-nilpotent homological representations.}
To construct weakly pro-nilpotent representations of $\Gamma$, we fix inputs \eqref{input:ses} and \eqref{input:fibration} from above and allow input \eqref{input:quotient} (a $\Gamma$-invariant quotient of $K$) to vary. More precisely, we construct a canonical tower of $\Gamma$-invariant quotients of $K$ determined by the split short exact sequence \eqref{eq:input-ses}. Lemma~\ref{lem:e-nilpotency-class} will then imply that this tower is pro-nilpotent as long as \eqref{eq:input-ses} is eNCP.

The construction is summarised in the following diagram. See \cite{BellingeriGodelleGuaschi} for similar diagrams, involving a single level of nilpotence, in the setting of surface braid groups.

\begin{equation}
\label{eq:big-diagram-of-quotients}
\begin{split}
\begin{tikzpicture}
[x=1.2mm,y=1mm]
\node (l1) at (0,0) {$K$};
\node (m1) at (20,0) {$G$};
\node (r1) at (40,0) {$\Gamma$};
\inclusion{above}{$\varphi$}{(l1)}{(m1)}
\draw[->>] (m1) to (r1);
\draw[->,dotted] (r1) to[out=150,in=30] (m1);
\node (l2) at (0,-15) {$Q_{\infty}$};
\node (m2) at (20,-15) {$G/\LCS_{\infty}$};
\node (r2) at (40,-15) {$\Gamma/\LCS_{\infty}$};
\node (a2) at (-20,-20) {$Q^{\mathrm{u}}_{\infty}$};
\incl{(l2)}{(m2)}
\draw[->>] (m2) to (r2);
\draw[->>] (l2) to (a2);
\draw[->,dotted] (r2) to[out=150,in=30] (m2);
\draw[->>] (l1) to (l2);
\draw[->>] (m1) to (m2);
\draw[->>] (r1) to (r2);
\node (l3) at (0,-30) {$\lim(Q_{\bullet})$};
\node (m3) at (20,-30) {$\widehat{G}_{\mathrm{nil}}$};
\node (r3) at (40,-30) {$\widehat{\Gamma}_{\mathrm{nil}}$};
\node (a3) at (-20,-35) {$\lim(Q^{\mathrm{u}}_{\bullet})$};
\incl{(l3)}{(m3)}
\draw[->>] (m3) to (r3);
\draw[->] (l3) to (a3);
\draw[->,dotted] (r3) to[out=150,in=30] (m3);
\incl{(l2)}{(l3)}
\incl{(m2)}{(m3)}
\incl{(r2)}{(r3)}
\draw[->] (a2) to (a3);
\node (l4) at (0,-45) {$Q_{r+1}$};
\node (m4) at (20,-45) {$G/\LCS_{r+1}$};
\node (r4) at (40,-45) {$\Gamma/\LCS_{r+1}$};
\node (a4) at (-20,-50) {$Q^{\mathrm{u}}_{r+1}$};
\incl{(l4)}{(m4)}
\draw[->>] (m4) to (r4);
\draw[->>] (l4) to (a4);
\draw[->,dotted] (r4) to[out=150,in=30] (m4);
\draw[->>] (l3) to (l4);
\draw[->>] (m3) to (m4);
\draw[->>] (r3) to (r4);
\draw[->>] (a3) to (a4);
\node (l5) at (0,-60) {$Q_{r}$};
\node (m5) at (20,-60) {$G/\LCS_{r}$};
\node (r5) at (40,-60) {$\Gamma/\LCS_{r}$};
\node (a5) at (-20,-65) {$Q^{\mathrm{u}}_{r}$};
\incl{(l5)}{(m5)}
\draw[->>] (m5) to (r5);
\draw[->>] (l5) to (a5);
\draw[->,dotted] (r5) to[out=150,in=30] (m5);
\draw[->>] (l4) to (l5);
\draw[->>] (m4) to (m5);
\draw[->>] (r4) to (r5);
\draw[->>] (a4) to (a5);
\node (l6) at (0,-75) {$Q_{2}$};
\node (m6) at (20,-75) {$G^{\ab}$};
\node (r6) at (40,-75) {$\Gamma^{\ab}$};
\node (a6) at (-20,-80) {$Q^{\mathrm{u}}_{2}$};
\incl{(l6)}{(m6)}
\draw[->>] (m6) to (r6);
\draw[double equal sign distance] (l6) to (a6);
\draw[->,dotted] (r6) to[out=150,in=30] (m6);
\draw[->>] (l5) to (l6);
\draw[->>] (m5) to (m6);
\draw[->>] (r5) to (r6);
\draw[->>] (a5) to (a6);
\end{tikzpicture}
\end{split}
\end{equation}

The top row is the split short exact sequence \eqref{eq:input-ses}. We first describe the second-to-bottom row, where $r\geq 2$. The right-hand side is obtained by applying the functorial construction $-/\LCS_{r}$ (quotienting a group by the $r$-th term in its lower central series) to the right-hand side of \eqref{eq:input-ses}. We then define $Q_{r}$ to be the kernel of the induced surjection $G/\LCS_{r} \twoheadrightarrow \Gamma/\LCS_{r}$. By Lemma~\ref{lem:Qr-alternative}, this coincides with the definition of $Q_{r}(\varphi)$ in Definition~\ref{defn:Qr}; here we abbreviate it to $Q_{r}$ to avoid cluttering the diagram. There is an induced action of $\Gamma$ on $Q_{r}$ given by projecting $\Gamma$ onto $\Gamma/\LCS_{r}$, following the section to $G/\LCS_{r}$ and then acting by conjugation on $Q_{r}$. The quotient $Q_{r}^{\mathrm{u}}$ is defined to be the coinvariants of this action; this coincides with Definition~\ref{defn:Qru}; we are again abbreviating $Q_{r}^{\mathrm{u}}(\varphi)$ to $Q_{r}^{\mathrm{u}}$.

This describes the second-to-bottom row for any $r\geq 2$. The third-from-bottom row is the same, replacing $r$ with $r+1$; it has a canonical surjection onto the second-from-bottom row since the lower central series is a descending series. The bottom row is simply the $r=2$ instance of the second-from-bottom row; note that quotienting by $\LCS_{2}$ is the same as abelianisation. In the case $r=2$, the $\Gamma$-action on $Q_{2}$ is trivial, since this action is given by conjugation in the abelian group $G/\LCS_{2} = G^{\ab}$. Hence in this case $Q_{2}^{\mathrm{u}} = Q_{2}$.

This explains the bottom three rows, which in fact represent an infinite tower of $4$-term rows indexed by integers $r\geq 2$. The third-from-top row is defined by taking inverse limits of these four towers. On the right-hand side this gives us the pro-nilpotent completions of $G$ and of $\Gamma$, by Definition~\ref{d:pro-nilpotent-completion}. Taking pro-nilpotent completions is functorial, so the homomorphism $\widehat{G}_{\mathrm{nil}} \to \widehat{\Gamma}_{\mathrm{nil}}$ is again a split surjection, as indicated in the diagram.

The second-from-top row is constructed just as the second-from-bottom row, except that we quotient by the \emph{residue} $\LCS_{\infty}$, which is the intersection of all finite terms in the lower central series. In general, the canonical homomorphism $G \to \widehat{G}_{\mathrm{nil}}$ of a group to its pro-nilpotent completion factors as the quotient by its residue $G \twoheadrightarrow G/\LCS_{\infty}$ followed by an injection (see Remark~\ref{rmk:morphism-from-pronilpotent-completion}); this is why we have vertical injections $G/\LCS_{\infty} \hookrightarrow \widehat{G}_{\mathrm{nil}}$ and $\Gamma/\LCS_{\infty} \hookrightarrow \widehat{\Gamma}_{\mathrm{nil}}$ in the diagram. The vertical injection $Q_{\infty} \hookrightarrow \lim(Q_{\bullet})$ then follows by commutativity. This completes the construction of diagram \eqref{eq:big-diagram-of-quotients}.

For completeness, we mention that the composed vertical morphisms $G/\LCS_{\infty} \to G/\LCS_{r}$ are surjective for every $r$ (this does not follow from the information given in the diagram). Similarly for the vertical morphisms $\Gamma/\LCS_{\infty} \to \Gamma/\LCS_{r}$ and $Q_{\infty} \to Q_{r}$ and $Q^{\mathrm{u}}_{\infty} \to Q^{\mathrm{u}}_{r}$. Finally, we note that the two arrows that are not decorated as surjections or injections in the diagram are deliberately so: they are not tautologically injective or surjective in general.

\begin{lem}
\label{lem:general-recipe-weakly-pro-nilpotent}
Suppose we are given inputs \eqref{input:ses} and \eqref{input:fibration} and assume that $\varphi$ in \eqref{eq:input-ses} is eNCP. Then this determines a well-defined weakly pro-nilpotent representation of $\Gamma$ for each $i\geq 0$.
\end{lem}
\begin{proof}
Input \eqref{input:ses}, namely the split short exact sequence \eqref{eq:input-ses}, induces diagram \eqref{eq:big-diagram-of-quotients}, in particular it induces the tower $Q_{\bullet}^{\mathrm{u}}$ of groups in the bottom-left of this diagram. Since $\varphi$ is eNCP, Lemma~\ref{lem:e-nilpotency-class} implies that $Q_{r}^{\mathrm{u}}$ has nilpotency class exactly $r-1$. Thus $Q_{\bullet}^{\mathrm{u}}$ is a nilpotent tower of groups in the sense of Definition~\ref{def:nilpotent-tower-of-groups}.

The quotient $K \twoheadrightarrow Q_{r}^{\mathrm{u}}$ in diagram \eqref{eq:big-diagram-of-quotients} is $\Gamma$-invariant by construction, so we may apply Lemma~\ref{lem:general-recipe} with this quotient as input \eqref{input:quotient} to obtain the representation $V_{r} = H_{i}(X;\cL_{Q_{r}^{\mathrm{u}}})$ of $\Gamma$ over $\bZ[Q_{r}^{\mathrm{u}}]$.

To complete the construction of a weakly pro-nilpotent representation, we must now construct comparison homomorphisms
\begin{equation}
\label{eq:comparison-homomorphism}
V_{r+1} \otimes_{\bZ[Q_{r+1}^{\mathrm{u}}]} \bZ[Q_{r}^{\mathrm{u}}] \longrightarrow V_{r}
\end{equation}
of $\Gamma$-representations over $\bZ[Q_{r}^{\mathrm{u}}]$. To do this, first recall that, for any space $X$, local system $\cL$ on $X$ defined over $R$ and ring homomorphism $\theta \colon R \to S$, there is a canonical homomorphism
\begin{equation}
\label{eq:canonical-homomorphism}
H_{i}(X;\cL) \otimes_{R} S \longrightarrow H_{i}(X;\cL \otimes_{R} S)
\end{equation}
of $S$-modules commuting with the action of $\pi_{0}(\mathrm{hAut}(X))$. This is one of the maps appearing in the universal coefficient theorem, although we only need its existence. If we take $R = \bZ[Q_{r+1}^{\mathrm{u}}]$ and $S = \bZ[Q_{r}^{\mathrm{u}}]$, with $\theta$ induced by the quotient $Q_{r+1}^{\mathrm{u}} \twoheadrightarrow Q_{r}^{\mathrm{u}}$, and set $\cL = \cL_{Q_{r+1}^{\mathrm{u}}}$, then \eqref{eq:canonical-homomorphism} becomes \eqref{eq:comparison-homomorphism} since $\cL_{Q_{r+1}^{\mathrm{u}}} \otimes_{\bZ[Q_{r+1}^{\mathrm{u}}]} \bZ[Q_{r}^{\mathrm{u}}] = \cL_{Q_{r}^{\mathrm{u}}}$. It is a homomorphism of $\Gamma$-representations since the action of $\Gamma$ factors through $\pi_{0}(\mathrm{hAut}(X))$.
\end{proof}

If $X$ is locally compact, we may alternatively apply Borel-Moore homology \eqref{eq:hom-rep-BM-homology} rather than ordinary homology \eqref{eq:hom-rep-ordinary-homology}, in which case we set $V_{r} = H_{i}^{\BM}(X;\cL_{Q_{r}^{\mathrm{u}}})$. We then construct the comparison homomorphisms \eqref{eq:comparison-homomorphism} as follows. The canonical homomorphism \eqref{eq:canonical-homomorphism} exists also for relative homology. Quantifying over all compact subsets $A \subseteq X$ and taking inverse limits, we obtain:
\begin{align}
\begin{split}
\label{eq:canonical-homomorphism-BM-homology}
H_{i}^{\BM}(X;\cL) \otimes_{R} S &= \lim_{A\in\compact(X)} (H_{i}(X,X \smallsetminus A;\cL)) \otimes_{R} S \\
&\to \lim_{A\in\compact(X)} (H_{i}(X,X \smallsetminus A;\cL) \otimes_{R} S) \\
&\to \lim_{A\in\compact(X)} (H_{i}(X,X \smallsetminus A;\cL \otimes_{R} S)) \\
&= H_{i}^{\BM}(X;\cL \otimes_{R} S).
\end{split}
\end{align}
The first arrow above is induced by the universal property of $\lim_{A\in\compact(X)} (H_{i}(X,X \smallsetminus A;\cL) \otimes_{R} S)$ as an inverse limit (it is not in general an isomorphism since $- \otimes_R S$ does not commute with inverse limits), while the second arrow is the map induced by taking the inverse limit of the homomorphisms analogous to \eqref{eq:canonical-homomorphism} in relative homology (its failure to be an isomorphism in general is controlled by the universal coefficient theorem).
Specialising as above, this gives us \eqref{eq:comparison-homomorphism} in the Borel-Moore setting.

\begin{rmk}
\label{rmk:twisted-version-of-the-construction}
If, in Lemma~\ref{lem:general-recipe-weakly-pro-nilpotent}, we make the weaker assumption that $\varphi$ in \eqref{eq:input-ses} is NCP, rather than eNCP, then we may apply the same construction using the tower of quotients $Q_{\bullet}$ of $K$ instead of $Q_{\bullet}^{\mathrm{u}}$. However, since the quotients $K \twoheadrightarrow Q_{r}$ are not in general $\Gamma$-invariant (although their kernels are preserved by the $\Gamma$-action), we obtain representations $H_{i}(X;\cL_{Q_{r}})$ of $\Gamma$ that commute with the $\bZ[Q_{r}]$-module structure only up to a ``twist''.
\end{rmk}

\paragraph*{Lifting to representations over a pro-nilpotent group.}

As an aside, we observe that, given the same inputs as in Lemma~\ref{lem:general-recipe-weakly-pro-nilpotent}, the weakly pro-nilpotent representation over $Q^{\mathrm{u}}_{\bullet}$ constructed in that lemma may be lifted to a representation over $\bZ[\lim(Q^{\mathrm{u}}_{\bullet})]$.

\begin{lem}
\label{lem:lifting-to-limit}
Suppose we are given inputs \eqref{input:ses} and \eqref{input:fibration} and assume that $\varphi$ in \eqref{eq:input-ses} is eNCP. Then this determines a representation of $\Gamma$ over the integral group ring of the pro-nilpotent group $\lim(Q^{\mathrm{u}}_{\bullet})$, which lifts the weakly pro-nilpotent representation of $\Gamma$ over $Q^{\mathrm{u}}_{\bullet}$ from Lemma~\ref{lem:general-recipe-weakly-pro-nilpotent}.
\end{lem}
\begin{proof}
Consider just the top two rows of diagram \eqref{eq:big-diagram-of-quotients}. The quotient $K \twoheadrightarrow Q^{\mathrm{u}}_{\infty}$ is $\Gamma$-invariant by construction, so Lemma~\ref{lem:general-recipe} constructs a representation $V_{\infty} = H_{i}(X;\cL_{Q_{\infty}^{\mathrm{u}}})$ of $\Gamma$ over $\bZ[Q_{\infty}^{\mathrm{u}}]$. This lifts the representation $V_{r} = H_{i}(X;\cL_{Q_{r}^{\mathrm{u}}})$ of $\Gamma$ over $\bZ[Q_{r}^{\mathrm{u}}]$ in the sense that there are homomorphisms
\begin{equation}
\label{eq:comparison-homomorphism-infinity}
V_{\infty} \otimes_{\bZ[Q_{\infty}^{\mathrm{u}}]} \bZ[Q_{r}^{\mathrm{u}}] \longrightarrow V_{r} ,
\end{equation}
constructed exactly as in the proof of Lemma~\ref{lem:general-recipe-weakly-pro-nilpotent}. The homomorphisms \eqref{eq:comparison-homomorphism-infinity} are compatible with the comparison homomorphisms \eqref{eq:comparison-homomorphism}, so the representation $V_{\infty}$ of $\Gamma$ over $\bZ[Q_{\infty}^{\mathrm{u}}]$ lifts the (weakly) pro-nilpotent representation $V_{\bullet}$ of $\Gamma$ over $Q^{\mathrm{u}}_{\bullet}$ constructed in Lemma~\ref{lem:general-recipe-weakly-pro-nilpotent}.

This is not yet exactly what we want, which is a representation over $\bZ[\lim(Q^{\mathrm{u}}_{\bullet})]$. We cannot construct this directly, since the homomorphism $K \to \lim(Q^{\mathrm{u}}_{\bullet})$ in diagram \eqref{eq:big-diagram-of-quotients} is not necessarily surjective. Also, the homomorphism $Q_{\infty}^{\mathrm{u}} \to \lim(Q^{\mathrm{u}}_{\bullet})$ in diagram \eqref{eq:big-diagram-of-quotients} is not necessarily injective, so we also cannot simply consider $V_{\infty}$ as a representation over $\bZ[\lim(Q^{\mathrm{u}}_{\bullet})]$ by inclusion of rings. Instead, let us denote by $Q^{\mathrm{u}}_{\lim}$ the image of the homomorphism $Q_{\infty}^{\mathrm{u}} \to \lim(Q^{\mathrm{u}}_{\bullet})$ in diagram \eqref{eq:big-diagram-of-quotients}. The homomorphism $K \to Q^{\mathrm{u}}_{\lim}$ from diagram \eqref{eq:big-diagram-of-quotients} is thus surjective by construction, so the (weakly) pro-nilpotent representation $V_{\bullet}$ lifts to a representation $V_{\lim}$ of $\Gamma$ over $\bZ[Q^{\mathrm{u}}_{\lim}]$ by a verbatim repeat of the previous paragraph with $Q_{\infty}^{\mathrm{u}}$ replaced by $Q^{\mathrm{u}}_{\lim}$. Since $\bZ[Q^{\mathrm{u}}_{\lim}]$ is now a subring of $\bZ[\lim(Q^{\mathrm{u}}_{\bullet})]$, this finishes the desired construction.
\end{proof}

We note that the above construction goes through equally well if we work with Borel-Moore homology instead of ordinary homology.

\begin{rmk}
This is in contrast to the general situation of (weakly) pro-nilpotent representations over a pro-nilpotent tower of groups vs.\ representations over the group-ring of the inverse limit of the tower: in \S\ref{s:pro-nilpotent} (see Lemma~\ref{lem:pronilpotent-counterexample}) we observed that a lift does \emph{not} always exist in general.
\end{rmk}

\paragraph*{Genuine pro-nilpotent homological representations.}
We now consider conditions on the space $X$ guaranteeing that the above weakly pro-nilpotent representations are in fact \emph{genuine} pro-nilpotent representations, i.e.~the comparison homomorphisms \eqref{eq:comparison-homomorphism} are isomorphisms.

\begin{lem}
\label{lem:criterion-for-comparison-isomorphisms}
Let $X$ be a space, $\cL$ a local system on $X$ over $R$ and $\theta \colon R \to S$ a ring homomorphism. Let $X' \subseteq X$ be a subspace such that the inclusion induces isomorphisms on twisted Borel-Moore homology for all local systems on $X$. Suppose moreover that $X'$ is homeomorphic to a disjoint union of copies of $\bR^k$ for fixed $k\geq 0$. Then \eqref{eq:canonical-homomorphism-BM-homology} is an isomorphism.
\end{lem}
\begin{proof}
Consider the commutative square:
\[
\begin{tikzcd}
H_{i}^{\BM}(X';\cL) \otimes_{R} S \ar[rr] \ar[d] && H_{i}^{\BM}(X';\cL \otimes_{R} S) \ar[d] \\
H_{i}^{\BM}(X;\cL) \otimes_{R} S \ar[rr] && H_{i}^{\BM}(X;\cL \otimes_{R} S).
\end{tikzcd}
\]
The vertical maps are isomorphisms by hypothesis. The top horizontal map is an isomorphism since both sides are canonically isomorphic to the free $S$-module generated by the components of $X'$. Note that all local systems on $X'$ are trivial, in other words untwisted, since it has contractible components. Hence the bottom horizontal map is also an isomorphism.
\end{proof}

The hypotheses of Lemma~\ref{lem:criterion-for-comparison-isomorphisms} appear a little ad hoc, but they occur very naturally when the space $X$ is a (partitioned) configuration space of points, as we discuss shortly. First, we combine Lemmas~\ref{lem:general-recipe-weakly-pro-nilpotent} and \ref{lem:criterion-for-comparison-isomorphisms} to prove:

\begin{coro}
\label{coro:general-recipe-pro-nilpotent}
Suppose we are given inputs \eqref{input:ses} and \eqref{input:fibration}. Assume that $\varphi$ in \eqref{eq:input-ses} is eNCP. Assume also that $X$ is locally compact and that there is a subspace $X'$ of $X$, homeomorphic to a disjoint union of copies of $\bR^k$ for a fixed $k\geq 0$, such that the inclusion $X' \hookrightarrow X$ induces isomorphisms on twisted Borel-Moore homology for all local systems on $X$. Then this determines a well-defined (genuine) pro-nilpotent representation of $\Gamma$.
\end{coro}
\begin{proof}
Since $X$ is locally compact, we may apply the Borel-Moore variant of the construction in Lemma~\ref{lem:general-recipe-weakly-pro-nilpotent} to obtain a weakly pro-nilpotent representation $(V_{r})_{r\geq 2}$ of $\Gamma$ with comparison homomorphisms $\eqref{eq:comparison-homomorphism} = \eqref{eq:canonical-homomorphism-BM-homology}$. Lemma~\ref{lem:criterion-for-comparison-isomorphisms} implies that these are isomorphisms, so this is in fact a genuine pro-nilpotent representation.
\end{proof}

In our examples, the space $X$ will typically be a partitioned configuration space
\[
C_{\bk}(N) = \{ (p_{1},\ldots,p_{k}) \in N^k \mid p_{i} \neq p_{j} \text{ for } i\neq j \} / \mathfrak{S}_{\bk} ,
\]
where $N$ is a manifold, $\bk = (k_{1},\ldots,k_{l})$ is an $l$-tuple of positive integers summing to $k\geq 1$ and $\mathfrak{S}_{\bk} = \mathfrak{S}_{k_{1}} \times \cdots \times \mathfrak{S}_{k_{l}}$, considered as a subgroup of $\mathfrak{S}_{k}$.

In \cite[Th.~2.1]{PSIIp}, we prove the following theorem, which gives sufficient conditions for an inclusion of configuration spaces to induce isomorphisms on twisted Borel-Moore homology. It is a generalisation of a result originally due to Bigelow~\cite[Lem.~3.1]{BigelowHomrep}, and also recovers, as special cases, similar results appearing in \cite[Lem.~3.3]{AnKo}, \cite[Th.~6.6]{anghelpalmer} and \cite[Th.~A(a)]{BlanchetPalmerShaukat}.

\begin{thm}[{\cite[Th.~2.1]{PSIIp}}]
\label{thm:BM-homology-lemma}
Let $M$ be a compact metric space with closed subspaces $A \subseteq B \subseteq M$, where $M$ and $B$ are locally compact. Suppose that there exists a strong deformation retraction $h$ of $M$ onto $B$, in other words a map $h \colon [0,1] \times M \to M$ satisfying the following two conditions:
\begin{itemizeb}
\item $h(t,x)=x$ whenever $t=0$ or $x \in B$,
\item $h(1,x) \in B$ for all $x \in M$,
\end{itemizeb}
such that moreover the following two additional conditions hold:
\begin{itemizeb}
\item $h(t,-)$ is non-expanding for all $t$, i.e. $d(x,y) \geq d(h(t,x),h(t,y))$ for all $x,y \in M$,
\item $h(t,-)$ is a topological self-embedding of $M$ for all $t<1$.
\end{itemizeb}
Then, for any $\bk = (k_{1},\ldots,k_{l})$, the inclusion of configuration spaces
\begin{equation}
\label{eq:inclusion-of-configuration-spaces}
C_{\bk}(B \smallsetminus A) \lhook\joinrel\longrightarrow C_{\bk}(M \smallsetminus A)
\end{equation}
induces isomorphisms on Borel-Moore homology in all degrees and for all local coefficient systems on $C_{\bk}(M \smallsetminus A)$ that extend to $C_{\bk}(M)$.
\end{thm}

\begin{coro}
\label{coro:BM-homology-lemma}
Let $M$ be a compact manifold with boundary, $\Lambda \subseteq M$ an embedded finite graph and $A \subseteq \Lambda \cap \partial M$ a subspace such that $\Lambda \smallsetminus A$ is a disjoint union of open intervals. Suppose that there is a strong deformation retraction of $M$ onto $\Lambda$ satisfying the two additional conditions of Theorem~\ref{thm:BM-homology-lemma} with respect to some metric on $M$. Then, for any $\bk = (k_{1},\ldots,k_{l})$, the space $X = C_{\bk}(M \smallsetminus A)$ satisfies the hypotheses of Corollary~\ref{coro:general-recipe-pro-nilpotent}.
\end{coro}
\begin{proof}
We may apply Theorem~\ref{thm:BM-homology-lemma} and take $X' = C_{\bk}(\Lambda \smallsetminus A)$. Since $A$ is contained in $\partial M$, the inclusion $C_{\bk}(M \smallsetminus A) \subseteq C_{\bk}(M)$ is a homotopy equivalence, so all local coefficient systems on $C_{\bk}(M \smallsetminus A)$ extend to $C_{\bk}(M)$. Moreover, any configuration space on a disjoint union of open intervals is homeomorphic to a disjoint union of copies of $\bR^k$, where $k$ is the total number of points in a configuration. Thus $X$ admits a subspace $X'$ with the required properties. Finally, $X = C_{\bk}(M \smallsetminus A)$ is locally compact since it is a manifold.
\end{proof}

\begin{outlook}
\label{outlook-for-examples}
To summarise, in order to construct weakly pro-nilpotent representations of $\Gamma$ it suffices to construct inputs \eqref{input:ses} and \eqref{input:fibration} -- in other words a split fibration sequence \eqref{eq:input-fibration-sequence} whose induced split short exact sequence of fundamental groups is \eqref{eq:input-ses} -- and to check that the inclusion $\varphi \colon K \hookrightarrow G$ in \eqref{eq:input-ses} is eNCP (this is Lemma~\ref{lem:general-recipe-weakly-pro-nilpotent}). If the space $X$ of \eqref{eq:input-fibration-sequence} is of the form $X = C_{\bk}(M \smallsetminus A)$ as in Corollary~\ref{coro:BM-homology-lemma}, then we obtain a \emph{genuine} pro-nilpotent representation of $\Gamma$ (by Corollaries~\ref{coro:general-recipe-pro-nilpotent} and \ref{coro:BM-homology-lemma}).

In \S\ref{s:Bn}--\S\ref{s:LBn} we apply this recipe in many different settings where $\Gamma$ is either a classical braid group, surface braid group or loop braid group. For classical and surface braid groups (\S\ref{s:Bn} and \S\ref{s:SBn}) we carry out the full recipe -- constructing \eqref{eq:input-fibration-sequence}, proving that $\varphi$ is eNCP and checking that $X$ has the form described in Corollary~\ref{coro:BM-homology-lemma} -- and thus obtain genuine pro-nilpotent representations of these groups. On the other hand, for loop braid groups (\S\ref{s:LBn}) the space $X$ in our construction is not of the right form to apply Corollary~\ref{coro:BM-homology-lemma}, so for these groups we only construct weakly pro-nilpotent representations. The reason is that, in this setting, $X$ is either a configuration space involving higher-dimensional objects than just points, or it is a configuration space of points in a manifold that does not deformation retract onto a graph as in Corollary~\ref{coro:BM-homology-lemma}.
\end{outlook}

\section{Classical braid groups}
\label{s:Bn}

In this section, we prove Theorems~\ref{athm:pro-nilpotent}, \ref{athm:ribbon-Lawrence} and \ref{athm:3-variable-LKB}: we construct the pro-nilpotent LKB representation in \S\ref{ss:Bn-pro-nilpotent}, the ribbon-Lawrence representations in \S\ref{ss:Bn-ribbon} and we compute explicit matrices for the second ribbon-Lawrence representation, which is also in a sense the ``limit'' of the pro-nilpotent LKB representation, in \S\ref{ss:Bn-matrices}.

\subsection{The pro-nilpotent LKB representation}
\label{ss:Bn-pro-nilpotent}

We apply the general construction of \S\ref{s:general-recipe} to prove Theorem \ref{athm:pro-nilpotent}.

\begin{proof}[Proof of Theorem \ref{athm:pro-nilpotent}]
Consider the split fibration sequence (cf.~\cite[Th.~3]{Fadell-Neuwirth})
\begin{equation}
\label{eq:input-fibration-sequence-Bn}
\begin{tikzcd}
C_{2}(\bD_{n}) \ar[r] & C_{2,n}(\bD^2) \ar[r] & C_{n}(\bD^2), \ar[l,bend right=25,dashed]
\end{tikzcd}
\end{equation}
where $\bD_{n}$ denotes the disc $\bD^2$ minus $n$ interior points and $C_{\bk}(\phantom{-})$ for a partition $\bk$ denotes the partitioned configuration space (defined just before Theorem~\ref{thm:BM-homology-lemma}). By Lemma~\ref{lem:general-recipe-weakly-pro-nilpotent}, this determines a weakly pro-nilpotent representation of $\B_{n} = \pi_{1}(C_{n}(\bD^2))$ as long as the inclusion $\varphi$ in the induced split short exact sequence of fundamental groups
\begin{equation}
\label{eq:input-ses-Bn}
\begin{tikzcd}
1 \ar[r] & \B_{2}(\bD_{n}) \ar[r,"\varphi"] & \B_{2,n} \ar[r] & \B_{n} \ar[l,bend right=25,dashed] \ar[r] & 1.
\end{tikzcd}
\end{equation}
is eNCP. By Corollary~\ref{coro:lifting-eNCP} (see also Remark~\ref{rmk:lifting-eNCP}), it will suffice to construct a quotient $\pi \colon \B_{2,n} \twoheadrightarrow G'$ fitting into the diagram
\begin{equation}
\label{eq:ses-with-quotient-Bn}
\begin{tikzcd}
1 \ar[r] & \B_{2}(\bD_{n}) \ar[r,"\varphi"] \ar[dr,two heads] & \B_{2,n} \ar[r] \ar[d,two heads,"\pi"] & \B_{n} \ar[l,bend right=25,dashed] \ar[r] \ar[dl,dashed,"0"] & 1 \\
&& G' &&
\end{tikzcd}
\end{equation}
where the lower central series of $G'$ does not stop. We will construct this for $G' = \bZ^2 \rtimes \mathfrak{S}_{2}$, whose lower central series does not stop by \cite[Prop.~A.28]{DPS}.

To construct the desired quotient $\B_{2,n} \twoheadrightarrow \bZ^2 \rtimes \mathfrak{S}_{2}$, we follow the proof of \cite[Prop.~3.12]{DPS}, where we set $l=2$. It is shown there that the quotient $\B_{2,n} / \LCS_{\infty}$ is isomorphic to $\bZ \times (\bZ^2 \rtimes \bZ)$, where a generator of the last $\bZ$ factor acts on the $\bZ^2$ factor by swapping the coordinates. This action has order two, so we may project the last $\bZ$ factor onto $\bZ/2 = \mathfrak{S}_{2}$, and also project away from the first (direct) $\bZ$ factor, to obtain a quotient onto $\bZ^2 \rtimes \mathfrak{S}_{2}$.

As shown in \cite[proof of Proposition~3.12]{DPS}, the generator $\sigma_{1}$ of $\B_{2,n}$ (a half-twist of the first two strands) is sent to the generator $s$ of $\mathfrak{S}_{2}$ in this quotient and the generator $a_{13}$ of $\B_{2,n}$ (a pure braid where all strands are vertical except the first one, which winds once around the third strand) is sent to the element $(1,0) \in \bZ^2 \subseteq \bZ^2 \rtimes \mathfrak{S}_{2}$ in this quotient. Since the two elements $s$ and $(1,0)$ generate $\bZ^2 \rtimes \mathfrak{S}_{2}$ and their pre-images $\sigma_{1}$ and $a_{13}$ lie in the subgroup $\B_{2}(\bD_{n})$, it follows that the restriction of the surjection $\B_{2,n} \twoheadrightarrow \bZ^2 \rtimes \mathfrak{S}_{2}$ to $\B_{2}(\bD_{n})$ is also surjective.

The quotient $\B_{2,n} \twoheadrightarrow \bZ^2 \rtimes \mathfrak{S}_{2}$ may be thought of as recording the winding numbers of the two configuration points in the first block of the partition around the $n$ configuration points in the second block (in the $\bZ^2$ factor) together with the induced permutation of the $2$-point block in the base configuration (in the $\mathfrak{S}_{2}$ factor). From this description it is clear that its composition with the section $\B_{n} \dashrightarrow \B_{2,n}$ is zero.

We have thus constructed the quotient $G' = \bZ^2 \rtimes \mathfrak{S}_{2}$ with the necessary properties, so $\varphi$ is eNCP by Corollary~\ref{coro:lifting-eNCP} and Lemma~\ref{lem:general-recipe-weakly-pro-nilpotent} implies that \eqref{eq:input-ses-Bn} induces a weakly pro-nilpotent representation of $\B_{n}$.

To see that this is a \emph{genuine} pro-nilpotent representation, we apply Corollary~\ref{coro:BM-homology-lemma}. Let $M$ be the $2$-disc with the interiors of $n$ pairwise disjoint closed discs removed, let $A$ be the union of its $n$ inner boundary circles and let $\Lambda$ be the embedded finite graph in $M$ given by the union of $A$ with $n-1$ arcs passing between consecutive boundary circles (see Figure~\ref{fig:basis}). It is easy to construct an appropriate deformation retraction of $M$ onto $\Lambda$, so Corollary~\ref{coro:BM-homology-lemma} implies that the configuration space $X = C_{2}(\bD_{n}) \cong C_{2}(M \smallsetminus A)$ satisfies the hypotheses of Corollary~\ref{coro:general-recipe-pro-nilpotent}, which implies that the weakly pro-nilpotent representation of $\B_{n}$ that we have constructed is a (genuine) pro-nilpotent representation.

Finally, we consider the bottom $(r=2)$ level of this pro-nilpotent representation. This is constructed from
\begin{equation}
\label{eq:input-ses-Bn-level-2}
\begin{tikzcd}
1 \ar[r] & \B_{2}(\bD_{n}) \ar[d,two heads] \ar[r] & \B_{2,n} \ar[d,two heads] \ar[r] & \B_{n} \ar[d,two heads] \ar[l,bend right=25,dashed] \ar[r] & 1 \\
1 \ar[r] & Q_{2} \ar[r] & \B_{2,n}^{\ab} \ar[r] & \B_{n}^{\ab} \ar[l,bend right=25,dashed] \ar[r] & 1.
\end{tikzcd}
\end{equation}
For $n\geq 2$, we have $\B_{2,n}^{\ab} \cong \bZ^3$ (see for example \cite[Prop.~3.5]{DPS}) and $\B_{n}^{\ab} \cong \bZ$, so $Q_{2} \cong \bZ^2$. Moreover, the quotient $\B_{2}(\bD_{n}) \twoheadrightarrow Q_{2} \cong \bZ^2$ precisely corresponds to the local system $\cL_{2}$ from the definition of the LKB representation \eqref{eq:LKB}; thus the $r=2$ level of our pro-nilpotent representation of $\B_{n}$ is the LKB representation. This concludes the proof of Theorem \ref{athm:pro-nilpotent}.
\end{proof}

\subsection{Ribbon Lawrence representations}
\label{ss:Bn-ribbon}

Let us consider in more detail the top two rows of diagram \eqref{eq:big-diagram-of-quotients} in the setting of \S\ref{ss:Bn-pro-nilpotent}:
\begin{equation}
\label{eq:input-ses-Bn-level-infinity}
\begin{tikzcd}
1 \ar[r] & \B_{2}(\bD_{n}) \ar[d,two heads] \ar[r] & \B_{2,n} \ar[d,two heads] \ar[r] & \B_{n} \ar[d,two heads] \ar[l,bend right=25,dashed] \ar[r] & 1 \\
1 \ar[r] & Q_{\infty} \ar[r] & \B_{2,n}/\LCS_{\infty} \ar[r] & \B_{n}/\LCS_{\infty} \ar[l,bend right=25,dashed] \ar[r] & 1.
\end{tikzcd}
\end{equation}
For $n\geq 2$, we have $\LCS_{2}(\B_{n}) = \LCS_{\infty}(\B_{n})$ (by \cite{GorinLin1969} for $n\geq 5$ and \cite[Example~2.3]{DPS} for $n\geq 2$) and so $\B_{n}/\LCS_{\infty} = \B_{n}^{\ab} \cong \bZ$. For $n\geq 3$, we also have $\B_{2,n}/\LCS_{\infty} \cong \bZ \times (\bZ^2 \rtimes \bZ)$, where, as described above, the right-hand $\bZ$ factor acts on the $\bZ^2$ factor by powers of the involution interchanging the two coordinates \cite[Prop.~3.12]{DPS}. Moreover, from the proof of \cite[Prop.~3.12]{DPS}, one sees that, under these identifications, the projection $\B_{2,n}/\LCS_{\infty} \twoheadrightarrow \B_{n}/\LCS_{\infty}$ is the projection onto the left-hand $\bZ$ factor. From this calculation we may conclude the following.
\begin{itemizeb}
\item We have $Q_{\infty} \cong \bZ^2 \rtimes \bZ$, where $1 \in \bZ$ acts on $\bZ^2$ by swapping the two coordinates.
\item The induced $\B_{n}$-action on $Q_{\infty}$ is trivial, since $\B_{2,n}/\LCS_{\infty}$ is the direct product of $Q_{\infty}$ and $\B_{n}/\LCS_{\infty}$. Hence we have $Q^{\mathrm{u}}_{\infty} = Q_{\infty}$.
\item For each finite $r \geq 2$, the homomorphism $Q_{\infty} \to Q_{r}$ is surjective (cf.\ the paragraph just before Lemma~\ref{lem:general-recipe-weakly-pro-nilpotent}) and $\B_{n}$-equivariant (because the $\B_{n}$-actions on $Q_{\infty}$ and on $Q_{r}$ are both induced by the commutative diagram \eqref{eq:big-diagram-of-quotients}), so it follows from the previous point that the induced $\B_{n}$-action on $Q_{r}$ is also trivial, and so we have $Q^{\mathrm{u}}_{r} = Q_{r}$.
\end{itemizeb}
This establishes the $k=2$, $n\geq 3$ case of Theorem \ref{athm:ribbon-Lawrence}. We next prove Theorem \ref{athm:ribbon-Lawrence} for all $k,n\geq 2$.

\begin{proof}[Proof of Theorem \ref{athm:ribbon-Lawrence}]
We first observe that, for any inclusion of surfaces $e \colon S \hookrightarrow T$, there is a homomorphism
\begin{equation}
\label{eq:inclusion-induced-map}
\B_{k}(S) \longrightarrow \pi_{1}(S) \wr \B_{k}(T) := \pi_{1}(S)^k \rtimes \B_{k}(T),
\end{equation}
defined as follows. Choose an ordering $\vec{c} = (c_{1},\ldots,c_{k})$ of the base configuration $c$ of $C_{k}(S)$ and an embedded disc $\iota \colon \bD^2 \hookrightarrow S$ whose image contains $c$. Given a loop $\beta \in \B_{k}(S) = \pi_{1}(C_{k}(S),c)$, we may lift it uniquely to a path $(\gamma_{1},\ldots,\gamma_{k})$ in the ordered configuration space $F_{k}(S)$ starting at $\vec{c}$. Each path $\gamma_{i}$ in $S$ begins and ends in the image of $\iota$, so collapsing $\iota(\bD^2)$ to a point and identifying the resulting surface $S/\iota(\bD^2)$ with $S$, we obtain a collection of loops $\bar{\gamma}_{i}$ in $\pi_{1}(S)$. The homomorphism \eqref{eq:inclusion-induced-map} is then defined to send $\beta$ to $((\bar{\gamma}_{1},\ldots,\bar{\gamma}_{k}),e \circ \beta)$.

Applying \eqref{eq:inclusion-induced-map} in the case $(T,S) = (\bD^2,\bD_{n})$, together with the projection $\pi_{1}(\bD_{n}) \cong \F_{n} \twoheadrightarrow \bZ$ sending each generator of a free basis for $\pi_{1}(\bD_{n})$ to $1 \in \bZ$, we obtain a quotient
\begin{equation}
\label{eq:ribbon-quotient}
\B_{k}(\bD_{n}) \relbar\joinrel\twoheadrightarrow \bZ \wr \B_{k} = \RB_{k}
\end{equation}
onto the $k$-th ribbon braid group. Explicitly, a braid $\beta = (\gamma_{1},\ldots,\gamma_{k}) \in \B_{k}(\bD_{n})$ is sent to the element $((w_{1},\ldots,w_{k}),\beta) \in \bZ^k \rtimes \B_{k} = \bZ \wr \B_{k}$, where $\gamma_{i}$ is the strand of $\beta$ starting at $c_{i}$ and $w_{i}$ is the total winding number of the loop $\bar{\gamma}_{i}$ in $\bD_{n} / \iota(\bD^2) \cong \bD_{n}$. Notice that, in the case where $k=2$, we have $\RB_{2} = \bZ \wr \bZ = \bZ^2 \rtimes \bZ$ and in addition the quotient \eqref{eq:ribbon-quotient} coincides with the quotient onto $Q_{\infty} = \bZ^2 \rtimes \bZ$ from diagram \eqref{eq:input-ses-Bn-level-infinity} if $n\geq 3$ (this follows from the explicit description of the quotient $\B_{2,n} \twoheadrightarrow \B_{2,n}/\LCS_{\infty}$ from the proof of \cite[Prop.~3.12]{DPS}).

The homomorphism \eqref{eq:inclusion-induced-map} is equivariant with respect to the evident action of the group of diffeomorphisms of $T$ that send $S$ onto itself. In the case $(T,S) = (\bD^2,\bD_{n})$ this means that the homomorphism
\begin{equation}
\label{eq:inclusion-induced-map-punctured-disc}
\B_{k}(\bD_{n}) \longrightarrow \pi_{1}(\bD_{n}) \wr \B_{k}
\end{equation}
is $\B_{n}$-equivariant, where the $\B_{n}$-action on $\B_{k}$ is trivial and its action on $\pi_{1}(\bD_{n}) \cong \F_{n}$ is the Artin representation. The projection $p \colon \F_{n} \twoheadrightarrow \bZ$ is clearly $\B_{n}$-invariant, hence so is $(p \wr \id) \circ \eqref{eq:inclusion-induced-map-punctured-disc} = \eqref{eq:ribbon-quotient}$.

Since the quotient \eqref{eq:ribbon-quotient} of $\pi_{1}(C_{k}(\bD_{n}))$ is $\B_{n}$-invariant, we may apply twisted Borel-Moore homology to obtain an induced $\B_{n}$-action on the $\bZ[\RB_{k}]$-module
\[
V_{\R}(k) := H_{k}^{\BM}(C_{k}(\bD_{n}) ; \cL_{\RB_{k}}),
\]
where $\cL_{\RB_{k}}$ is the local system on $C_{k}(\bD_{n})$ corresponding to \eqref{eq:ribbon-quotient}. This is the desired representation $\mathscr{RL}_{k}$ of Theorem \ref{athm:ribbon-Lawrence}. The $k$-th Lawrence representation \eqref{eq:Lawrence} is
\[
V_{k} = H_{k}^{\BM}(C_{k}(\bD_{n}) ; \cL_{k}),
\]
where $\cL_{k}$ is the local system on $C_{k}(\bD_{n})$ corresponding to the quotient $\B_{k}(\bD_{n}) \twoheadrightarrow \bZ^2$ sending a braid $\beta = (\gamma_{1},\ldots,\gamma_{k}) \in \B_{k}(\bD_{n})$ to $(w,b)$, where $w$ is the total winding number of all strands of the braid $\beta$ around the punctures of $\bD_{n}$ and $b \in \B_{k}^{\ab} = \bZ$ is the abelianisation of $\beta$ viewed as a braid in $\bD^2$. Notice that, in terms of the explicit description of \eqref{eq:ribbon-quotient} above, we have $w = w_{1} + \cdots + w_{k}$. Hence this quotient factors as
\begin{equation}
\label{eq:ribbon-quotient-factorisation}
\begin{tikzcd}
\B_{k}(\bD_{n}) \ar[r,two heads,"{\eqref{eq:ribbon-quotient}}"] & \RB_{k} = \bZ \wr \B_{k} \ar[r,two heads,"(*)"] & (\bZ \wr \B_{k})^{\ab} = \bZ \times \B_{k}^{\ab} = \bZ^2 .
\end{tikzcd}
\end{equation}
If we use the projection $(*)$ from \eqref{eq:ribbon-quotient-factorisation} to view $\bZ[\bZ^2]$ as a $\bZ[\RB_{k}]$-module, we have an isomorphism $\cL_{2} \cong \cL_{\RB_{k}} \otimes_{\bZ[\RB_{k}]} \bZ[\bZ^2]$, and hence:
\begin{align*}
V_{\R}(k) \otimes_{\bZ[\RB_{k}]} \bZ[\bZ^2] &= H_{k}^{\BM}(C_{k}(\bD_{n}) ; \cL_{\RB_{k}}) \otimes_{\bZ[\RB_{k}]} \bZ[\bZ^2] \\
&\cong H_{k}^{\BM}(C_{k}(\bD_{n}) ; \cL_{\RB_{k}} \otimes_{\bZ[\RB_{k}]} \bZ[\bZ^2]) \\
&\cong H_{k}^{\BM}(C_{k}(\bD_{n}) ; \cL_{2}) = V_{k} = \eqref{eq:Lawrence},
\end{align*}
where the middle isomorphism follows from the fact that $H_{k}^{\BM}(C_{k}(\bD_{n}) ; \cL)$, for \emph{any} rank-$1$ local system $\cL$ over a ring $R$, is a free $R$-module; this latter fact follows from Corollary~\ref{coro:BM-homology-lemma} exactly as in the proof of Theorem \ref{athm:pro-nilpotent} above. Thus we have shown that $V_{\R}(k)$ recovers \eqref{eq:Lawrence} after reducing along the abelianisation of $\RB_{k}$. This concludes the proof of Theorem \ref{athm:ribbon-Lawrence} except for its last statement, which was already explained in the paragraph before this proof.
\end{proof}

\begin{rmk}
The definition of the homomorphism \eqref{eq:inclusion-induced-map} is inspired by the definition of the homomorphism $\pi_S$ of \cite[\S 6.2.1]{DPS}. Indeed, the latter is the composition of \eqref{eq:inclusion-induced-map} for $S=T$ with the canonical quotient of $\pi_{1}(S) \wr \B_{k}(S)$ onto $\pi_{1}(S) \wr \mathfrak{S}_{k}$.
\end{rmk}

\begin{rmk}
We note that the group $\RB_{2} \cong \bZ^2 \rtimes \bZ$ is residually nilpotent but not nilpotent, so its lower central series does not stop (cf.~\cite[Prop.~A.10]{DPS}), which is a necessary condition for inducing a pro-nilpotent representation. On the other hand, for $k\geq 3$, the lower central series of $\RB_{k}$ stops at $\LCS_{2}$, so $\mathscr{RL}_{k}$ does not induce a pro-nilpotent representation in this case.
\end{rmk}

\subsection{Formulas for the three-variable LKB representation}
\label{ss:Bn-matrices}

We now consider in more detail the representation $\mathscr{RL}_{2}$ of $\B_{n}$ on the $\bZ[\RB_{2}]$-module $V_\R(2)$. Recall from the introduction that we write:
\[
\Theta = \bZ[\RB_{2}] = \bZ[\bZ^2 \rtimes \bZ] = \bZ\langle q_{1}^{\pm 1},q_{2}^{\pm 1},t^{\pm 1} \rangle / (q_{1} q_{2} = q_{2} q_{1} , q_{1} t = t q_{2} , q_{2} t = t q_{1}).
\]
As a $\Theta$-module, $V_\R(2)$ is free of rank $\binom{n}{2}$ (we will recall an explicit basis below), so each generator $\sigma_{i}$ of $\B_{n}$ acts by an $\binom{n}{2} \times \binom{n}{2}$ matrix over $\Theta$. Our goal in the remainder of this section is to prove Theorem \ref{athm:3-variable-LKB}, which states that this matrix is the one described in Table~\ref{tab:LKB3}.

\paragraph*{Basis.}
We first describe a basis of $V_{\R}(k)$ over $\Theta$ for all $k\geq 1$ (we will later come back to the special case $k=2$). As in the proof of Theorem \ref{athm:pro-nilpotent} above, let $M$ be the $2$-disc with the interiors of $n$ pairwise disjoint closed discs removed, let $A$ be the union of its $n$ inner boundary circles and let $\Lambda$ be the embedded finite graph in $M$ given by the union of $A$ with $n-1$ arcs passing between consecutive boundary circles; see Figure~\ref{fig:basis}.

\begin{figure}[htb]
    \centering
    \includegraphics[scale=0.8]{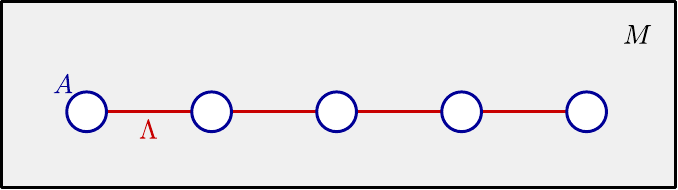}
    \caption{The connected, compact, planar surface $M$ with $n+1$ boundary components. The subspace $A$ is the union of its $n$ inner boundary components; the embedded graph $\Lambda$ is the union of $A$ with the $n-1$ arcs depicted.}
    \label{fig:basis}
\end{figure}

There is then a strong deformation retraction of $M$ onto $B = \Lambda$ satisfying the conditions of Theorem~\ref{thm:BM-homology-lemma}, so that theorem implies that the inclusion $I_{n} := \Lambda \smallsetminus A \hookrightarrow M \smallsetminus A \cong \bD_{n}$ induces an isomorphism
\[
H_{k}^{\BM}(C_{k}(I_{n}) ; \cL_{\RB_{k}}) \;\cong\; H_{k}^{\BM}(C_{k}(\bD_{n}) ; \cL_{\RB_{k}}) = V_{\R}(k).
\]
Since $I_{n}$ is a disjoint union of open $n-1$ intervals, the configuration space $C_{k}(I_{n})$ is a disjoint union of open $k$-balls indexed by tuples $(k_{1},\ldots,k_{n-1})$ of non-negative integers summing to $k$. The path-component corresponding to the tuple $(k_{1},\ldots,k_{n-1})$ naturally identifies with the product of open simplices $\mathring{\Delta}^{k_{1}} \times \cdots \times \mathring{\Delta}^{k_{n-1}}$, which is homeomorphic to an open $k$-ball. In particular, each path-component is simply-connected, so the restriction of the local system $\cL_{\RB_{k}}$ to $C_{k}(I_{n})$ is trivial. We deduce:

\begin{prop}
\label{prop:explicit-basis}
As a module, the representation $V_{\R}(k) \cong H_{k}^{\BM}(C_{k}(I_{n});\cL_{\RB_{k}})$ has a free basis over $\bZ[\RB_{k}]$ indexed by tuples $(k_{1},\ldots,k_{n-1})$ of non-negative integers summing to $k$. A generator corresponding to this tuple is given by the fundamental class of the properly-embedded submanifold $L(k_{1},\ldots,k_{n-1})$ of $C_{k}(\bD_{n})$ equal to the subspace of all configurations where exactly $k_{i}$ points lie on the $i$-th open interval of $I_{n} = \Lambda \smallsetminus A$.
\end{prop}

\paragraph*{Intersection form.}
We now note that the subspace $C_{k}^\partial(\bD_{n}) \subset C_{k}(\bD_{n})$ consisting of configurations that intersect the boundary $\partial \bD_{n} = \partial \bD^2$ non-trivially is preserved under the $\B_{n}$-action. This is in fact precisely the boundary of the manifold $C_{k}(\bD_{n})$. Thus, instead of applying twisted Borel-Moore homology $H_{k}^{\BM}$ to the space $C_{k}(\bD_{n})$ equipped with the ($\B_{n}$-invariant) local system $\cL_{\RB_{k}}$, we may alternatively apply ordinary (relative) homology to the pair of spaces $(C_{k}(\bD_{n}) , C_{k}^\partial(\bD_{n}))$ equipped with the same local system, to obtain another $\B_{n}$-representation
\[
V_{\R}(k)^\partial := H_{k}(C_{k}(\bD_{n}) , C_{k}^\partial(\bD_{n}) ; \cL_{\RB_{k}})
\]
over $\bZ[\RB_{k}]$. For each tuple $(k_{1},\ldots,k_{n-1})$ of non-negative integers summing to $k$, consider the submanifold $L(k_{1},\ldots,k_{n-1})^\partial$ of $C_{k}(\bD_{n})$ consisting of configurations where exactly one point lies on each vertical closed interval from Figure~\ref{fig:basis2}, where there are exactly $k_{i}$ such intervals between the $i$-th and $(i+1)$-st inner boundaries of $M$. This is a compact, contractible submanifold (homeomorphic to the $k$-cube $[0,1]^k$) whose boundary lies in $C_{k}^\partial(\bD_{n})$, so it has a fundamental class in $V_{\R}(k)^\partial$. These classes are ``dual'' to the basis of Proposition~\ref{prop:explicit-basis} with respect to the following bilinear form.

\begin{defn}
The intersection form
\begin{equation}
\label{eq:intersection-form}
\langle - \mathbin{,} - \rangle \colon V_{\R}(k) \otimes V_{\R}(k)^\partial \longrightarrow \bZ[\RB_{k}]
\end{equation}
is defined by $\langle x \mathbin{,} y \rangle = x^\vee \cap y$, where $x^\vee \in H^k(C_{k}(\bD_{n}) , C_{k}^\partial(\bD_{n}) ; \cL_{\RB_{k}})$ is the Poincar{\'e} dual of $x$ and $\cap$ is the relative cap product, taking values in $H_0(C_{k}(\bD_{n}) ; \cL_{\RB_{k}}) \cong \bZ[\RB_{k}]$.
\end{defn}

For the fundamental classes described above, we have
\begin{equation}
\label{eq:intersection-form-on-bases}
\bigl\langle [L(k_{1},\ldots,k_{n-1})] \mathbin{,} [L(k'_{1},\ldots,k'_{n-1})^\partial] \bigr\rangle = \begin{cases}
1 & \text{if } (k_{1},\ldots,k_{n-1}) = (k'_{1},\ldots,k'_{n-1}) \\
0 & \text{otherwise}.
\end{cases}
\end{equation}
As a consequence, the elements $[L(k_{1},\ldots,k_{n-1})^\partial] \in V_{\R}(k)^\partial$ are linearly independent over $\bZ[\RB_{k}]$ and if we decompose an element $x \in V_{\R}(k)$ as a linear combination in the basis $[L(k_{1},\ldots,k_{n-1})]$, the coefficients of this decomposition are $\langle x \mathbin{,} [L(k_{1},\ldots,k_{n-1})^\partial] \rangle$.

\begin{figure}[htb]
    \centering
    \includegraphics[scale=0.8]{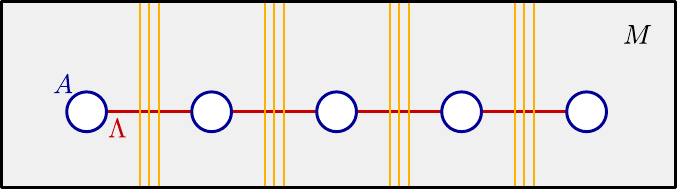}
    \caption{Vertical closed intervals determining an element of 
    $V_{\R}(k)^\partial$. In this case, we have $n=5$, $k=12$ and the element corresponds to the tuple $(3,3,3,3)$.}
    \label{fig:basis2}
\end{figure}

\begin{rmk}
We have elided a small subtlety above, namely that the fundamental class of a submanifold determines a homology class only up to the action of a unit of the ground ring, in this case $\bZ[\RB_{k}]^\times = \{\pm 1\} \times \RB_{k}$. To resolve the $\{\pm 1\}$ ambiguity we may choose an orientation of the submanifold and to resolve the $\RB_{k}$ ambiguity we may choose a path from a point on the submanifold to the basepoint of $C_{k}(\bD_{n})$ (two choices of such a path differ by an element of $\pi_{1}(C_{k}(\bD_{n}))$, which projects to an element of $\RB_{k}$). We assume that we have made such choices above so that \eqref{eq:intersection-form-on-bases} holds as written; without fixing these choices we can only say that \eqref{eq:intersection-form-on-bases} is equal to a unit when $(k_{1},\ldots,k_{n-1}) = (k'_{1},\ldots,k'_{n-1})$, not necessarily $1$. We will now make these choices explicit in the case $k=2$.
\end{rmk}

\paragraph*{Explicit orientations and paths to the basepoint.}
We now specialise to the case $k=2$. In this case, the homology elements under consideration are fundamental classes of embedded surfaces in the $4$-manifold $C_{2}(\bD_{n})$, and come in two kinds: those with $k_{i} = 2$ for some $i$ and those with $k_{i} = k_{j} = 1$ for some $i\neq j$ (and all other $k_{l} = 0$). In this case, we will make explicit choices, for each embedded surface in $C_{2}(\bD_{n})$, of an orientation and a path to the basepoint. The chosen paths to the basepoint are illustrated in Figure~\ref{fig:tethers}. The orientations are determined by chosen orientations of the arcs (also illustrated in Figure~\ref{fig:tethers}), together with the paths to the basepoint, as prescribed by the following convention.

\begin{convention}
\label{orientation-convention}
Let $p = (p_{1},p_{2}) \in C_{2}(\bD_{n})$ be the endpoint of the path to the basepoint that lies on the embedded surface in question. We order the two points $p_{1},p_{2}$ of this configuration so that, after following the path to the basepoint (which is a configuration in the bottom edge of the rectangle), the point $p_{1}$ ends up being to the left of the point $p_{2}$. It is enough to specify a local orientation of the embedded surface at the point $p$, in other words an ordered pair of linearly independent tangent vectors to the embedded surface at this point. The points $p_{1}$ and $p_{2}$ each lie on a smooth arc (possibly the same arc, possibly different arcs) with a chosen orientation as illustrated in Figure~\ref{fig:tethers}; this determines a non-zero tangent vector $v_{i}$ at $p_{i}$ in $\bD_{n}$ for $i=1,2$. A tangent vector at $p$ in $C_{2}(\bD_{n})$ is a choice of tangent vectors in $\bD_{n}$ at each of $p_{1}$ and $p_{2}$; for example we have $(v_{1},0)$ and $(0,v_{2})$, which are both tangent to the embedded surface. The local orientation at $p$ of the embedded surface is then the ordered pair $((v_{1},0),(0,v_{2}))$. Notice that this convention for choosing an orientation of the embedded surface depends not only on the orientations of the arcs involved, but also critically on the path to the basepoint: if we introduce a half-twist to this path, the orientation will be reversed.
\end{convention}

\begin{figure}[htb]
    \centering
    \includegraphics[scale=0.8]{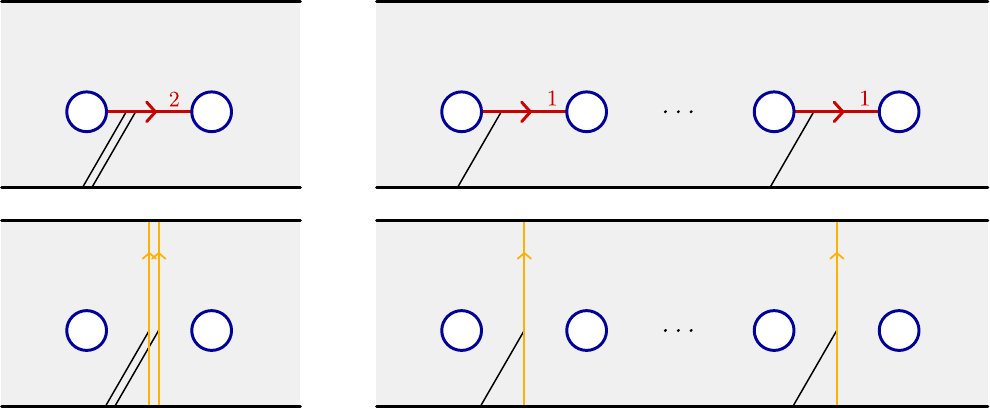}
    \caption{Chosen paths to the basepoint of $C_{2}(\bD_{n})$ for the embedded surfaces $L(k_{1},\ldots,k_{n-1})$ (top row) and $L(k_{1},\ldots,k_{n-1})^\partial$ (bottom row) for a tuple $(k_{1},\ldots,k_{n-1})$ of non-negative integers summing to $2$. On the left is the case when $k_{i} = 2$ for some $i$ and on the right is the case when $k_{i} = k_{j} = 1$ for some $i\geq j$. We consider any $2$-point configuration contained in the bottom edge of the rectangle to be the basepoint: this ambiguity does not matter since this subspace of $C_{2}(\bD_{n})$ is contractible.}
    \label{fig:tethers}
\end{figure}

\paragraph*{General w-classes and v-classes.}
We will need to consider homology elements of a slightly more general form. Let $\alpha$ be an arc in $M$ whose endpoints lie on $A$ and whose interior $\mathring{\alpha}$ lies in $M \smallsetminus A = \bD_{n}$. The fundamental class $[C_{2}(\mathring{\alpha})] \in H_{2}^{\BM}(C_{2}(\bD_{n});\cL_{\RB_{2}})$ is then well-defined up to a unit; it becomes well-defined on the nose, by Convention~\ref{orientation-convention}, after choosing an orientation of the arc $\alpha$ and a path from some point of $C_{2}(\mathring{\alpha})$ to the basepoint of $C_{2}(\bD_{n})$. Similarly, let $\alpha_{1}$ and $\alpha_{2}$ be two disjoint arcs of this form and consider the subsurface $(\mathring{\alpha}_{1} \times \mathring{\alpha}_{2})/\mathfrak{S}_{2} \subset C_{2}(\bD_{n})$ consisting of configurations with one point in the interior of each arc. Choosing orientations of $\alpha_{1}$ and $\alpha_{2}$ and a path from some point of $(\mathring{\alpha}_{1} \times \mathring{\alpha}_{2})/\mathfrak{S}_{2}$ to the basepoint of $C_{2}(\bD_{n})$, we have a well-defined fundamental class $[(\mathring{\alpha}_{1} \times \mathring{\alpha}_{2})/\mathfrak{S}_{2}] \in H_{2}^{\BM}(C_{2}(\bD_{n}) ; \cL_{\RB_{2}})$. We denote these elements (\emph{w-classes} and \emph{v-classes} respectively) by:
\begin{align*}
w(\alpha) &= [C_{2}(\mathring{\alpha})] \\
v(\alpha_{1},\alpha_{2}) &= [(\mathring{\alpha}_{1} \times \mathring{\alpha}_{2})/\mathfrak{S}_{2}],
\end{align*}
where it is implicit that the arcs are oriented and we have chosen an appropriate path to the basepoint. Finally, we may perform the same construction to arcs $\alpha_{1}$ and $\alpha_{2}$ that, instead of having their endpoints on $A \subset M$, have their endpoints on $\partial \bD^2 \subset M$ (and lie entirely in $M \smallsetminus A = \bD_{n}$); in this case we obtain an element $v(\alpha_{1},\alpha_{2}) \in H_{2}(C_{2}(\bD_{n}),C_{2}^\partial(\bD_{n});\cL_{\RB_{2}})$. Clearly all of the elements $[L(k_{1},\ldots,k_{n-1})]$ and $[L(k_{1},\ldots,k_{n-1})^\partial]$ considered above are of this form (in Figure~\ref{fig:tethers}, the top-left element is of the form $w(\alpha)$ and the others are all of the form $v(\alpha_{1},\alpha_{2})$).

\paragraph*{The intersection form on w-classes and v-classes.}
There is an explicit description of the intersection form \eqref{eq:intersection-form} for homology classes of this kind. Let $x$ be either $w(\alpha)$ or $v(\alpha_{1},\alpha_{2})$ and let $y$ be $v(\alpha'_{1},\alpha'_{2})$ for appropriate oriented arcs. Denote by $\gamma_x,\gamma_y$ the chosen paths to the basepoint for $x,y$ respectively. Assume that $\alpha$ or $\alpha_{1} \sqcup \alpha_{2}$ intersects $\alpha'_{1} \sqcup \alpha'_{2}$ transversely. For each intersection point $p \in x \cap y$, define a loop $\ell_p$ in $C_{2}(\bD_{n})$ by following $\gamma_x$ from the basepoint to the subsurface $x$, then following a path in $x$ to the intersection point $p$, then following a path in $y$ to the endpoint of $\gamma_y$ and then following $\gamma_y$ back to the basepoint. This loop induces a permutation of the base configuration of $C_{2}(\bD_{n})$; denote the sign of this permutation by $\mathrm{sgn}(\ell_p)$. It also determines an element $\phi(\ell_p) \in \RB_{2}$ via the projection $\phi \colon \pi_{1}(C_{2}(\bD_{n})) \twoheadrightarrow \RB_{2}$. Finally, write $p = \{p_{1},p_{2}\}$ and denote by $\mathrm{sgn}(p_{i})$ the sign of the intersection of the oriented arcs at $p_{i} \in \bD_{n}$. Then we have:
\begin{equation}
\label{eq:intersection-form-formula}
\langle x \mathbin{,} y \rangle = \sum_{p \in x \cap y} \mathrm{sgn}(p_{1})\mathrm{sgn}(p_{2})\mathrm{sgn}(\ell_p)\phi(\ell_p) \in \bZ[\RB_{2}] = \Theta.
\end{equation}
See \cite[page 475, ten lines above Lemma 2.1]{bigelow2001braid} and \cite[Appendix~B]{BlanchetPalmerShaukat} for an explanation of the signs appearing in this formula.

\paragraph*{Calculation of the matrices.}
With this setup, and especially the explicit formula \eqref{eq:intersection-form-formula} for the intersection form, we may now begin the proof of Theorem \ref{athm:3-variable-LKB}.

\begin{proof}[Proof of Theorem \ref{athm:3-variable-LKB}]
Let $1\leq i\leq n-1$ and let $\sigma_{i}$ be a diffeomorphism of $\bD_{n}$ representing $\sigma_{i} \in \B_{n}$. Using the basis $[L(k_{1},\ldots,k_{n-1})]$ of $V_\R(2)$, it follows from the discussion above that the entry of the matrix for $\mathscr{RL}_{2}(\sigma_{i})$ in the column corresponding to $(k_{1},\ldots,k_{n-1})$ and the row corresponding to $(k'_{1},\ldots,k'_{n-1})$ is
\[
\bigl\langle [\sigma_{i}(L(k_{1},\ldots,k_{n-1}))] \mathbin{,} [L(k'_{1},\ldots,k'_{n-1})^\partial] \bigr\rangle \in \bZ[\RB_{2}] = \Theta,
\]
which we may calculate using the formula \eqref{eq:intersection-form-formula}. Let us assume for convenience that $2\leq i\leq n-2$ (the edge cases $i \in \{1,n-1\}$ may be dealt with similarly). We order the basis for $V_\R(2)$ as follows:
\begin{itemizeb}
\item the six basis elements corresponding to the tuple $\cdots xyz \cdots$ (where $y$ is in the $i$-th position) for $xyz = 101, 200, 110, 020, 011, 002$;
\item $i-2$ blocks of three basis elements corresponding to the tuple $\cdots 1 \cdots xyz \cdots$ (where $y$ is in the $i$-th position) for $xyz = 100, 010, 001$;
\item $n-i-2$ blocks of three basis elements corresponding to the tuple $\cdots xyz \cdots 1 \cdots$ (where $y$ is in the $i$-th position) for $xyz = 100, 010, 001$;
\item the $\binom{n-3}{2}$ basis elements with $000$ in the $(i-1)$-st, $i$-th and $(i+1)$-st positions, in any order.
\end{itemizeb}
Since $\sigma_{i}$ is supported in a punctured subdisc of $\bD_{n}$ containing the $i$-th and $(i+1)$-st punctures and no other punctures, it is easy to see that the matrix $\mathscr{RL}_{2}(\sigma_{i})$ is a block matrix with respect to this partition: in other words, it consists of a $6 \times 6$ block, then $n-4 = (i-2) + (n-i-2)$ separate $3 \times 3$ blocks, followed by the identity $\binom{n-3}{2} \times \binom{n-3}{2}$ matrix. It remains to show that these $6 \times 6$ and $3 \times 3$ blocks are as claimed in Table~\ref{tab:LKB3}. We will explicitly compute three entries of the $6 \times 6$ matrix to explain how to do this; the remaining entries follow by exactly the same method.

Let us first compute the intersection
\begin{equation}
\label{eq:intersection-form-computation-1}
\bigl\langle [\sigma_{i}(L(\cdots 020 \cdots))] \mathbin{,} [L(\cdots 020 \cdots)^\partial] \bigr\rangle .
\end{equation}
Figure~\ref{fig:computation1} illustrates the two embedded surfaces $\sigma_{i}(L(\cdots 020 \cdots))$ and $L(\cdots 020 \cdots)^\partial$ and their unique intersection point $p$. The local signs $\mathrm{sgn}(p_{1})$ and $\mathrm{sgn}(p_{2})$ are both $-1$. The loop $\ell_p$ may be written as
\begin{equation}
\label{eq:ellp1}
\ell_p = \theta_{1,i+1} \theta_{2,i+1} \sigma ,
\end{equation}
where we are writing composition in $\pi_{1}(C_{2}(\bD_{n}))$ from left to right, $\sigma$ is the element that swaps the two points anticlockwise and $\theta_{j,l}$, for $j\in \{1,2\}$ and $l \in \{1,\ldots,n\}$, is the element where the $j$-th point loops once anticlockwise around the $l$-th puncture. Its induced permutation of the base configuration is non-trivial, so $\mathrm{sgn}(\ell_p) = -1$. The projection $\phi \colon \pi_{1}(C_{2}(\bD_{n})) \twoheadrightarrow \RB_{2}$ sends $\sigma \mapsto t$, $\theta_{1,l} \mapsto q_{1}$ and $\theta_{2,l} \mapsto q_{2}$ for all $l$, so the image of \eqref{eq:ellp1} is $q_{1} q_{2} t = q_{1} t q_{1} = t q_{2} q_{1} = t q_{1} q_{2}$. Thus, according to the formula \eqref{eq:intersection-form-formula}, we have $\eqref{eq:intersection-form-computation-1} = (-1)^3 t q_{1} q_{2} = -tq_{1} q_{2}$, as claimed in Table~\ref{tab:LKB3}.

Next let us compute the intersection
\begin{equation}
\label{eq:intersection-form-computation-2}
\bigl\langle [\sigma_{i}(L(\cdots 011 \cdots))] \mathbin{,} [L(\cdots 020 \cdots)^\partial] \bigr\rangle .
\end{equation}
Figure~\ref{fig:computation2} illustrates the two embedded surfaces $\sigma_{i}(L(\cdots 011 \cdots))$ and $L(\cdots 020 \cdots)^\partial$ and their two intersection points $p$ (the solid dots) and $q$ (the hollow dots). The corresponding two loops may be written as
\[
\ell_p = \theta_{1,i+1} \theta_{2,i+1} \qquad\qquad \ell_q = \theta_{1,i+1} \theta_{2,i+1} \sigma
\]
which have signs $\mathrm{sgn}(\ell_p) = +1$ and $\mathrm{sgn}(\ell_q) = -1$. Taking into account also the local signs, we see from the formula \eqref{eq:intersection-form-formula} that
\begin{align*}
\eqref{eq:intersection-form-computation-2} &= (-1)(+1)(+1)q_{1} q_{2} + (-1)(+1)(-1)q_{1} q_{2} t \\
&= -q_{1} q_{2} + tq_{1} q_{2} \\
&= (t-1)q_{1} q_{2},
\end{align*}
as claimed in Table~\ref{tab:LKB3}.
As a final example, we compute the intersection
\begin{equation}
\label{eq:intersection-form-computation-3}
\bigl\langle [\sigma_{i}(L(\cdots 101 \cdots))] \mathbin{,} [L(\cdots 110 \cdots)^\partial] \bigr\rangle .
\end{equation}
Figure~\ref{fig:computation3} illustrates the relevant embedded surfaces and their unique intersection point $p$. The corresponding loop is $\ell_p = \theta_{2,i+1}$, whose image in $\RB_{2}$ is $q_{2}$ and whose sign is $+1$. The local signs $\mathrm{sgn}(p_{1})$ and $\mathrm{sgn}(p_{2})$ are both $+1$, so we have $\eqref{eq:intersection-form-computation-3} = q_{2}$, as claimed in Table~\ref{tab:LKB3}.

All of the remaining entries in Table~\ref{tab:LKB3} may be verified similarly to these two examples; this completes the proof of Theorem \ref{athm:3-variable-LKB}.
\end{proof}

\begin{figure}[htb]
    \centering

    \includegraphics[scale=0.8]{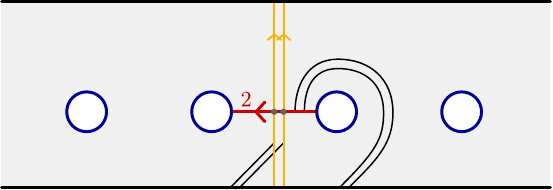}
    \caption{The computation of the intersection $\eqref{eq:intersection-form-computation-1} = -tq_{1} q_{2}$.}
    \label{fig:computation1}
	\vspace{2ex}

    \includegraphics[scale=0.8]{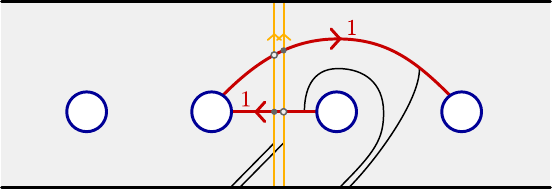}
    \caption{The computation of the intersection $\eqref{eq:intersection-form-computation-2} = (t-1)q_{1} q_{2}$.}
    \label{fig:computation2}
	\vspace{2ex}

    \includegraphics[scale=0.8]{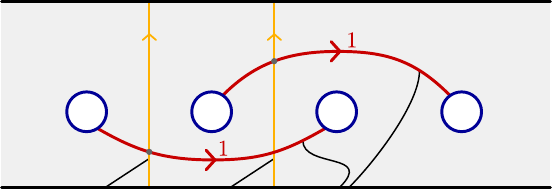}
    \caption{The computation of the intersection $\eqref{eq:intersection-form-computation-3} = q_{2}$.}
    \label{fig:computation3}
\end{figure}

\begin{rmk}
Table~\ref{tab:LKB3} in particular gives formulas for the classical Lawrence-Krammer-Bigelow representation when we set $q = q_{1} = q_{2}$ (which then commutes with $t$). We were not able to find explicit formulas for the Lawrence-Krammer-Bigelow representation in the literature using the basis that we describe above. Formulas, in different bases, may be found in \cite[Th.~4.1]{bigelow2001braid}, \cite[\S 3]{KrammerLK} and \cite[\S 1]{paoluzziparis}, but it is not entirely trivial to pass between the different bases. We also note that the $3 \times 3$ block matrices in Table~\ref{tab:LKB3} are, for obvious reasons, the same as the usual matrices of the Burau representation. Analogous formulas over a \emph{non-commutative ground ring} (as is the case of our formulas) have been computed in the context of mapping class groups in \cite[\S 7]{BlanchetPalmerShaukat}.
\end{rmk}

\begin{rmk}
The ribbon-LKB representation $V_\R(2)$ is the representation over $\bZ[Q_{\infty}]$ associated to the first row ($G=\B_{2,n}$) of Table~\ref{table:pro-nilpotent-representations}, where in this case $Q_{\infty} = \RB_{2}$. This is generalised by the second row of Table~\ref{table:pro-nilpotent-representations}, where $G=\B_{2,\bk,n}$ for any tuple $\bk$ of positive integers (the first row corresponds to the empty tuple). In this more general setting we also have an explicit description of the group $Q_{\infty}$: see Proposition~\ref{prop:residually-nilpotent-quotient} at the end of the next section. With sufficient patience, one could generalise the matrices of Table~\ref{tab:LKB3} to obtain matrices for the $\B_{n}$-representation over $\bZ[Q_{\infty}]$ associated to $G=\B_{2,\bk,n}$, with $Q_{\infty}$ as in Proposition~\ref{prop:residually-nilpotent-quotient}.
\end{rmk}

\paragraph*{Faithfulness.}

By construction, we have a tower of surjections of $\B_{n}$-representations
\begin{equation}
\label{eq:tower-of-surjections}
V_\R(2) \longtwoheadrightarrow \cdots \longtwoheadrightarrow V_{r}(2) \longtwoheadrightarrow V_{r-1}(2) \longtwoheadrightarrow \cdots \longtwoheadrightarrow V_{2}(2),
\end{equation}
where $V_{\bullet}(2)$ is the pro-nilpotent representation of $\B_{n}$ over $Q_{\bullet}$ from Theorem \ref{athm:pro-nilpotent} and $V_\R(2)$ is the ribbon-LKB representation of $\B_{n}$ over $\bZ[\RB_{2}] = \bZ[Q_{\infty}]$ from Theorems~\ref{athm:ribbon-Lawrence} and \ref{athm:3-variable-LKB}.

\begin{prop}
\label{prop:faithfulness}
Each representation $V_\R(2)$ and $V_{r}(2)$ in \eqref{eq:tower-of-surjections} is faithful.
\end{prop}
\begin{proof}
As modules, these representations are all free of the same (finite) rank over their respective ground rings and the surjections in \eqref{eq:tower-of-surjections} are induced by surjections of the ground rings. Since $V_{2}(2)$ is faithful, by \cite{bigelow2001braid,KrammerLK}, it follows that all of the other representations in the tower \eqref{eq:tower-of-surjections} are also faithful.
\end{proof}

\subsection{An embedding into a matrix ring}
\label{ss:embedding-into-matrix-ring}

We first recall a basic fact about embedding Laurent polynomial rings into the field of complex numbers.

\begin{fact}\label{fact:injection_algebraically_independent}
A ring homomorphism $\bZ[\bZ^{n}]\to \bC$ is injective if and only if the image $\{ z_{1},\ldots,z_{n} \}$ of a free generating set of $\bZ^{n}$ is algebraically independent.
\end{fact}
\begin{proof}
By definition, the subset $\{ z_{1},\ldots,z_{n} \} \subset \bC$ is algebraically independent if and only if the associated ring homomorphism $\bZ[\bN^n] \to \bC$ is injective. We therefore just have to show that this holds if and only if its unique extension to $\bZ[\bZ^n] \to \bC$ is injective. One direction is trivial; for the other, suppose that $f \in \bZ[\bZ^n] = \bZ[t_{1}^{\pm 1},\ldots,t_{n}^{\pm 1}]$ is a non-zero polynomial with $f(z_{1},\ldots,z_{n}) = 0$. For $1\leq i\leq n$ denote by $e_{i} \in \bZ$ the lowest exponent of $t_{i}$ appearing in a monomial of $f$. We then have $f(z_{1},\ldots,z_{n})z_{1}^{-e_{1}}\cdots z_{n}^{-e_{n}} = 0$, hence $ft_{1}^{-e_{1}}\cdots t_{n}^{-e_{n}} \in \bZ[\bN^n] = \bZ[t_{1},\ldots,t_{n}]$ is a non-zero polynomial in the kernel of $\bZ[\bN^n] \to \bC$.
\end{proof}

Let $H$ denote the subgroup of $\RB_{2}$ generated by $q_{1}$, $q_{2}$ and $t^{2}$.
\begin{rmk}
The group $H$ is free abelian on $\{q_{1},q_{2},t^2\}$ and is an index-$2$ subgroup of $\RB_{2}$.
\end{rmk}

Let $\Psi$ denote the ring homomorphism $\Theta = \bZ[\RB_{2}] \to \mathrm{Mat}_{2}(\bC)$ given on generators by \eqref{eq:embedding-into-matrix-ring}.
\begin{lem}\label{lem:restriction_Psi_injective}
Under the hypothesis of Proposition \ref{aprop:embedding-into-matrix-ring}, the restriction of $\Psi$ to $\bZ[H]$ is injective.
\end{lem}

\begin{proof}
We first observe that $\Psi(t^{2})=yz \cdot \mathrm{Id}$. Hence the image of $\bZ[H]$ under $\Psi$ is contained in the subring of diagonal matrices of $\mathrm{Mat}_{2}(\bC)$, which is isomorphic to $\bC\times \bC$.
Denote by $\bar{\Psi} \colon \bZ[H] \to \bC$ the composition of $\Psi|_{\bZ[H]} \colon \bZ[H] \to \bC \times \bC \subset \mathrm{Mat}_{2}(\bC)$ with the projection $\mathrm{pr}_{1} \colon \bC \times \bC \twoheadrightarrow \bC$ onto the first diagonal entry.
The three free abelian generators $q_{1}$, $q_{2}$ and $t^2$ of $H$ are sent under $\bar{\Psi}$ to the complex numbers $w$, $x$ and $yz$ respectively. By hypothesis, these are algebraically independent, so $\bar{\Psi}$ is injective by Fact~\ref{fact:injection_algebraically_independent}, and hence $\Psi|_{\bZ[H]}$ is injective.
\end{proof}

\begin{proof}[Proof of Proposition~\ref{aprop:embedding-into-matrix-ring}, i.e.~that $\Psi$ is injective]
Let $\kappa$ be an element of $\ker(\Psi)$. We may decompose $\kappa$ as $\lambda t + \mu$ where $\lambda,\mu \in \bZ[H]$. We recall from the proof of Lemma~\ref{lem:restriction_Psi_injective} that $\Psi(\lambda)$ and $\Psi(\mu)$ are diagonal matrices that we write with $\alpha,\beta,\gamma,\delta\in \bC$
\begin{equation}
\Psi(\lambda) = \left(\begin{matrix} \alpha & 0 \\ 0 & \beta \end{matrix}\right)
\qquad \mathrm{and} \qquad
\Psi(\mu) = \left(\begin{matrix} \gamma & 0 \\ 0 & \delta \end{matrix}\right).
\end{equation}
Hence, by computing the matrix of $\Psi(\kappa)$, we have
\begin{equation}
\Psi(\kappa) = \left(\begin{matrix} \gamma & \alpha y \\ \beta z & \delta \end{matrix}\right)=0.
\end{equation}
Since $y$ and $z$ are non-zero, this implies that $\Psi(\lambda) = \Psi(\mu) = 0$. We then deduce from Lemma~\ref{lem:restriction_Psi_injective} that $\lambda=\mu=0$, whence the result.
\end{proof}

\section{Surface braid groups}
\label{s:SBn}

In this section, we construct the pro-nilpotent representations of $\B_{n}$ and of $\B_{n}(S)$ listed in Table~\ref{table:pro-nilpotent-representations}, assuming throughout that $n\geq 3$. In each case (i.e.~row of Table~\ref{table:pro-nilpotent-representations}), the input is a split fibration sequence
\begin{equation}
\label{eq:input-fibration-sequence-repeated}
\begin{tikzcd}
X \ar[r,"i"] & Y \ar[r,"f"] & Z, \ar[l,bend right=40,dashed]
\end{tikzcd}
\end{equation}
whose induced split short exact sequence of fundamental groups is
\begin{equation}
\label{eq:input-ses-repeated}
\begin{tikzcd}
1 \ar[r] & K \ar[r,"\varphi"] & G \ar[r] & \Gamma \ar[l,bend right=30,dashed] \ar[r] & 1
\end{tikzcd}
\end{equation}
(i.e.~inputs \eqref{input:ses} and \eqref{input:fibration} from the beginning of \S\ref{s:general-recipe}), with $\Gamma$ and $G$ as in the given row of Table~\ref{table:pro-nilpotent-representations}. In each setting that we consider in this section, the space $X$ is a configuration space of the form $C_{\bk}(M\smallsetminus A)$ as in Corollary~\ref{coro:BM-homology-lemma}. Thus, by Corollaries~\ref{coro:general-recipe-pro-nilpotent} and \ref{coro:BM-homology-lemma}, the split fibration sequence \eqref{eq:input-fibration-sequence-repeated} induces a (genuine) pro-nilpotent representation of $\Gamma$ via the construction of \S\ref{s:general-recipe} as long as the inclusion $\varphi \colon K \hookrightarrow G$ in \eqref{eq:input-ses-repeated} is eNCP. Recall from Corollary~\ref{coro:lifting-eNCP} and Remark~\ref{rmk:lifting-eNCP} that a sufficient criterion to prove this is to find a quotient $G'$ of $G = K \rtimes \Gamma$ that is surjective when restricted to $K$ and zero when restricted to $\Gamma$:
\begin{equation}
\label{eq:ses-with-quotient-repeated}
\begin{tikzcd}
1 \ar[r] & K \ar[r] \ar[dr,two heads] & G \ar[r] \ar[d,two heads] & \Gamma \ar[l,bend right=30,dashed] \ar[r] \ar[dl,dashed,"0"] & 1 \\
&& G' &&
\end{tikzcd}
\end{equation}
and such that the lower central series of $G'$ does not stop. We will find such a quotient $G'$ in each case using the results of \cite{DPS}.

\begin{proof}[Proof of Theorem \ref{athm:other-pro-nilpotent-reps} for classical and surface braid groups]
We begin with the case of $G=\B_{2,\bk,n}$, namely the \uline{second row of Table~\ref{table:pro-nilpotent-representations}}, which generalises Theorem \ref{athm:pro-nilpotent} (which corresponds to the first row of Table~\ref{table:pro-nilpotent-representations}). Let $\bk = (k_{1},\ldots,k_{l})$ be a tuple of integers with $k_{i} \geq 3$ for each $1\leq i\leq l$. There is a split fibration sequence
\begin{equation}
\label{eq:input-fibration-sequence-B2kn}
\begin{tikzcd}
C_{2,\bk}(\bD_{n}) \ar[r,"i"] & C_{2,\bk,n}(\bD^2) \ar[r,"f"] & C_{n}(\bD^2), \ar[l,bend right=25,dashed]
\end{tikzcd}
\end{equation}
given by forgetting all blocks of points except for the last block of size $n$; the section is given by pushing in a new block of $n$ points near the boundary. The fundamental group of the middle space is $G = \B_{2,\bk,n}$. By \cite[Prop.~3.12]{DPS}, we have
\begin{equation}
\label{eq:identification-quotient-by-Gamma-infty}
\B_{2,\bk,n}/\LCS_{\infty} \cong \bZ^{\binom{l+2}{2}} \times (\bZ^{2(l+1)} \rtimes \bZ),
\end{equation}
where the generators of the last $\bZ$ factor act on the $\bZ^{2(l+1)}$ factor by swapping its coordinates in $l+1$ pairs. We may then quotient further onto $\bZ^{2(l+1)} \rtimes \bZ$ by killing the $\bZ^{\binom{l+2}{2}}$ factor and projecting the $\bZ$ factor onto $\bZ/2 = \mathfrak{S}_{2}$, since it acts by involutions. Finally, we may quotient $\bZ^{2(l+1)}$ onto $\bZ^2$ by sending half of the generators to $(1,0)$ and the other half to $(0,1)$ respecting the involution by which $\mathfrak{S}_{2}$ acts. The result is a quotient onto $G' = \bZ^2 \rtimes \mathfrak{S}_{2}$. By \cite[Prop.~A.28]{DPS}, its lower central series is given by $\LCS_{i}(G') = 2^{i-2}(\delta\bZ)$ for $i\geq 2$, with $\delta\bZ := \{(x,-x) \mid x \in \bZ \} \leq \bZ^{2}$. Hence $\LCS_{i}(G')\neq \LCS_{i+1}(G')$ for all $i\geq 2$ and so the lower central series $\LCS_{*}(G')$ does not stop. It remains to check that this fits into a diagram of the form \eqref{eq:ses-with-quotient-repeated}, i.e.:
\begin{itemizeb}
\item that $\B_{2,\bk}(\bD_{n})$ surjects onto $G'$ and
\item that the homomorphism $\B_{n} \dashrightarrow \B_{2,\bk,n} \to G'$ is zero.
\end{itemizeb}
From the proof of \cite[Prop.~3.12]{DPS} one sees that the standard generator of $\B_{2,\bk}(\bD_{n}) \subset \B_{2,\bk,n}$ that swaps the two points in the first block of the partition is sent to the generator of $\mathfrak{S}_{2} \subset G'$. Similarly, one sees that the standard generator of $\B_{2,\bk}(\bD_{n}) \subset \B_{2,\bk,n}$ that fixes all points except for the first one (in the first block of $2$), which loops once around one of the $n$ punctures, is sent to $(1,0) \in \bZ^2 \subset G'$. These two elements generate $G'$, so we have established the first claim above. For the second claim, one sees from the proof of \cite[Prop.~3.12]{DPS} that, under the identification \eqref{eq:identification-quotient-by-Gamma-infty}, each standard generator of $\B_{n}$ is sent to one of the copies of $\bZ$ in the $\bZ^{\binom{l+2}{2}}$ factor, and hence to zero in $G'$. By Corollary~\ref{coro:lifting-eNCP} (and Remark~\ref{rmk:lifting-eNCP}) it follows that the inclusion $\varphi \colon \B_{2,\bk}(\bD_{n}) \hookrightarrow \B_{2,\bk,n}$ is eNCP. Thus by Corollary~\ref{coro:general-recipe-pro-nilpotent} (and Corollary~\ref{coro:BM-homology-lemma}) we obtain from \eqref{eq:input-fibration-sequence-B2kn} a (genuine) pro-nilpotent representation of $\B_{n}$.

\uline{All of the other rows of Table~\ref{table:pro-nilpotent-representations}} concerning classical or surface braid groups may be proven in the same way. Recall that $S$ is a surface with non-empty boundary (but it may have infinite type; no additional complexity arises if we allow this). Let $\lambda$ be a tuple of positive integers and consider the split fibration sequence
\begin{equation}
\label{eq:input-fibration-sequence-surface-braids}
\begin{tikzcd}
C_{\lambda}(S_{n}) \ar[r,"i"] & C_{\lambda,n}(S) \ar[r,"f"] & C_{n}(S), \ar[l,bend right=25,dashed]
\end{tikzcd}
\end{equation}
given by forgetting all blocks of points except for the last block of size $n$. Here $S_{n}$ denotes $S$ minus $n$ interior points and the section is given by pushing in a new block of $n$ points near the boundary. We may then consider the following diagram:
\begin{equation}
\label{eq:ses-with-quotient-repeated-again}
\begin{tikzcd}
1 \ar[r] & \B_{\lambda}(S_{n}) \ar[r] \ar[dr,two heads] & \B_{\lambda,n}(S) \ar[r] \ar[d,two heads] & \B_{n}(S) \ar[l,bend right=30,dashed] \ar[r] \ar[dl,dashed,"0"] & 1 \\
&& \B_{\lambda}(S), &&
\end{tikzcd}
\end{equation}
where the vertical quotient is given by forgetting the last block of strands of size $n$. Notice that the left-hand diagonal map is clearly surjective: without loss of generality we may assume that the $n$ punctures are contained in a collar neighbourhood of $S$ and then every $\lambda$-braid on $S$ may be lifted to a $\lambda$-braid on $S_{n}$ by pushing it away from this collar neighbourhood. Also, the right-hand diagonal map is obviously zero. Thus, by the discussion at the beginning of this section, it suffices to check that the lower central series of the group $\B_{\lambda}(S)$ does not stop. This is the case for:
\begin{itemizeb}
\item $\B_{1,1,1,\bk}(\bD^2)$, by \cite[Lem.~3.8]{DPS};
\item $\B_{2,2,\bk}(\bD^2)$, by \cite[Cor.~3.15]{DPS};
\item $\B_{1,2,\bk}(\bD^2)$, by \cite[Cor.~3.18]{DPS};
\item $\B_{2,\bk}(S)$ for $S \neq \bD^2$, by \cite[Prop.~6.62]{DPS};
\item $\B_{1,\bk}(S)$ for $S \notin \{ \bD^2 , \text{Ann} , \text{M{\"o}b} \}$, by \cite[Prop.~6.62]{DPS};
\item $\B_{1,\bk}(\text{M{\"o}b})$ for $\bk \neq \varnothing$, by \cite[Prop.~6.62, Corollary~6.67 and Proposition~6.68]{DPS};
\item $\B_{1,1,\bk}(\text{Ann})$, by \cite[Lem.~6.63]{DPS}.
\end{itemizeb}
We note that earlier results on the stopping or non-stopping of the lower central series of (pure) surface braid groups were obtained in \cite{BellingeriGervaisGuaschi,GoncalvesGuaschi2009,BellingeriGervais2016}.
These cases correspond precisely to the rows in Table~\ref{table:pro-nilpotent-representations} concerning classical or surface braid groups (except for the top two rows, which were dealt with above). This completes the proof of Theorem \ref{athm:other-pro-nilpotent-reps} in these cases.
\end{proof}

To finish this section, we describe, for each row of Table~\ref{table:pro-nilpotent-representations} concerning classical or surface braid groups, the ground ring of the bottom ($r=2$) layer of the pro-nilpotent representation that we have constructed. This amounts to calculating the abelian group $A = Q_{2}$, since the ground ring of the bottom layer is $\bZ[Q_{2}]$. By construction (see diagram \eqref{eq:big-diagram-of-quotients}), this is the kernel of the (split) surjection $G^{\ab} \twoheadrightarrow \Gamma^{\ab}$ induced by the given (split) surjection $G \twoheadrightarrow \Gamma$.

\begin{prop}
\label{prop:abelian-quotients-1}
In the first \textup{(}$\B_{n}$\textup{)} block of Table~\ref{table:pro-nilpotent-representations}, the group $A$ is free abelian of rank
\[
\begin{cases}
\binom{l+2}{2}+l+1 & \text{when } G = \B_{2,\bk,n} \text{ \textup{(}with each $k_{i} \geq 3$\textup{)}} \\
\binom{l+4}{2}+l' & \text{when } G = \B_{1,1,1,\bk,n} \\
\binom{l+3}{2}+l'+2 & \text{when } G = \B_{2,2,\bk,n} \\
\binom{l+3}{2}+l'+1 & \text{when } G = \B_{1,2,\bk,n} \\
\end{cases}
\]
where $l$ is the number of blocks of $\bk$ and $l'$ is the number of blocks of $\bk$ of size at least $2$.
\end{prop}
\begin{proof}
In each case, $\Gamma$ is $\B_{n}$, whose abelianisation is $\bZ$, so $A$ is $G^{\ab}$ minus one $\bZ$ summand. The abelianisation $G^{\ab}$ in each of the four cases may be read off from \cite[Prop.~3.5]{DPS}.
\end{proof}

\begin{prop}
\label{prop:abelian-quotients-2}
In the second \textup{(}$\B_{n}(S)$\textup{)} block of Table~\ref{table:pro-nilpotent-representations}, the group $A$ is isomorphic to
\[
\begin{cases}
H_{1}(S)^{l+1} \times \bZ^{l'+1} \times \bZ^{\binom{l+2}{2}} & \text{when } G = \B_{2,\bk,n}(S) \text{ with } S \text{ planar} \\
H_{1}(S)^{l+1} \times (\bZ/2)^{l'+1} & \text{when } G = \B_{2,\bk,n}(S) \text{ with } S \text{ non-planar} \\
H_{1}(S)^{l+1} \times \bZ^{l'} \times \bZ^{\binom{l+2}{2}} & \text{when } G = \B_{1,\bk,n}(S) \text{ with } S \text{ planar} \\
H_{1}(S)^{l+1} \times (\bZ/2)^{l'} & \text{when } G = \B_{1,\bk,n}(S) \text{ with } S \text{ non-planar} \\
\bZ^{l+1} \times (\bZ/2)^{l'} & \text{when } G = \B_{1,\bk,n}(\text{\textup{M{\"o}b}}) \\
\bZ^{l+2} \times \bZ^{l'} \times \bZ^{\binom{l+3}{2}} & \text{when } G = \B_{1,1,\bk,n}(\text{\textup{Ann}})
\end{cases}
\]
where $l$ is the number of blocks of $\bk$ and $l'$ is the number of blocks of $\bk$ of size at least $2$.
\end{prop}
\begin{proof}
In each case, we have $G^{\ab} \cong A \oplus \Gamma^{\ab}$, so we may compute $A$ from the abelianisations of $\Gamma = \B_{n}(S)$ and of $G = \B_{\lambda,n}(S)$, which are computed explicitly in \cite[Prop.~6.32]{DPS} and \cite[Prop.~6.47]{DPS} respectively. Precisely, if $t$ denotes the number of blocks of $\lambda$ and $t'$ denotes the number of blocks of $\lambda$ of size at least $2$, the split surjection $G^{\ab} \twoheadrightarrow \Gamma^{\ab}$ is
\[
H_{1}(S)^{t+1} \times \bZ^{t'+1} \times \bZ^{\binom{t+1}{2}} \longtwoheadrightarrow H_{1}(S) \times \bZ
\]
when $S$ is planar and
\[
H_{1}(S)^{t+1} \times (\bZ/2)^{t'+1} \longtwoheadrightarrow H_{1}(S) \times \bZ/2
\]
when $S$ is non-planar. The specific computations of the proposition follow from these computations, specialising $\lambda$ and $S$ as appropriate.
\end{proof}

For the first two rows of the table, we may also compute the ground ring of the limit of the pro-nilpotent tower of representations. As we will explain, this essentially amounts to computing the residually nilpotent group $Q_{\infty}$. By construction (see diagram \eqref{eq:big-diagram-of-quotients}), this is the kernel of the (split) surjection $G/\LCS_{\infty} \twoheadrightarrow \Gamma/\LCS_{\infty}$ induced by the given (split) surjection $G \twoheadrightarrow \Gamma$. In the example in question, we have $G = \B_{2,\bk,n}$ and $\Gamma = \B_{n}$, for which
\[
\B_{2,\bk,n}/\LCS_{\infty} \cong \bZ^{\binom{l+2}{2}} \times (\bZ^{2(l+1)} \rtimes \bZ) \qquad\text{and}\qquad \B_{n}/\LCS_{\infty} = \B_{n}^{\ab} \cong \bZ ,
\]
by \cite[Prop.~3.12 and Example~2.3]{DPS}. Thus we have:
\begin{equation}
\label{eq:input-ses-B2kn-level-infinity}
\begin{tikzcd}
1 \ar[r] & \B_{2,\bk}(\bD_{n}) \ar[d,two heads] \ar[r] & \B_{2,\bk,n} \ar[d,two heads] \ar[r] & \B_{n} \ar[d,two heads] \ar[l,bend right=25,dashed] \ar[r] & 1 \\
1 \ar[r] & Q_{\infty} \ar[r] & \bZ^{\binom{l+2}{2}} \times (\bZ^{2(l+1)} \rtimes \bZ) \ar[r] & \bZ \ar[l,bend right=30,dashed] \ar[r] & 1.
\end{tikzcd}
\end{equation}
From the explicit presentation of $\B_{2,\bk,n}/\LCS_{\infty}$ given in the proof of \cite[Prop.~3.12]{DPS}, we see that the split surjection on the bottom row of \eqref{eq:input-ses-B2kn-level-infinity} is the projection onto one of the copies of $\bZ$ in the direct $\bZ^{\binom{l+2}{2}}$ factor. As a consequence, just as we observed at the beginning of \S\ref{ss:Bn-ribbon}, the induced $\B_{n}$-action on $Q_{\infty}$ is trivial. This means, by definition, that $Q^{\mathrm{u}}_{\infty} = Q_{\infty}$ and thus $Q^{\mathrm{u}}_{\bullet} = Q_{\bullet}$. The $\B_{n}$-representation over $\bZ[Q_{\infty}]$ associated to the projection $\B_{2,\bk}(\bD_{n}) \twoheadrightarrow Q_{\infty}$ may be thought of as the limit of the pro-nilpotent tower of representations associated to the tower $Q_{\bullet}$. This is explained more precisely in Lemma~\ref{lem:lifting-to-limit}, but note that the situation in this case is simpler, due to the fact that $Q^{\mathrm{u}}_{\infty} = Q_{\infty}$ and $Q^{\mathrm{u}}_{\bullet} = Q_{\bullet}$. To summarise:

\begin{prop}
\label{prop:residually-nilpotent-quotient}
The inverse limit of the pro-nilpotent representation of $\B_{n}$ corresponding to the second row of Table~\ref{table:pro-nilpotent-representations} is defined over the group ring $\bZ[Q_{\infty}]$, for
\[
Q_{\infty} = \bZ^{\binom{l+2}{2} - 1} \times (\bZ^{2(l+1)} \rtimes \bZ),
\]
where $1 \in \bZ$ acts on $\bZ^{2(l+1)}$ by swapping its coordinates in $l+1$ pairs.
\end{prop}

In the special case $\bk = \varnothing$, corresponding to the first row of Table~\ref{table:pro-nilpotent-representations} (and Theorem \ref{athm:pro-nilpotent}), we have $l=0$ and so $Q_{\infty} = \bZ^2 \rtimes \bZ = \RB_{2}$, as we already observed in this particular case at the beginning of \S\ref{ss:Bn-ribbon}.

\begin{coro}
In the second row of Table~\ref{table:pro-nilpotent-representations}, for $r\geq 2$ we have
\[
Q_{r} = \bZ^{\binom{l+2}{2} - 1} \times ((\bZ^2/2^{r-2}\bar{\Delta})^{l+1} \rtimes \bZ),
\]
where $\bar{\Delta} = (1,-1) \in \bZ^2$ and $1 \in \bZ$ acts on each copy of $\bZ^2/2^{r-2}\bar{\Delta}$ by swapping coordinates.
\end{coro}
\begin{proof}
To obtain $Q_{r}$, we may start with the bottom row of \eqref{eq:input-ses-B2kn-level-infinity}, quotient the middle and right-hand groups by $\LCS_{r}$ and then take the kernel. This uses the fact that quotienting by $\LCS_{\infty}$ and then by $\LCS_{r}$ is the same as simply quotienting by $\LCS_{r}$. Clearly quotienting by $\LCS_{r}$ does not affect the right-hand group $\bZ$ or the direct $\bZ^{\binom{l+2}{2}}$ factor in the middle group. We therefore just have to show that $(\bZ^{2(l+1)} \rtimes \bZ)/\LCS_{r} \cong (\bZ^2/2^{r-2}\bar{\Delta})^{l+1} \rtimes \bZ$. This follows from \cite[Prop.~A.10]{DPS}.
\end{proof}

\section{Loop braid groups}
\label{s:LBn}

For $n\geq 1$ we may consider the configuration space $C_{nS^1}(\bD^3)$ of $n$-component unlinks in the closed $3$-ball $\bD^3$. Formally this is constructed from the space $\mathrm{Emb}(nS^1,\bD^3)$ of smooth embeddings of $nS^1 = \bigsqcup_{n} S^1$ into $\bD^3$ with the Whitney topology by restricting to the path-component containing the standard unlink and taking the quotient by the action of $\mathrm{Diff}(nS^1)$:
\[
C_{nS^1}(\bD^3) = \mathrm{Emb}^{\mathrm{unl}}(nS^1,\bD^3) / \mathrm{Diff}(nS^1).
\]
By the main theorem of Brendle and Hatcher~\cite[Th.~1]{BrendleHatcher}, this space deformation retracts onto the subspace consisting of all embedded unlinks each of whose components are dilations, rotations and translations of the standard embedded circle in $\bD^3$. This subspace has the advantage of being a finite-dimensional manifold -- hence in particular locally compact -- allowing one to apply Borel-Moore homology to it and its variants. We will implicitly make this replacement whenever we apply Borel-Moore homology. The fundamental group of this space is the $n$-th \emph{extended loop braid group}:
\[
\exwB_{n} = \pi_{1}(C_{nS^1}(\bD^3)).
\]
If we quotient only by the orientation-preserving diffeomorphisms of $nS^1$, we obtain the space of \emph{oriented} $n$-component unlinks in $\bD^3$, whose fundamental group in the $n$-th \emph{loop braid group}:
\[
C_{nS^1}^+(\bD^3) = \mathrm{Emb}^{\mathrm{unl}}(nS^1,\bD^3) / \mathrm{Diff}^+(nS^1) \qquad\qquad \wB_{n} = \pi_{1}(C_{nS^1}^+(\bD^3)).
\]
The space $C_{nS^1}^+(\bD^3)$ is a $2^n$-fold covering of $C_{nS^1}(\bD^3)$, so $\wB_{n}$ is an index-$2^n$ subgroup of $\exwB_{n}$. In fact, there is a split projection $\exwB_{n} \twoheadrightarrow (\bZ/2)^n$ whose kernel is $\wB_{n}$, given by recording for each component of the base configuration whether the given loop of configurations preserves or reverses the orientation of that component. More generally, we may consider the space
\[
\C(n_P,n_{S_+},n_S) = \mathrm{Emb}^{\mathrm{unl}}(n_P \sqcup (n_{S_+} + n_S)S^1 , \bD^3) / \bigl( \mathfrak{S}_{n_P} \times \mathrm{Diff}^+(n_{S_+}S^1) \times \mathrm{Diff}(n_S S^1) \bigr)
\]
of configurations of $n_P$ points, $n_{S_+}$ oriented circles and $n_S$ unoriented circles (forming an unlink) in $\bD^3$, whose fundamental group is by definition the \emph{tripartite loop braid group} $\wB(n_P,n_{S_+},n_S)$. As special cases, we have
\begin{equation}
\label{eq:tripartite-welded-special-cases}
\begin{aligned}
\wB(n,0,0) &= \mathfrak{S}_{n}, \\
\wB(0,n,0) &= \wB_{n}, \\
\wB(0,0,n) &= \exwB_{n}.
\end{aligned}
\end{equation}
Finally, if $\lambda_P,\lambda_{S_+},\lambda_S$ are partitions of $n_P,n_{S_+},n_S$ respectively, we may also consider the subgroup $\wB(\lambda_P,\lambda_{S_+},\lambda_S) \subseteq \wB(n_P,n_{S_+},n_S)$ of those loops whose induced permutation of the base configuration preserves the given partition. This is the fundamental group of the corresponding covering space $\C(\lambda_P,\lambda_{S_+},\lambda_S)$ of $\C(n_P,n_{S_+},n_S)$. For more details, see \cite[\S 4.4--\S 4.6 and \S 5]{DPS}.

\paragraph*{Generators.}
An explicit generating set for $\wB(\lambda_P,\lambda_{S_+},\lambda_S)$ is given in \cite[Lem.~5.9]{DPS}, which we illustrate in Figure~\ref{fig:partitioned-tripartite-welded-braid-group-generators}. The generators $\tau_\alpha$ and $\sigma_\alpha$ involve two points or two circles that are both in the same block of the partition $\lambda = \lambda_P \lambda_{S_+} \lambda_S$, whereas the generators $\chi_{\alpha\beta}$ involve either two circles that are in different blocks of $\lambda$ or one point and one circle (which are therefore necessarily in different blocks of $\lambda$). Assuming that the base configuration consists of $n_P$ points and $n_{S_+} + n_S$ circles arranged linearly on the $xy$-plane, these generators have the following descriptions (where $n = n_P + n_{S_+} + n_S$):
\begin{itemizeb}
\item $\tau_\alpha$ (for $1\leq \alpha \leq n-1$ such that $\alpha,\alpha+1$ are in the same block of $\lambda$) interchanges the $\alpha$-th and $(\alpha + 1)$-st points or circles without either of them passing through the other;
\item $\sigma_\alpha$ (for $n_P + 1 \leq \alpha \leq n-1$ such that $\alpha,\alpha+1$ are in the same block of $\lambda$) interchanges the $\alpha$-th and $(\alpha + 1)$-st circles while one passes through the other;
\item $\rho_\alpha$ (for $n_P + n_{S_+} + 1 \leq \alpha \leq n$) rotates the $\alpha$-th circle by $180$ degrees, reversing its orientation;
\item $\chi_{\alpha\beta}$ (for $1\leq \alpha \leq n$ and $n_P + 1 \leq \beta \leq n$ such that $\alpha,\beta$ are in different blocks of $\lambda$) sends the $\alpha$-th point or circle in a loop passing through the $\beta$-th circle.
\end{itemizeb}
Specialising as in \eqref{eq:tripartite-welded-special-cases}, we obtain generating sets for $\wB_n$ and $\exwB_n$ involving only the first two (respectively three) families of generators $\tau_\alpha$ and $\sigma_\alpha$ (and $\rho_\alpha$). See Figure~\ref{fig:partitioned-tripartite-welded-braid-group-generators} for pictures.

\begin{figure}[t]
    \centering
    \includegraphics[scale=0.7]{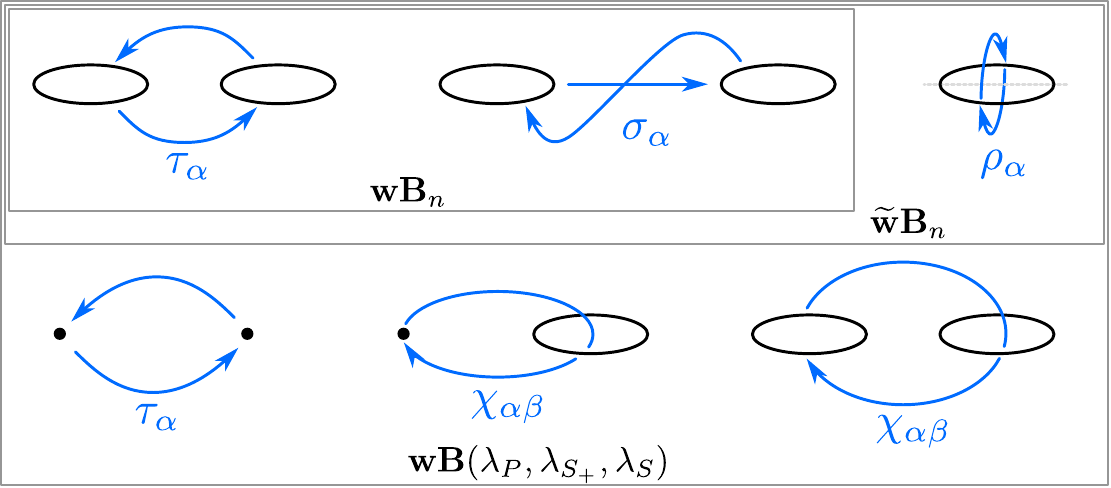}
    \caption{A generating set for the partitioned tripartite welded braid group $\wB(\lambda_P,\lambda_{S_+},\lambda_S)$. The top row generates $\exwB_n$ and the top row without $\rho_\alpha$ generates $\wB_n$.}
    \label{fig:partitioned-tripartite-welded-braid-group-generators}
\end{figure}

\begin{notation}
Loop braid groups (and their variations) are often also known as \emph{welded} braid groups, and the corresponding notation $\LB$ is synonymous with $\wB$. We adopt the slightly dissonant convention of using the name \emph{loop braid groups}, since we use their interpretation as loops of configurations of loops in an essential way, but the notation $\wB$ instead of $\LB$, since this is the notation used in the key reference \cite{DPS} for our proofs. We also refer the reader to \cite{Damianijourney} for a complete introduction to the equivalence between the notions of \emph{loop} braid groups and of \emph{welded} braid groups.
\end{notation}

In this section, we construct the weakly pro-nilpotent representations of the loop braid groups $\wB_{n}$ and extended loop braid groups $\exwB_{n}$ listed in Table~\ref{table:pro-nilpotent-representations}. As in \S\ref{s:SBn}, the input in each case is a split fibration sequence \eqref{eq:input-fibration-sequence-repeated} whose induced split short exact sequence of fundamental groups is the top row of diagram \eqref{eq:ses-with-quotient-repeated-again-again} below. By Lemma~\ref{lem:general-recipe-weakly-pro-nilpotent}, this induces a weakly pro-nilpotent representation of $\Gamma$ via the construction of \S\ref{s:general-recipe} as long as the inclusion $\varphi \colon K \hookrightarrow G$ is eNCP. By Corollary~\ref{coro:lifting-eNCP} and Remark~\ref{rmk:lifting-eNCP}, a sufficient criterion for this is the existence of a quotient $G'$ of $G = K \rtimes \Gamma$ that is surjective when restricted to $K$ and zero when restricted to $\Gamma$:
\begin{equation}
\label{eq:ses-with-quotient-repeated-again-again}
\begin{tikzcd}
1 \ar[r] & K \ar[r,"\varphi"] \ar[dr,two heads] & G \ar[r] \ar[d,two heads] & \Gamma \ar[l,bend right=30,dashed] \ar[r] \ar[dl,dashed,"0"] & 1 \\
&& G' &&
\end{tikzcd}
\end{equation}
and such that the lower central series of $G'$ does not stop. As in \S\ref{s:SBn}, we will find such a quotient $G'$ in each case using the results of \cite{DPS}.

\begin{proof}[Proof of Theorem \ref{athm:other-pro-nilpotent-reps} for loop braid groups]
In each case, we have $G = \wB(\lambda_P,\lambda_{S_+},\lambda_S)$, with the triple of partitions $(\lambda_P,\lambda_{S_+},\lambda_S)$ containing a certain triple of sub-partitions as specified in Table~\ref{table:pro-nilpotent-representations}.

For the \uline{$\wB_{n}$ block of Table~\ref{table:pro-nilpotent-representations}}, we have $\lambda_{S_+} = (n,\mu_{S_+})$ and the split fibration sequence \eqref{eq:input-fibration-sequence-repeated} in each case will be
\[
\begin{tikzcd}
X \ar[r,"i"] & \C(\lambda_P,\lambda_{S_+},\lambda_S) \ar[r,"f"] & \C(\varnothing,n,\varnothing), \ar[l,bend right=25,dashed]
\end{tikzcd}
\]
where $f$ forgets all blocks of strands except for a block of size $n$ in $\lambda_{S_+}$. We choose not to introduce additional notation for the fibre $X$ of $f$, since it will not be needed. The induced split short exact sequence of fundamental groups is then
\[
\begin{tikzcd}
1 \ar[r] & K=\pi_{1}(X) \ar[r,"\varphi"] & G = \wB(\lambda_P,\lambda_{S_+},\lambda_S) \ar[r] & \Gamma = \wB_{n} \ar[l,bend right=25,dashed] \ar[r] & 1.
\end{tikzcd}
\]

We first consider the \uline{$5$-th row of the $\wB_{n}$ block of Table~\ref{table:pro-nilpotent-representations}}, where we assume that $\lambda_{S_+}$ contains a block of size $n$ and $\lambda_S$ contains a block of size $1$. Under this assumption, there is a quotient map $G \twoheadrightarrow \wB(\varnothing,n,1)$ given by forgetting all blocks of strands except for these two. A generating set for the group $\wB(\varnothing,n,1)$ is described in \cite[Lem.~5.9]{DPS} (see the beginning of this section and Figure~\ref{fig:partitioned-tripartite-welded-braid-group-generators}) and consists of:
\begin{itemizeb}
\item $n-1$ elements $\tau_\alpha$ for $2\leq \alpha\leq n$ that swap two circles in the block of size $n$ without either passing through the other,
\item $n-1$ elements $\sigma_\alpha$ for $2\leq \alpha\leq n$ that swap two circles in the block of size $n$ while one passes through the other,
\item an element $\rho_{1}$ that rotates the unoriented circle (in the block of size $1$) by $180$ degrees about an axis lying the plane of the circle,
\item $n$ elements $\chi_{1\alpha}$ for $2\leq \alpha\leq n+1$ where the circle in the block of size $1$ follows a loop that passes once through the $\alpha$-th circle in the block of size $n$,
\item $n$ elements $\chi_{\alpha 1}$ for $2\leq \alpha\leq n+1$ where the circle in the block of size $1$ follows a loop that passes once \emph{around} the $\alpha$-th circle in the block of size $n$ (such that, in another frame of reference, the $\alpha$-th circle passes \emph{through} the circle from the block of size $1$).
\end{itemizeb}
For notational convenience, we have numbered the $n+1$ circles in the base configuration so that the first one corresponds to the block of size $1$. The proof of \cite[Prop.~5.17]{DPS} constructs a quotient map
\[
\wB(\varnothing,n,1) \longtwoheadrightarrow \bZ \rtimes (\bZ/2)
\]
that sends $\chi_{\alpha 1} \mapsto (1,0)$ for each $\alpha$, $\rho_{1} \mapsto (0,1)$ and all other generators to $(0,0)$. The right-hand side is the non-trivial semi-direct product of $\bZ$ with $\bZ/2$, where the generator of $\bZ/2$ acts on $\bZ$ by inversion. Composing these two quotient maps we obtain a surjection $G \twoheadrightarrow G' := \bZ \rtimes (\bZ/2)$. The lower central series of $\bZ \rtimes (\bZ/2)$ does not stop by \cite[Cor.~A.8]{DPS}. The elements $\chi_{\alpha 1}$ and $\rho_{1}$ become trivial under the projection onto $\wB_{n}$, so they lie in $K$. Since their images generate $\bZ \rtimes (\bZ/2)$, it follows that the restriction of $G \twoheadrightarrow \bZ \rtimes (\bZ/2)$ to $K$ is surjective. Finally, we have to check that the restriction of $G \twoheadrightarrow \bZ \rtimes (\bZ/2)$ to $\Gamma$ via the section is trivial. This follows since the image of $\Gamma$ under the section is generated by the elements $\sigma_\alpha$ and $\tau_\alpha$, which are all sent to $(0,0)$ in $\bZ \rtimes (\bZ/2)$.

We next consider the \uline{$4$-th row of the $\wB_{n}$ block of Table~\ref{table:pro-nilpotent-representations}}, where we assume that $\lambda_{S_+}$ contains a block of size $n$ and $\lambda_P$ contains a block of size $2$. Under this assumption, there is a quotient map $G \twoheadrightarrow \wB(2,n,\varnothing)$ given by forgetting all blocks of strands except for these two. By \cite[Lem.~5.9]{DPS} (see also the beginning of this section and Figure~\ref{fig:partitioned-tripartite-welded-braid-group-generators}), a generating set for $\wB(2,n,\varnothing)$ consists of:
\begin{itemizeb}
\item elements $\sigma_\alpha$ and $\tau_\alpha$ for $3\leq \alpha\leq n+1$ as above,
\item an element $\tau_{1}$ that swaps the two points,
\item $n$ elements $\chi_{1\alpha}$ for $3\leq\alpha\leq n+2$ where the first point loops through the $\alpha$-th circle.
\end{itemizeb}
The proof of \cite[Prop.~5.18]{DPS} constructs a quotient map
\[
\wB(2,n,\varnothing) \longtwoheadrightarrow \bZ^2 \rtimes \mathfrak{S}_{2}
\]
that sends $\chi_{1\alpha} \mapsto (1,0) \in \bZ^2$, $\tau_{1}$ to the generator of $\mathfrak{S}_{2}$ and each $\sigma_\alpha$, $\tau_\alpha$ to the trivial element. Composing the two quotient maps we obtain a surjection $G \twoheadrightarrow G' := \bZ^2 \rtimes \mathfrak{S}_{2}$. The lower central series of $\bZ^2 \rtimes \mathfrak{S}_{2}$ does not stop by \cite[Cor.~A.29]{DPS}. The elements $\chi_{1\alpha}$ and $\tau_{1}$ become trivial under the projection onto $\wB_{n}$, so they lie in $K$. Since their images generate $\bZ^2 \rtimes \mathfrak{S}_{2}$, it follows that the restriction of $G \twoheadrightarrow \bZ^2 \rtimes \mathfrak{S}_{2}$ to $K$ is surjective. Finally, the fact that the elements $\sigma_\alpha$ and $\tau_\alpha$ are sent to the trivial element in $\bZ^2 \rtimes \mathfrak{S}_{2}$ implies that the restriction of $G \twoheadrightarrow \bZ^2 \rtimes \mathfrak{S}_{2}$ to $\Gamma$ via the section is trivial.

We now consider simultaneously the \uline{first three rows of the $\wB_{n}$ block of Table~\ref{table:pro-nilpotent-representations}}. There is a quotient map $G \twoheadrightarrow \wB(\lambda_P,\mu_{S_+},\lambda_S)$ given by forgetting a block of size $n$ in $\lambda_{S_+}$ (remember that we have $\lambda_{S_+} = (n,\mu_{S_+})$ by assumption). Its restriction to $K \subset G$ is surjective since any loop braid in $\wB(\lambda_P,\mu_{S_+},\lambda_S)$ may be lifted to a loop braid in $G = \wB(\lambda_P,\lambda_{S_+},\lambda_S)$ by adding a stationary block of $n$ oriented circles near the boundary of $\bD^3$ and this lift lies in $K \subset G$ since it projects to the trivial element of $\wB_{n}$. Its restriction to $\Gamma$ via the section is trivial since $\Gamma \subset G$ consists of braids that are trivial except on the block of circles that is forgotten under the projection to $\wB(\lambda_P,\mu_{S_+},\lambda_S)$. We therefore only need to verify that the lower central series of $\wB(\lambda_P,\mu_{S_+},\lambda_S)$ does not stop. By assumption, the tuple of partitions $(\lambda_P,\mu_{S_+},\lambda_S)$ contains either $(\varnothing,b,\varnothing)$, $(\varnothing,\varnothing,b)$ or $(\varnothing,\{1,1\},\varnothing)$ as a tuple of sub-partitions, where $b \in \{2,3\}$. Hence the lower central series of $\wB(\lambda_P,\mu_{S_+},\lambda_S)$ does not stop by \cite[Th.~4.47]{DPS}.

For the \uline{$\exwB_{n}$ block of Table~\ref{table:pro-nilpotent-representations}}, we have $\lambda_S = (n,\mu_S)$ and the split fibration sequence \eqref{eq:input-fibration-sequence-repeated} in each case will be
\[
\begin{tikzcd}
X \ar[r,"i"] & \C(\lambda_P,\lambda_{S_+},\lambda_S) \ar[r,"f"] & \C(\varnothing,\varnothing,n), \ar[l,bend right=25,dashed]
\end{tikzcd}
\]
where $f$ forgets all blocks of strands except for a block of size $n$ in $\lambda_S$. The induced split short exact sequence of fundamental groups is then
\[
\begin{tikzcd}
1 \ar[r] & K=\pi_{1}(X) \ar[r,"\varphi"] & G = \wB(\lambda_P,\lambda_{S_+},\lambda_S) \ar[r] & \Gamma = \exwB_{n} \ar[l,bend right=25,dashed] \ar[r] & 1.
\end{tikzcd}
\]

We first consider the \uline{$5$-th row of the $\exwB_{n}$ block of Table~\ref{table:pro-nilpotent-representations}}, where we assume that $\lambda_S$ contains a block of size $n$ and a block of size $1$. Under this assumption, there is a quotient map $G \twoheadrightarrow \wB(\varnothing,\varnothing,\{n,1\})$ given by forgetting all blocks of strands except for these two. A generating set for the group $\wB(\varnothing,\varnothing,\{n,1\})$ consists of the five families of elements described above for the group $\wB(\varnothing,n,1)$, together with an additional family of $n$ elements $\rho_\alpha$ for $2\leq \alpha\leq n+1$, where $\rho_\alpha$ rotates the $\alpha$-th circle by $180$ degrees about an axis lying the plane of the circle. Similarly to the case of $\wB(\varnothing,n,1)$, the proof of \cite[Prop.~4.52]{DPS} constructs a quotient map
\[
\wB(\varnothing,\varnothing,\{n,1\}) \longtwoheadrightarrow \bZ \rtimes (\bZ/2)
\]
that sends $\chi_{\alpha 1} \mapsto (1,0)$ for each $\alpha$, $\rho_{1} \mapsto (0,1)$ and all other generators to $(0,0)$. Composing these two quotient maps we obtain a surjection $G \twoheadrightarrow G' := \bZ \rtimes (\bZ/2)$. The lower central series of $\bZ \rtimes (\bZ/2)$ does not stop by \cite[Cor.~A.8]{DPS}. The generators $(1,0)$ and $(0,1)$ of $\bZ \rtimes (\bZ/2)$ lift to $\chi_{\alpha 1}$ and $\rho_{1}$ respectively, which lie in $K$ since their projections to $\exwB_{n}$ are trivial, so the restriction of $G \twoheadrightarrow \bZ \rtimes (\bZ/2)$ to $K$ is surjective. Its restriction to $\Gamma = \exwB_{n}$ via the section is trivial since the image of $\exwB_{n}$ under the section is generated by the elements $\sigma_\alpha$, $\tau_\alpha$ and $\rho_\alpha$, which are all sent to $(0,0)$ in $\bZ \rtimes (\bZ/2)$.

Finally, we consider simultaneously the \uline{first four rows of the $\exwB_{n}$ block of Table~\ref{table:pro-nilpotent-representations}}. The proof in this case is very similar to the proof in the case of the first three rows of the $\wB_{n}$ block of Table~\ref{table:pro-nilpotent-representations}. There is a quotient map $G \twoheadrightarrow \wB(\lambda_P,\lambda_{S_+},\mu_S)$ given by forgetting a block of size $n$ in $\lambda_S$ (remember that we have $\lambda_S = (n,\mu_S)$ by assumption). Its restriction to $K \subset G$ is surjective since any loop braid in $\wB(\lambda_P,\lambda_{S_+},\mu_S)$ may be lifted to a loop braid in $G = \wB(\lambda_P,\lambda_{S_+},\lambda_S)$ by adding a stationary block of $n$ unoriented circles near the boundary of $\bD^3$ and this lift lies in $K \subset G$ since it projects to the trivial element of $\exwB_{n}$. Its restriction to $\Gamma$ via the section is trivial since $\Gamma \subset G$ consists of braids that are trivial except on the block of circles that is forgotten under the projection to $\wB(\lambda_P,\lambda_{S_+},\mu_S)$. We therefore only need to verify that the lower central series of $\wB(\lambda_P,\lambda_{S_+},\mu_S)$ does not stop. By assumption, the tuple of partitions $(\lambda_P,\lambda_{S_+},\mu_S)$ contains either $(\varnothing,b,\varnothing)$, $(\varnothing,\varnothing,b)$, $(\varnothing,\{1,1\},\varnothing)$ or $(2,i,\varnothing)$ with $i\geq 1$ as a tuple of sub-partitions, where $b \in \{2,3\}$. Then by \cite[Th.~4.47]{DPS} in the first three cases and by \cite[Prop.~5.18]{DPS} in the fourth case the lower central series of $\wB(\lambda_P,\mu_{S_+},\lambda_S)$ does not stop.
\end{proof}

To finish this section, we describe, for each row of Table~\ref{table:pro-nilpotent-representations} concerning $\wB_{n}$ or $\exwB_{n}$, the ground ring of the bottom ($r=2$) layer of the weakly pro-nilpotent representation that we have constructed. In other words, we calculate the abelian group $A = Q_{2}$, since the ground ring of the bottom layer is $\bZ[Q_{2}]$. By construction (see diagram \eqref{eq:big-diagram-of-quotients}), this is the kernel of the (split) surjection $G^{\ab} \twoheadrightarrow \Gamma^{\ab}$ induced by the given (split) surjection $G \twoheadrightarrow \Gamma$.

\begin{prop}
\label{prop:abelian-quotients-3}
For $n\geq 2$, in the $\wB_{n}$ and $\exwB_{n}$ blocks of Table~\ref{table:pro-nilpotent-representations}, the group $A$ is isomorphic to
\[
\begin{cases}
\bZ^{N-1} \times (\bZ/2)^{M-1} & \text{for } \Gamma = \wB_{n} , \\
\bZ^{N-1} \times (\bZ/2)^{M-2} & \text{for } \Gamma = \exwB_{n} ,
\end{cases}
\]
where for $\star \in \{P,S_+,S\}$ we write $l_\star$ for the number of blocks of the partition $\lambda_\star$ and $l'_\star$ for the number of blocks of the partition $\lambda_\star$ of size at least $2$, we set $l = l_P + l_{S_+} + l_S$ and $l' = l'_P + l'_{S_+} + l'_S$ and we define
\[
N = l'_{S_+} + l_{S_+}(l-1) \qquad\text{and}\qquad M = l' + l'_S + l_S l.
\]
\end{prop}
\begin{proof}
By \cite[Prop.~5.10]{DPS}, we have $G^{\ab} = \wB(\lambda_P,\lambda_{S_+},\lambda_S)^{\ab} \cong \bZ^N \times (\bZ/2)^M$. The result then follows since $A$ is the kernel of the split surjection $G^{\ab} \twoheadrightarrow \Gamma^{\ab}$ and we have $\Gamma^{\ab} \cong \bZ \times \bZ/2$ for $\Gamma = \wB_{n}$ and $\Gamma^{\ab} \cong \bZ \times (\bZ/2)^2$ for $\Gamma = \exwB_{n}$.
\end{proof}

\phantomsection
\addcontentsline{toc}{section}{References}
\renewcommand{\bibfont}{\normalfont\small}
\setlength{\bibitemsep}{0pt}
\printbibliography

\end{document}